\documentclass[12pt,a4paper]{article}

\usepackage{amssymb}
\usepackage{amsmath, amsthm}
\usepackage{epsfig}

\usepackage{geometry}
\def\RR{{\bf R}}
\def\ZZ{{\bf Z}}

\def\lgate{{\langle\langle}}
\def\rgate{{\rangle\rangle}}

\def\bigmid {\ \left|{\Large \strut}\right.}

\numberwithin{equation}{section}
\newtheorem{Thm}{Theorem}[section]
\newtheorem{Prop}[Thm]{Proposition}
\newtheorem{Lem}[Thm]{Lemma}

\theoremstyle{definition}

\newtheorem{Rem}[Thm]{Remark}
\newtheorem{Ex}[Thm]{Example}

\newcommand{\Conv}{\mathop{\rm Conv} }

\newcommand{\dom}{\mathop{\rm dom} }

\title{L-convexity on graph structures}
\author{Hiroshi HIRAI \\
Department of Mathematical Informatics, \\
Graduate School of Information Science and Technology,   \\
University of Tokyo, Tokyo, 113-8656, Japan.\\
\texttt{\normalsize hirai@mist.i.u-tokyo.ac.jp}}

\begin{document}

\maketitle

\begin{abstract}
In this paper, 
we study classes of discrete convex functions: 
submodular functions on modular semilattices and 
L-convex functions on oriented modular graphs. 
They were introduced by the author in complexity classification of minimum 0-extension problems.
We clarify the relationship to other discrete convex functions, such as 
$k$-submodular functions, skew-bisubmodular functions, L$^\natural$-convex functions, 
tree submodular functions, and UJ-convex functions.
We show that they are actually viewed
as submodular/L-convex functions in our sense.
We also prove a sharp iteration bound of the steepest descent algorithm 
for minimizing our L-convex functions.
The underlying structures, modular semilattices and oriented modular graphs, 
have rich connections to projective and polar spaces, Euclidean building, 
and metric spaces of global nonpositive curvature (CAT(0) spaces).
By utilizing these connections, 
we formulate an analogue of the Lov\'asz extension, 
introduce well-behaved subclasses of submodular/L-convex functions, 
and show that these classes
can be characterized by the convexity of the Lov\'asz extension.
We demonstrate applications of our theory to combinatorial optimization problems that include 
multicommodity flow, multiway cut, and related labeling problems: 
these problems have been outside the scope of discrete convex analysis so far.  
\end{abstract}


\section{Introduction} 
{\em Discrete Convex Analysis (DCA)} is 
a theory of ``convex" functions on integer lattice $\ZZ^n$, 
developed by K. Murota~\cite{Murota98,MurotaBook} 
and his collaborators (including S. Fujishige, A. Shioura, and A. Tamura), 
and aims to provide 
a unified theoretical framework to 
well-solvable combinatorial optimization problems related to network flow and matroid/submodular optimization; see also \cite[Chapter VII]{FujiBook}.
In DCA, two classes of discrete convex functions, 
{\em M-convex functions} and {\em L-convex functions},  
play primary roles; the former originates from the base exchange property of matroids,
the latter abstracts potential energy in electric circuits,  
and they are in relation of conjugacy.
They generalize several known concepts in matroid/submodular optimization 
by means of analogy of convex analysis in continuous optimization.
Besides its fundamental position
in combinatorial optimization and operations research, 
the scope of DCA has been enlarging over past 20 years, and 
several DCA ideas and concepts have been successfully 
and unexpectedly applied to other areas of applied mathematics that include 
mathematical economics, game theory~\cite{TamuraBook}, inventory theory~\cite{Logistics} and so on;
see recent survey~\cite{MurotaDevelop}.

The present article addresses a new emerging direction of DCA, which might be called 
a theory of discrete convex functions on graph structures, 
or {\em DCA beyond~$\ZZ^n$}.
Our central subjects are graph-theoretic generalizations of 
L$^\natural$-convex function.
An {\em L$^\natural$-convex function}~\cite{FM00} is an essentially equivalent variant of L-convex function, 
and is defined as a function $g$ on $\ZZ^n$ 
satisfying a discrete version of the convexity inequality, 
called the {\em discrete midpoint convexity}: 
\begin{equation}\label{eqn:midpoint_original}
g(x) + g(y) \geq g(\lfloor (x + y)/2\rfloor) + g(\lceil (x + y)/2\rceil)
\quad (x,y \in \ZZ^n),
\end{equation}
where $\lfloor \cdot \rfloor$ (resp. $\lceil \cdot \rceil$) 
denotes an operation on $\RR^n$ 
that rounds down (resp. up) the fractional part of each component.
L$^\natural$-convex functions have 
several fascinating properties for 
optimization and algorithm design; 
see \cite[Chapters 7 and 10]{MurotaBook}.
They are submodular functions on each cube $x+ \{0,1\}^n$ $(x \in \ZZ^n)$,
and are extended to convex functions on $\RR^n$ via the Lov\'asz extension.
The global optimality is guaranteed by the local optimality 
({\em L-optimality criterion}), which is checked by submodular function minimization (SFM).
This fact brings about a simple descent algorithm, 
called the {\em steepest descent algorithm (SDA)}, to minimize L$^\natural$-convex functions through successive application of an SFM algorithm.
The number of iterations (= calls of the SFM algorithm) of SDA is estimated by the $l_\infty$-diameter of domain~\cite{KS09,MurotaShioura14}.

Observe that the discrete midpoint convexity (\ref{eqn:midpoint_original}) 
is still definable if $\ZZ^n$ is replaced by the Cartesian product $P^n$ of 
$n$ directed paths $P$, and L$^\natural$-convex functions are also definable on $P^n$.
Starting from this observation, 
Kolmogorov~\cite{Kolmogorov11} considered a generalization of L$^\natural$-convex functions defined on 
the product of rooted binary trees, and discussed an SDA framework, 
where  SFM is replaced by bisubmodular function minimization.
One may ask:
	{\em Can we define
		analogues of L$^\natural$-convex functions on more general graph structures,
		and develop a similar algorithmic framework to  
		attack combinatorial optimization problems beyond the current scope of DCA?}

This question was answered affirmatively in the study~\cite{HH150ext} 
of {\em the minimum 0-extension problem} (alias {\em multifacility location problem})~\cite{Kar98a}---the problem of 
finding locations $x_1,x_2,\ldots,x_n$ of $n$ facilities
in a graph $\varGamma$ such that 
the weighed sum 
\[
\sum_{v} \sum_{i} b_{vi} d_{\varGamma}(v,x_i) + \sum_{i,j} c_{ij} d_{\varGamma}(x_i,x_j)
\]
 of their distances 
is minimum. 
This problem is viewed as 
an optimization over the
graph $\varGamma \times \varGamma \times \cdots \times \varGamma$, and motivates us to consider ``convex" functions on graphs.
The 0-extension problem was known to be NP-hard 
unless $\varGamma$ is an {\em orientable modular graph}, 
and known to be polynomial time solvable for special subclasses of orientable modular graphs~\cite{Chepoi96,Kar98a,Kar04a}. 
In~\cite{HH150ext}, we revealed several intriguing structural properties of orientable modular graphs: 
they are amalgamations of modular lattices 
and {\em modular semilattices} in the sense of \cite{BVV93}.
On the basis of the structure, 
we introduced two new classes of discrete convex functions: 
{\em submodular functions} on modular semilattices
and {\em L-convex functions} 
on oriented modular graphs (called {\em modular complexes} in \cite{HH150ext}).
Here an {\em oriented modular graph} is 
an orientable modular graph endowed with a specific edge-orientation.
We established analogues of local submodularity,  
L-optimality criterion and SDA 
for the new L-convex functions, 
and proved the VCSP-tractability~\cite{KTZ13,TZ12FOCS,ZivnyBook} of the new submodular functions.
Finally we showed that the 0-extension problem on an orientable modular graph 
can be formulated as an L-convex function minimization
on an oriented modular graph, and is solvable in polynomial time by SDA.
This completes the complexity classification of the 0-extension problem.
In the subsequent work~\cite{HH14extendable,HH15node_multi}, 
by developing analogous theories of L-convex functions on certain graph structures,
we succeeded in designing efficient combinatorial polynomial algorithms for some classes of multicommodity flow problems for which such algorithms had not been known.

Although these discrete convex functions on graph structures 
brought algorithmic developments,
their relations and connections to other discrete convex functions
(in particular, original L$^\natural$-convex functions) were not made clear.
In this paper, we continue to study submodular functions on 
modular semilattices and L-convex functions on oriented modular graphs. 
The main objectives are (i) to clarify
their relations to other classes of discrete convex functions,  
(ii) to introduce several new concepts and pursue further L-convexity properties,  
and (iii) to present their occurrences and applications in 
actual combinatorial optimization problems.

A larger part of our investigation 
is built on {\em Metric Graph Theory (MGT)}~\cite{BandeltChepoi}, 
which studies graph classes from distance properties
and provides a suitable language 
for analysis if orientable modular graphs and, more generally, {\em weakly modular graphs}.
Recently MGT~\cite{CCHO14, Chepoi00} explored rich connections among 
weakly modular graphs, 
incidence geometries~\cite{Foundation}, 
Euclidean building~\cite{BuildingBook,TitsBuilding}, 
and metric spaces of nonpositive curvature~\cite{BrHa}.
To formulate and prove our results, 
we will utilize some of the results and concepts from 
these fields of mathematics, 
previously unrelated to combinatorial optimization.
We hope that this
will provide fruitful interactions among these fields, 
explore new perspective, and trigger further evolution of DCA.

The rest of the paper is organized as follows.
In Section~\ref{sec:preliminaries}, we introduce
basic notation, and summarize some of the concepts 
and results from MGT~\cite{CCHO14} that we will use.
We introduce modular semilattices, {\em polar spaces}, orientable modular graphs, {\em sweakly modular graphs} ({\em swm-graphs})~\cite{CCHO14}, and {\em Euclidean building of type C}.
They are underlying structures of discrete convex functions considered in the subsequent sections.
In particular, polar spaces turn out to be appropriate generalizations of underlying structures of bisubmodular and {\em $k$-submodular functions}~\cite{HK12}. 
A Euclidean building is a generalization of tree, and is a simplicial complex having 
an amalgamated structure of Euclidean spaces, which
naturally admits an analogue of discrete midpoint operators.
We introduce continuous spaces, {\em CAT(0) metric spaces}~\cite{BrHa} and {\em orthoscheme complexes}~\cite{BM10}.
They play roles of continuous relaxations of orientable modular graphs, 
analogous to continuous relaxation $\RR^n$ of $\ZZ^n$, 
and enable us to formulate 
the Lov\'asz extension and its convexity for our discrete convex functions.

In Section~\ref{sec:submodular}, 
we study submodular functions on modular semilattices. 
After reviewing their definition and basic properties,  
we focus on submodular functions on polar spaces, 
and establish their relationship 
to $k$-submodular functions 
and {\em skew-bisubmodular functions}~\cite{HKP14}, and 
also show that submodular functions on polar spaces are characterized by the convexity of the Lov\'asz extension.

In Section~\ref{sec:L-convex}, 
we study L-convex functions on oriented modular graphs.
We give definition and basics established by \cite{HH150ext}.
We then present a sharp $l_{\infty}$-iteration bound of the steepest descent algorithm.
We introduce notions of {\em L-extendability} and {\em L-convex relaxations}, 
extending those considered by~\cite{HH14extendable} for the product of trees.
We study L-convex functions on 
a Euclidean building (of type C). 
After establishing a characterization by the discrete midpoint convexity  and the convexity of the Lov\'asz extension, 
we explain how our framework captures original L$^\natural$-convex functions, {\em UJ-convex functions}~\cite{Fujishige14}, 
{\em strongly-tree submodular functions}~\cite{Kolmogorov11}, 
and {\em alternating L-convex functions}~\cite{HH14extendable}.  

In Section~\ref{sec:applications}, 
we present applications of our theory to combinatorial optimization problems including multicommodity flow, multiway cut, and related labeling problems. We see that our L-convex/submodular functions 
arise as the dual objective functions of several multiflow problems 
or half-integral relaxations of multiway cut problems.  

The beginning of Sections~\ref{sec:submodular}, \ref{sec:L-convex}, and \ref{sec:applications} includes a more detailed summary of results.
The proofs of all results in Sections~\ref{sec:submodular} and \ref{sec:L-convex} 
are given in Section~\ref{sec:proofs}.
Some of the results which we present in this paper
were announced in \cite{HH13HJ, HH14extendable, HH15node_multi, HH150ext}.

\section{Preliminaries}\label{sec:preliminaries}

\subsection{Basic notation}\label{subsec:notation}
Let $\RR$, $\RR_+$, $\ZZ$, and $\ZZ_+$ 
denote the sets of reals, nonnegative reals,  integers, and nonnegative integers, respectively. 
The $i$th unit vector in $\RR^n$ is denoted by $e_i$.
Let $\infty$ denote the infinity element treated 
as $a + \infty = \infty$, $a < \infty$ for $a \in \RR$, and $\infty + \infty = \infty$.
Let $\overline{\RR} := \RR \cup \{\infty\}$.
For a function $f: X \to \overline\RR$, 
let $\dom f$ denote the set of elements $x$ with $f(x) < \infty$.

By a graph $G$ we mean a connected simple undirected graph. 
An edge joining vertices $x$ and $y$ is denoted by $xy$.
We do not assume that $G$ is a finite graph.
If every vertex of $G$ has a finite degree, then $G$ is called {\em locally-finite}.
For notational simplicity,
the vertex set $V(G)$ of $G$ is also denoted by $G$.
A path or cycle is written by a sequence $(x_1,x_2,\ldots,x_k)$ of vertices.
If $G$ has an edge-orientation, then $G$ is called an {\em oriented} graph, 
whereas paths, cycles, or distances are considered in the underlying graph. 
We write $x \to y$ if edge $xy$ is oriented from $x$ to $y$.
The {\em (Cartesian) product} of two graphs $G,G'$
is the graph on $V(G) \times V(G')$ with an edge given between $(x,x')$ and $(y,y')$ if 
$x = y$ and $x'y'$ is an edge of $G'$, or $x' = y'$ and $xy$ is an edge of $G$.
The product of $G,G'$ is denoted by $G \times G'$. 
In the case where both $G$ and $G'$ have edge-orientations,
the edge-orientation of $G \times G'$ is defined by
$(x,x') \to (y,y')$ if $x \to y$ and $x' = y'$, or 
$x' \to y'$ and $x = y$.

We will use the standard terminology of posets (partially ordered sets), 
which basically follows~\cite[Section 2.1.3]{CCHO14}.
The partial order of a poset ${\cal P}$ is denoted by $\preceq$, 
where $p \prec q$ means $p \preceq q$ and $p \neq q$.
The join and meet of elements $p,q$, if they exist, 
are denoted by $p \vee q$ and $p \wedge q$, respectively.
The rank function of a meet-semilattice is denoted by $r$.
An element of rank 1 is called an {\em atom}.
Lattices and semilattices are supposed to have finite rank.
The minimum element of a meet-semilattice is denoted by $0$.
For $p \preceq q$,
define the {\em interval} $[p,q]$ by 
$[p,q]:= \{u \in {\cal P} \mid p \preceq u \preceq q \}$.
The {\em principal filter} ${\cal F}_p$ of $p \in {\cal P}$ 
is the set of elements $u$ with $p \preceq u$, and is regarded as a poset by the restriction of $\preceq$.
The {\em principal ideal} ${\cal I}_p$ of $p$ is the set of elements $u$ with $u \preceq p$, and  regarded as a poset by the reverse of the restriction of $\preceq$ (so that $p$ is the minimum element).
We say that $p$ {\em covers} $q$ if 
$p \succ q$ and there is no $u$ with $p \succ u \succ q$.
In this case, we write $p \rightarrow q$.
The {\em covering graph} of poset  ${\cal P}$ is a graph on ${\cal P}$ 
such that elements $p,q$ are adjacent if and only if $p$ covers $q$ or $q$ covers $p$.
The covering graph is naturally oriented as $p \to q$ if $p$ covers $q$.

\subsection{Modular semilattices and polar spaces}\label{subsec:lattice}
In this section, we introduce modular semilattices and polar spaces, which are underlying structures of submodular functions in Section~\ref{sec:submodular}.
A lattice ${\cal L}$ is called {\em modular} 
if for every $x,y,z \in {\cal L}$ with $x \succeq z$ it holds $x \wedge (y \vee z) = (x \wedge y) \vee z$.
A semilattice ${\cal L}$ is called {\em modular}~\cite{BVV93}  if every principal ideal is a modular lattice, 
and for every $x,y,z \in {\cal L}$ 
the join $x \vee y \vee z$ exists provided 
$x \vee y$, $y \vee z$, and $z \vee x$ exist.
A {\em complemented} modular (semi)lattice is a modular (semi)lattice 
such that every element is the join of atoms.
\begin{Ex}[${{\cal S}_k}^n$]\label{ex:S_k}
		For nonnegative integer $k \geq 0$, 
		let ${\cal S}_{k}$ denote the $k+1$ element set containing a special element $0$.
		The partial order on ${\cal S}_k$ is given by 
		$0 \leftarrow u$ for $u \in {\cal S}_{k} \setminus \{0\}$; other pairs are incomparable. 
		Clearly, ${\cal S}_k$ is a complemented modular semilattice, 
		and so is the product ${{\cal S}_k}^n$.
	   Notice that ${{\cal S}_2}^n = \{-1,0,1\}^n$ 
		is the domain of bisubmodular functions.
		More generally,   
		${{\cal S}_k}^n$ is the domain of $k$-submodular functions.
\end{Ex}

\begin{Ex}[totally isotropic subspaces]
	Let $V$ be a finite dimensional vector space. 
	It is well-known that the poset of all subspaces (ordered by inclusion) 
	is a complemented modular lattice. 
	Suppose further that $V$ has a bilinear form $B: V \times V \to \RR$.
	A {\em totally isotropic subspace} is a vector subspace $X$ of $V$ 
	such that $B(x,y) = 0$ for $x,y \in X$.
	Then the poset ${\cal L}$ of all totally isotropic subspaces of $V$ 
	(ordered by inclusion) is a complemented modular semilattice.
	Indeed, for a totally isotropic subspace $X$, any subspace of $X$ is again totally isotropic.
	This implies that ${\cal L}$ is a meet-semilattice with $\wedge = \cap$, and
	every principal ideal is a complemented modular lattice.
	In addition, the join of $X,Y$, if it exists, is equal to $X+Y$.
	Suppose that $X$, $Y$, and $Z$ are totally isotropic subspaces 
	such that $X+Y$, $Y+Z$, and $X+Z$ are totally isotropic.
	Then $X+Y+Z$ is totally isotropic, 
	Indeed, for $x_1,x_1' \in X$, $x_2,x_2' \in Y$, and $x_3,x_3' \in Z$, 
	$B(x_1+x_2+x_3,x_1'+ x_2'+x_3') = \sum_{i,j} B(x_i,x_j') = 0$. 	
\end{Ex}
A canonical example of a modular semilattice is a polar space.
A {\em polar space} ${\cal L}$ of rank $n$ 
is a semilattice that is the union of subsemilattices, 
called {\em polar frames}, satisfying the following axiom.\footnote{Our definition of a polar space is not standard; see \cite{Foundation} 
for the definition of a polar space as an incidence geometry.
It is known that polar spaces and spherical buildings of type C
constitute the same object~\cite{TitsBuilding}.
Our conditions reformulate the axiom that 
the order complex of ${\cal L} \setminus \{0\}$
is a spherical building of type C, 
where the order complex of ${{\cal S}_2}^n \setminus \{0\}$ 
is a Coxeter complex of type C, and
apartments are subcomplexes corresponding to polar frames.}
	\begin{description}
		\item[P0:] Each polar frame is isomorphic to ${{\cal S}_2}^n = \{-1,0,1\}^n$.
		\item[P1:] 
		For two chains $C,D$ in ${\cal L}$, 
		there is a polar frame ${\cal F}$ containing them.
		\item[P2:] If polar frames ${\cal F},{\cal F}'$ 
		both contain two chains $C,D$, then there is an isomorphism ${\cal F} \to {\cal F}'$ being identity on $C$ and $D$.
	\end{description}
\begin{Prop}[{\cite[Lemma 2.19]{CCHO14}}]
	A polar space is a complemented modular semilattice.
\end{Prop}
In Section~\ref{sec:submodular}, 
we define submodular functions on polar spaces, 
which turn out to be appropriate generalizations of bisubmodular and $k$-submodular functions.
\begin{Ex}\label{ex:S_kl}
	For nonnegative integer $k \geq 2$, 
	${\cal S}_{k}$ is a polar space of rank $1$, 
	where polar frames are $\{0, a, b\} \simeq {\cal S}_2$ for distinct $a,b \in {\cal S}_k \setminus \{0\}$.
	The product ${{\cal S}_{k}}^n$ is also a polar space (of rank $n$).
	We give another example.
	For nonnegative integers $k,l \geq 2$, 
	let  ${\cal S}_{k,l}$ denote the poset on ${\cal S}_{k} \times {\cal S}_{l}$
	with partial order: $(a,0) \to (a,b) \leftarrow (0,b)$ 
	and $(a,b) \to (0,0)$ for $a \in {\cal S}_k \setminus \{0\}$ 
	and $b \in {\cal S}_{l} \setminus \{0\}$.
	Notice that this partial order is different from the one in the direct product.
	Then ${\cal S}_{k,l}$  is a polar space of rank $2$, 
	where polar frames take the form of $\{ 0, a,a'\} \times \{0, b,b'\}$ 
	with $a \neq a'$ and $b \neq b'$. 
	%
\end{Ex}

\subsection{Orientable modular graphs and swm-graphs}\label{subsec:modular}
In this section,  we introduce orientable modular graphs and swm-graphs, 
which are the underlying structures
of L-convex and L-extendable functions in Section~\ref{sec:L-convex}.
For a graph $G$, let $d = d_{G}$ 
denote the shortest path metric 
with respect to a specified uniform positive edge-length of $G$.
The {\em metric interval} $I(x,y)$ of vertices $x,y$ is 
the set of points $z$ satisfying $d(x,z) + d(z,y) = d(x,y)$.
We regard $G$ as a metric space with this metric $d$.
A nonempty subset $Y$ of vertices is said to be {\em $d$-convex}
if $I(x,y) \subseteq Y$ for every $x,y \in Y$

\subsubsection{Modular graphs and orientable modular graphs}
A {\em modular graph} is a graph such that 
for every triple $x,y,z$ of vertices, 
the intersection $I(x,y) \cap I(y,z) \cap I(z,x)$ is nonempty.
A canonical example of a modular graph is the covering graph of a modular lattice.
More generally, the following is known.
\begin{Thm}[{\cite[Theorem 5.4]{BVV93}}]
	A semilattice is modular if and only if its covering graph is modular.
\end{Thm}
An edge-orientation of a modular graph is called {\em admissible}~\cite{Kar98a} 
if for every 4-cycle $(x_1, x_2, x_3, x_4)$,
$x_1 \to  x_2$ implies $x_4 \to x_3$. 
A modular graph may or may not have an admissible orientation.
A modular graph is said to be {\em orientable}~\cite{Kar98a} if it has
an admissible orientation.
Observe that the covering graph of a modular semilattice is 
orientable with respect to the covering relation. 
A modular graph endowed with an admissible orientation is called 
an {\em oriented modular graph}.
Any admissible orientation is acyclic~\cite{HH150ext}, and induces a partial order~$\preceq$ on the set of vertices.
In this way, an oriented modular graph is also viewed as a poset.
Define a binary relation $\sqsubseteq$ by: 
$p \sqsubseteq q$ if $p \preceq q$ and $[p,q]$ 
is a complemented modular lattice.
Notice that $\sqsubseteq$ is not a partial order 
but $p \sqsubseteq q$ and $p \preceq p' \preceq q' \preceq q$ 
imply $p' \sqsubseteq q'$ 
(since any interval of a complemented modular lattice 
is a complemented modular lattice).
For a vertex $p$, the {\em principal $\sqsubseteq$-ideal} 
${\cal I}'_p$ (resp. {\em  $\sqsubseteq$-filter} ${\cal F}'_p$) 
are the set of vertices $q$ with  $p \sqsupseteq q$
(resp.  $p \sqsubseteq q$).
In particular, 
${\cal I}'_p \subseteq {\cal I}_p$ and ${\cal F}'_p \subseteq {\cal F}_p$ hold.
\begin{Prop}[{\cite[Proposition 4.1, Theorem 4.2]{HH150ext}}]\label{prop:ideal}
	Let $\varGamma$ be an oriented modular graph.
	\begin{description}
		\item[{\rm (1)}] Every principal ideal (filter) is a modular semilattice and $d$-convex. 
		In particular, every interval is a modular lattice.	
		\item[{\rm (2)}] Every principal $\sqsubseteq$-ideal (filter) is a complemented modular semilattice and $d$-convex.
	\end{description}
\end{Prop}
An oriented modular graph is said to be {\em well-oriented}
if $\preceq$ and $\sqsubseteq$ are the same, i.e., 
every interval is a complemented modular lattice.
If $\varGamma$ is well-oriented, 
then ${\cal I}'_p = {\cal I}_p$ and ${\cal F}'_p = {\cal F}_p$.
\begin{Ex}[linear and alternating orientation of grid]
		The set $\ZZ$ of integers is regarded as a graph (path) 
		by adding an edge to each pair $x,y \in \ZZ$ with $|x - y| = 1$.
		Any orientation of $\ZZ$ is admissible. 
		There are two canonical orientations.
		The {\em linear orientation} 
		is the edge-orientation such that $x \leftarrow y$ if $y = x + 1$.
		Let $\vec{\ZZ}$ denote the oriented modular graph  
		$\ZZ$ endowed with the linear orientation.
		The {\em  alternating orientation} of $\ZZ$ is the edge-orientation  
		such that $x \to y$ if $x$ is even. 
		Let $\check{\ZZ}$ denote the oriented modular graph $\ZZ$ endowed with the alternating orientation.
		Observe that $\vec\ZZ$ is not well-oriented and $\check{\ZZ}$ is well-oriented.  
		The products $\vec\ZZ^n$ and $\check{\ZZ}^{n}$ 
		are oriented modular graphs, where ${\check{\ZZ}}^n$ is well-oriented.
		In $\vec{\ZZ}^n$, 
		each principal filter is isomorphic to $\ZZ_+^n$ (with respect to the vector order) and each principal $\sqsubseteq$-filter is isomorphic to  
		${{\cal S}_1}^n = \{0,1\}^n$.
		In $\check{\ZZ}^{n}$, 
		the principal filter of vector $p$ is isomorphic to 
		${{\cal S}_{2}}^k = \{-1,0,1\}^k$, where $k$ is 
		the number of odd components of $p$.
\end{Ex}
\subsubsection{Weakly modular graphs and swm-graphs}
A {\em weakly modular graph} is a connected graph satisfying 
the following conditions:  
\begin{description}
	\item[TC:] For every triple of vertices $x,y,z$ with 
	$d(x,y) = 1$ and $d(x,z) = d(y,z)$, 
	there exists a common neighbor $u$ of $x,y$ with $d(u,z) = d(x,z)-1$.
	\item[QC:] For every quadruple of vertices $x,y,w,z$ with $d(x,y) = 2$,
	$d(w,y) = d(w,x) = 1$ and $d(w,z) - 1 = d(x,z) = d(y,z)$, there 
	exists a common neighbor $u$ of $x,y$ with $d(u,z) = d(x,z)-1$.
\end{description}
Modular graphs are precisely 
bipartite weakly modular graphs.

In a graph $G$,
a nonempty subset $X$ of vertices 
is said to be {\em $d$-gated} (or {\em gated})
if for every vertex $x$ there is (unique) $y \in X$
such that $d(x,z) = d(x,y) + d(y,z)$ for every $z \in X$.
A $d$-gated set is always $d$-convex.
Instead of this definition, 
the following simple characterization of gated sets may be regarded as the definition.
\begin{Lem}[{Chepoi's lemma; see \cite[Lemma 2.2]{CCHO14}}]\label{lem:chepoi}
	For a weakly modular graph $G$, 
	a nonempty subset $X$ is $d$-gated if and only if
	$X$ induces a connected subgraph and 
	for $x,y \in X$ every common neighbor of $x,y$ belongs to $X$.
	If $G$ is modular, then $d$-gated sets and $d$-convex sets are the same.
\end{Lem}

A {\em sweakly modular graph} ({\em swm-graph} for short)~\cite{CCHO14} 
is a weakly modular graph having no
induced $K_4^-$-subgraph and isometric $K_{3,3}^-$-subgraph, 
where $K_4^-$ and $K_{3,3}^-$ are graphs obtained from $K_4$ and $K_{3,3}$, respectively, by removing one edge. 
Since $K_{3,3}^-$ is not orientable, 
an orientable modular graph cannot have induced $K_{3,3}^-$.
Therefore, any orientable modular graph is an swm-graph. 
A canonical example of an swm-graph is a {\em dual polar space}~\cite{Cameron82}, 
which is a graph such that vertices are maximal elements of a polar space 
and two vertices $p,q$ are adjacent 
if $p \rightarrow p \wedge q \leftarrow q$.
\begin{Thm}[{\cite[Theorem 5.2]{CCHO14}}]\label{thm:dual_polar}
	A graph is a dual polar space if and only if it is a thick swm-graph of finite diameter. 
\end{Thm}
Here a weakly modular graph is said to be {\em thick} 
if for every pair of vertices $x,y$ with $d(x,y) = 2$ 
there are two common neighbors $z,w$ of $x,y$ with $d(z,w) = 2$.

For an swm-graph $G$, a nonempty set $X$ of 
vertices is called a {\em Boolean-gated} 
if $X$ is gated and induces a connected thick subgraph.
In a weakly modular graph 
the subgraph induced by a gated set is weakly modular.
Thus, by Theorem~\ref{thm:dual_polar}, 
Boolean-gated sets are precisely
gated sets inducing dual polar spaces.
Also the nonempty intersection of Boolean-gated sets is again Boolean-gated.
In this way, an swm-graph is viewed as an amalgamation of dual polar graphs.
The set of all Boolean-gated sets of $G$ is denoted by ${\cal B}(G)$.
The partial order on ${\cal B}(G)$ is defined as the reverse inclusion order.
Let $G^*$ denote the covering graph of ${\cal B}(G)$, called the {\em barycentric subdivision} of~$G$.
The edge-length of $G^*$ is defined as one half of the edge-length of $G$.

The barycentric subdivision $G^*$ serves as  
a {\em half-integral relaxation} of (problems on) $G$, and 
is used to define the notion of L-extendability in Section~\ref{sec:L-convex}.
	\begin{figure}[t]
		\begin{center}
			\includegraphics[scale=0.7]{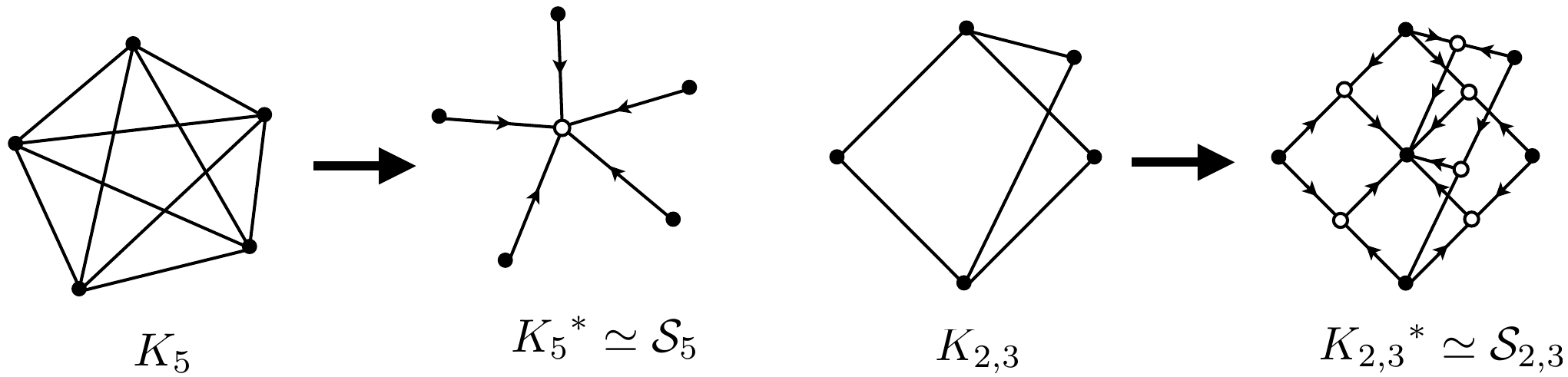}
			\caption{Barycentric subdivision}
			\label{fig:subdivision}
		\end{center}
	\end{figure}\noindent

	\begin{Ex}\label{ex:grid}
		Consider the path $\ZZ$.
		The barycentric subdivision $\ZZ^*$ 
		is naturally identified with the set $\ZZ/2$ of half-integers with edge $xy$
		given if $|x - y| = 1/2$ and oriented $x \leftarrow y$ if $y \in \ZZ$. 
		The product $\ZZ^n$ is the $n$-dimensional grid graph.
		A Boolean-gated set is exactly a vertex subset inducing a cube-subgraph, 
		which is equal to $\{z \in \ZZ^n \mid x_i \leq z_i \leq y_i, i=1,2,\ldots,n\}$ 
		for some $x,y \in \ZZ^n$ with $|x_i - y_i| \leq 1$ $(i=1,2,\ldots,n)$.
		The barycentric subdivision $(\ZZ^*)^n$ is isomorphic to the grid graph on 
		the half-integral lattice $(\ZZ/2)^n$ with the alternating orientation. 
		Note that $\check{\ZZ}^n$ is isomorphic to 
		the half-integer grid $(\ZZ^*)^n$ by $x \mapsto x/2$.
	\end{Ex}
	\begin{Ex}\label{ex:subdivision}
		Trees, cubes, complete graphs, and complete bipartite graphs are all swm-graphs. 
		The last three are dual polar spaces.
		In the case of a tree, Boolean-gated sets are 
		all the singleton, and 
		all the edges (pairs of vertices inducing edges).
		Hence the barycentric subdivision is 
		just the edge-subdivision, 
		where the original vertices are sources (i.e., have no entering edges). 
		Consider the case of a cube.
		Boolean-gated sets are vertex subsets inducing cube-subgraphs.
		Therefore the barycentric subdivision is 
		the facial subdivision of the cube, and is equal to ${{\cal S}_2}^n$.
		Suppose that $G$ is a complete graph $K_k$ of vertices $v_1,v_2,\ldots, v_k$.
		Boolean-gated sets are $\{v_1\}, \{v_2\}, \ldots, \{v_k\}$, and $\{v_1,v_2,\ldots, v_k\}$.
		Hence $G^*$ is a star with $k$ leaves, 
		and is isomorphic to polar space ${\cal S}_k$ (Example~\ref{ex:S_k}); 
		see the left of Figure~\ref{fig:subdivision}.
		Suppose that $G$ is the $n$-product ${K_k}^n$ of complete graph $K_k$. In this case, $G^*$ is isomorphic to ${{\cal S}_k}^n$.
		Suppose that $G$ is a complete bipartite graph $K_{k,l}$.
		Boolean-gated sets are all the singletons, all the edges, and the whole set of vertices.
		Therefore $G^*$ is obtained by 
		subdividing each edge $e = xy$ into a series $xv_e, v_ey$ and 
		joining $v_e$ to a new vertex corresponding to the whole set.
		Then $G^*$ is isomorphic to the polar space ${\cal S}_{k,l}$; see the right of Figure~\ref{fig:subdivision}.
	\end{Ex}
	Basic properties of Boolean-gated sets and the barycentric subdivision are summarized as follows.
\begin{Lem}[{\cite[Lemma 6.7]{CCHO14}}]\label{lem:om_Bgated}
	For an oriented modular graph $\varGamma$,
	a vertex set $X$ is Boolean-gated if and only if $X = [x,y]$ 
	for some vertices $x,y$ with $x \sqsubseteq y$.
	Hence $\varGamma^*$ is the poset of all intervals $[x,y]$ with $x \sqsubseteq y$, 
	where the partial order is the reverse inclusion order.
\end{Lem}
In this way, an oriented modular graph is viewed 
as an amalgamation of complemented modular lattices.

Any singleton $\{v\}$ of a vertex $v$ is Boolean-gated.
Thus we can regard $G \subseteq G^*$. Moreover $G$ is an isometric subspace of $G^*$.
\begin{Thm}[{\cite[Theorem 6.9]{CCHO14} and \cite[Proposition 4.5]{HH150ext}}]\label{thm:G->G*}
	For an swm-graph $G$, the barycentric subdivision $G^*$ is oriented modular.
	Moreover, $G$ is isometrically embedded into $G^*$ by $x \mapsto \{x\}$, i.e., $d_G(x,y) = d_{G^*}(\{x\},\{y\})$.
	In addition, if $G$ is an oriented modular, then it holds
	\[
	d_{G^*}([p,q],[p',q']) = ( d_{G}(p,p') + d_{G}(q,q'))/2 \quad ([p,q],[p',q'] \in G^*).
	\] 
\end{Thm}
The barycentric subdivision $G^*$ of a dual polar graph $G$ 
is identical with the covering graph of the polar space corresponding to $G$. 
Thus the barycentric subdivision of a general swm-graph is obtained 
by replacing each dual polar subgraph induced by a Boolean-gate set  
with the covering graph of the corresponding polar space.
\begin{Lem}[{\cite[Proposition 6.10]{CCHO14}}]\label{lem:well-oriented}
	For an swm-graph $G$, 
	the principal filter of every vertex of $G^*$ is a polar space. 
	Hence $G^*$ is well-oriented.
\end{Lem}
The product of swm-graphs (oriented modular graphs) 
is an swm-graph (oriented modular graph)~\cite[Proposition 2.16]{CCHO14}.
Also the product operation commutes with the barycentric subdivision.
\begin{Lem}\label{lem:product_subdivision}
\begin{description}
	\item[{\rm (1)}] For swm-graphs $G, H$, 
	it holds $(G \times H)^* = G^* \times H^*$.
	\item[{\rm (2) \cite[Lemma 4.6]{HH150ext}}] For oriented modular graphs $\varGamma, \varGamma'$, 
	it holds $(x,x') \sqsubseteq (y,y')$ in $\varGamma \times \varGamma'$ 
	if and only if it holds $x \sqsubseteq y$ in $\varGamma$ and $x' \sqsubseteq y'$ in $\varGamma'$.
\end{description}	
\end{Lem} 
\begin{proof}
	(1). For a Boolean-gated set $X$ in $G$ and Boolean-gated set $Y$ in $H$, 
	it is easy to see that $X \times Y$ is Boolean-gated in $G \times H$.
	Hence it suffices to show that every Boolean-gated set $Z$ in $G \times H$
	is represented in this way.
	Let $X := \{ x \in  G \mid \exists y \in H: (x,y) \in Z\}$ 
	 and $Y := \{ y \in H \mid \exists y \in G: (x,y) \in Z\}$.
	 Obviously $Z \subseteq X \times Y$.
	 We show that the equality holds.
	  Pick $(x,y) \in X \times Y$. 
	  There are $x' \in G$, $y' \in H$ with $(x,y'),(x',y) \in Z$.
	 	Then $(x,y)$ belongs to the metric interval of $(x,y')$ and $(x',y)$ in $G \times H$.
	 	A Boolean-gated set is a special gated set, and is always
	 	$d_{G \times H}$-convex. 
	 	Thus $(x,y)$ belongs to $Z$, and $Z = X \times Y$.
	 	It is easy to see from the definition that
	 	$X$ is Boolean-gated in $G$ and $Y$ is Boolean-gated in $H$.
\end{proof}

\subsubsection{Thickening of swm-graph}\label{subsub:thickening}
The {\em thickening} $G^\varDelta$ of an swm-graph $G$ 
is the graph obtained from $G$ by joining all pairs of vertices 
belonging to a common Boolean-gated set~\cite[Section 6.5]{CCHO14}. 
The shortest path metric with respect to $G^{\varDelta}$
is denoted by $d^{\varDelta} (:= d_{G^{\varDelta}})$, 
where the edge-length of new edges are the same as in $G$. 
A path in $G^{\varDelta}$ is called a {\em $\varDelta$-path}.
In Section~\ref{sec:L-convex}, 
the metric $d^{\varDelta}$ is used to estimate the number of iterations of
the steepest descent algorithm (SDA).
For a vertex $x$ and a nonnegative integer $r$, 
let $B^{\varDelta}_r(x)$ denote the set of vertices $y$ with $d^{\varDelta}(x,y) \leq r$, 
i.e., the ball of center $x$ and radius $r$ in $G^{\varDelta}$.
\begin{Lem}[{\cite[Proposition 6.15]{CCHO14}}]\label{lem:ball_convex} 
	$B^{\varDelta}_r(x)$ is $d$-gated.	
\end{Lem}
The thickening is a kind of $l_{\infty}$-metrization.
The relation between $G$ and $G^\varDelta$ 
is similar to that between $l_{1}$- and $l_{\infty}$-metrics.
\begin{Ex}\label{ex:thickening}
	For a grid graph $G = \ZZ^n$,
	the thickening $G^{\varDelta}$ is obtained from $G$ 
	by joining each pair of vertices $x,y$ with $\|x-y\|_{\infty} \leq 1$.
	The distances of $G$ and $G^{\varDelta}$ are given as
	$d(x,y) = \| x - y\|_1$ and 
	$d^{\varDelta}(x,y) = \|x - y \|_{\infty}$.
	 Here $d$-convex sets ($d$-gated sets) in $G$ are precisely box subsets.
	 Any $l_{\infty}$-ball $B^{\varDelta}_r(x)$ is a box subset, and is $d$-convex.
\end{Ex}
For two graphs $G$ and $H$, the {\em strong product} $G \otimes H$ is 
the graph on $V(G) \times V(H)$ such that vertices $(x,y)$ and $(x',y')$
are adjacent if $x$ and $x'$ are adjacent or equal 
and $y$ and $y'$ are adjacent or equal.
\begin{Lem}\label{lem:product_thickening}
	For two swm-graphs $G$ and $H$, 
	it holds $(G \times H)^{\varDelta} = G^{\varDelta} \otimes H^{\varDelta}$.
\end{Lem}
\begin{proof}
	By Lemma~\ref{lem:product_subdivision},
	$(x,y)$ and $(x',y')$ belong to a common Boolean-gated set in $G \times H$ if 
	and only if $x$ and $x'$ belong to a common Boolean-gated set in $G$ and
	$y$ and $y'$ belong to a common Boolean-gated set in $H$.
	The claim follows from this fact.
\end{proof}

\subsection{CAT$(0)$ space and orthoscheme complex}\label{subsec:CAT(0)}
Here we introduce continuous spaces/relaxations 
into which the discrete structures introduced in Sections~\ref{subsec:lattice} and \ref{subsec:modular}
are embedded, analogously to 
$\ZZ^n \hookrightarrow \RR^n$.
Let $X$ be a metric space with metric $d: X \times X \to \RR_+$.
A {\em path} in $X$ is a continuous map $\gamma$ from $[0,1]$ to $X$.
The length of a path $\gamma$ is defined as
$
\sup \sum_{i=0}^{n-1} d(\gamma(t_i), \gamma(t_{i+1}))$,
where the supremum is taken over all sequences $t_0,t_1,\ldots,t_n$ in $[0,1]$
with $0=t_0 < t_1 < t_2 < \cdots < t_n = 1$ $(n > 0)$.
A path $\gamma$ {\em connects} $x,y \in X$ if $\gamma(0) = x$ and $\gamma(1) = y$.
A {\em geodesic} is a path $\gamma$ satisfying 
$d(\gamma(s), \gamma(t)) = d(\gamma(0), \gamma(1)) |s - t|$ for every $s,t \in [0,1]$.
If every pair of points in $X$ can be connected by a geodesic, 
then $X$ is called a {\em geodesic metric space}.
If every pair of points can be connected by a unique geodesic,
then $X$ is said to be {\em uniquely geodesic}.

\subsubsection{CAT$(0)$ space}
We introduce a class of uniquely-geodesic metric spaces, called CAT(0) spaces; see~\cite{BacakBook,BrHa}.
Let $X$ be a geodesic metric space.
For points $x,y$ in $X$, let $[x,y]$ 
denote the image of some geodesic $\gamma$ connecting $x,y$
(though such a geodesic is not unique).
For $t \in [0,1]$, the point $p$ on $[x,y]$ 
with $d(x,p)/d(x,y) = t$ is denoted by the formal sum $(1-t)x + t y$.
A {\em geodesic triangle} of vertices $x,y,z \in X$ is 
the union of $[x,y]$, $[y,z]$, and $[z,x]$.
A {\em comparison triangle} with $x,y,z \in X$ is 
the union of three segments $[\bar x, \bar y]$, $[\bar y, \bar z]$, and $[\bar z,\bar x]$
in the Euclidean plane $\RR^2$ 
such that $d(x,y) = \|\bar x - \bar y\|_2$, 
$d(y,z) = \|\bar y - \bar z\|_2$, and
$d(z,x) = \|\bar z - \bar x\|_2$.
For $p \in [x,y]$, the {\em comparison point} $\bar p$ is 
the point on $[\bar x, \bar y]$ with $d(x,p) = \| \bar x - \bar p \|_2$.

A geodesic metric space is called CAT$(0)$ if 
for every geodesic triangle $\Delta = [x,y] \cup [y,z] \cup [z,x]$ and every $p,q \in \Delta$, it holds
$
d(p,q) \leq \| \bar p - \bar q\|_{2}. 
$
Intuitively this says that triangles in $X$ are thinner than Euclidean plane triangles.
\begin{Prop}[{\cite[Proposition 1.4]{BrHa}}]\label{prop:uniquely-geodesic}
	A {\rm CAT}$(0)$ space is uniquely geodesic.
\end{Prop}
By this property, several convexity concepts are  
defined naturally in CAT(0) spaces.
Let  $X$ be a {\rm CAT}$(0)$ space. 
A subset $C$ of $X$ is said to be {\em convex} if
for arbitrary $x,y \in C$ the geodesic $[x,y]$ is contained in $C$. 
A function $f: X \to \overline\RR$ on $X$ is said to {\em convex}
if for every $x,y \in X$ and $t \in [0,1]$ it holds
\begin{equation}\label{eqn:convexity}
(1-t) f(x) + t f(y) \geq f((1-t) x + t y).
\end{equation}
%
The Euclidean space is an obvious example of a CAT$(0)$ space.
Other examples include a {\em metric tree} 
($1$-dimensional contractible complex endowed with the length metric), and the product of metric trees.

\subsubsection{Orthoscheme complex and Lov\'asz extension}\label{subsub:orthoscheme}
An $n$-dimensional {\em orthoscheme} is a simplex in $\RR^n$ with vertices
\[
0, e_1, e_1 + e_2, e_1 + e_2 + e_3, \ldots, e_1 + e_2 + \cdots + e_n,
\]
where $e_i$ is the $i$th unite vector.
An orthoscheme complex, introduced by Brady and McCammond~\cite{BM10} in the context of geometric group theory,  
is a metric simplicial complex obtained by gluing orthoschemes as in Figure~\ref{fig:ortho}.
In our view, an orthoscheme complex is a generalization of the well-known simplicial subdivision of a cube $[0,1]^n$ on which the Lov\'asz extension of $f:\{0,1\}^n \to \RR$ is defined via the piecewise interpolation; 
see Example~\ref{ex:Boolean} below.
	\begin{figure}[t]
		\begin{center}
			\includegraphics[scale=0.7]{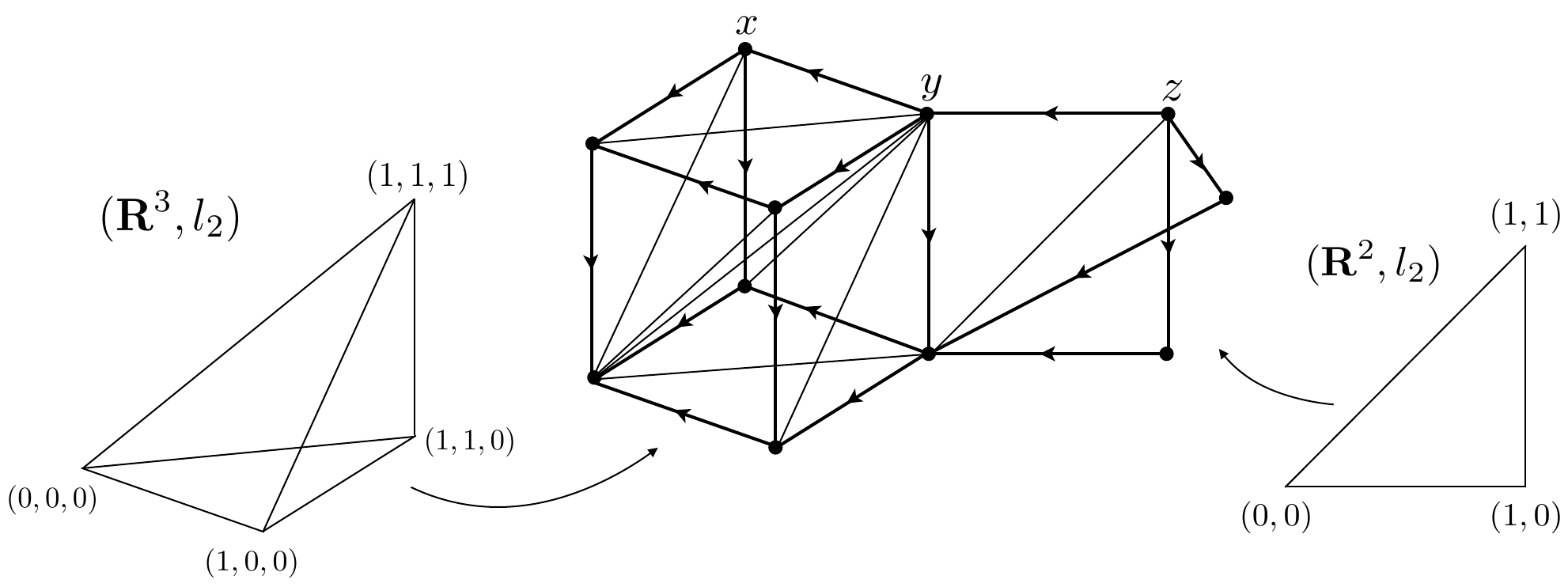}
			\caption{Orthosheme complex}
			\label{fig:ortho}
		\end{center}		
	\end{figure}\noindent

Let us introduce formally.
Let ${\cal P}$ be a {\em graded} poset,  that is,  
a poset having a function 
$r: {\cal P} \to \ZZ$ with $r(q) = r(p) + 1$ whenever $q \to p$. 
Let $K({\cal P})$ denote the set of 
all formal convex combinations $\sum_{p \in {\cal P}} \lambda(p) p$ of elements in ${\cal P}$ 
such that the nonzero support $\{ p \mid \lambda(p) > 0 \}$ forms a chain.
The set of all formal combinations of some chain $C$
is called a {\em simplex} of $K({\cal P})$.
For a simplex $\sigma$ corresponding to a chain $C = p_0 \prec p_1 \prec \cdots \prec p_k$, 
define a map $\varphi_{\sigma}$ from $\sigma$ to 
the $(r(p_k)- r(p_0))$-dimensional orthoscheme
by
\begin{equation}\label{eqn:phi_sigma}
\varphi_{\sigma}(x) = \sum_{i=1}^{k} \lambda_i (e_1 + e_2 + \cdots + e_{r(p_i) - r(p_0)}), 
\end{equation}
where $x = \sum_{i=0}^k \lambda_k p_k \in \sigma$.
For each simplex $\sigma$, 
a metric $d_{\sigma}$ on $\sigma$ is defined as
\begin{equation*}
d_{\sigma}(x,y) := \| \varphi_\sigma (x) - \varphi_\sigma (y) \|_2 \quad (x,y \in \sigma). 
\end{equation*}
The length of a path $\gamma$ in $K({\cal P})$
is defined as $\sup \sum_{i=0}^{m-1} d_{\sigma_i}(\gamma(t_i),\gamma(t_{i+1}))$, 
where $\sup$ is taken over all
$0 = t_0 < t_1 < t_2 < \cdots < t_m = 1$ $(m \geq 1)$ 
such that $\gamma([t_i,t_{i+1}])$ belongs to a simplex $\sigma_i$ for each $i$.
Then the metric on $K({\cal P})$ is (well-)defined as above.
The resulting metric space $K({\cal P})$ 
is called the {\em orthoscheme complex} of ${\cal P}$~\cite{BM10}.
If the lengths of chains are uniformly bounded,
then $K({\cal P})$ is a (complete) 
geodesic metric space~\cite[Theorem 7.19]{BrHa}.

By considering the orthoscheme complex,
we can define an analogue of the Lov\'asz extension 
for a function defined on any graded poset~${\cal L}$.
For a function $f:{\cal L} \to \overline{\RR}$, 
the {\em Lov\'asz extension} $\overline f$ of $f$ is a function on the orthoscheme complex $K({\cal L})$
defined by
\begin{equation}
\overline f (x) := \sum_{i} \lambda_i f(p_i) \quad  \left(x = \sum_{i} \lambda_i p_i \in K({\cal L}) \right).
\end{equation}
In the case where $K({\cal L})$ is {\rm CAT}$(0)$, 
we can discuss the convexity property of $\overline f$
with respect to the CAT$(0)$-metric.
The following examples of $K({\cal L})$ show 
that our Lov\'asz extension actually 
generalizes the Lov\'asz extension for functions on $\{0,1\}^n$ and on $\{-1,0,1\}^n$.
%
\begin{Ex}\label{ex:Boolean}
	Let ${\cal L}$ be a Boolean lattice with atoms $a_1,a_2,\ldots,a_n$. 
	Then $K({\cal L})$ consists of points 
	$p = \sum_{k=0}^n \lambda_k 
	(a_{i_1} \vee a_{i_2} \vee \cdots \vee a_{i_k})$ 
	for some permutation $(i_1,i_2,\ldots,i_n)$ of $\{1,2,\ldots,n\}$ and nonnegative coefficients $\lambda_k$ whose sum is equal to one.
	The map $p = \sum_{k=0}^n \lambda_k 
	(a_{i_1} \vee a_{i_2} \vee \cdots \vee a_{i_k}) \mapsto \sum_{k=0}^n \lambda_k 
	(e_{i_1} + e_{i_2} + \cdots + e_{i_k})$ is an isometry 
	from $K({\cal L})$ to cube $[0,1]^n$ 
	with respect to $l_2$-metric~\cite[Lemma 7.7]{CCHO14}.
	In particular, $K({\cal L})$ is viewed as the well-known simplicial subdivision of cube $[0,1]^n$
	consisting of vertices  
	$
	 0, e_{i_1}, e_{i_1} + e_{i_2}, \ldots, e_{i_1} + e_{i_2} + \cdots + e_{i_n}
	$
	for all permutations $(i_1,i_2,\ldots,i_n)$.
	Let ${\cal L}$ be a distributive lattice of rank $n$.
	By the Birkhoff representation theorem, ${\cal L}$ is 
	the poset of principal ideals of a poset ${\cal P}$ of $n$ elements $a_1,a_2,\ldots,a_n$.
	Then $K({\cal L})$ is isometric to the well-known simplicial subdivision of the {\em order polytope}
    $\{ x \in [0,1]^n \mid x_i \geq x_j \ (i,j: a_i \prec a_j )\}$
    into simplices of the above form~\cite[Proposition 7.7]{CCHO14}.
\end{Ex}	
\begin{Ex}
	Consider ${\cal S}_2 = \{-1,0,1\}$ and the product ${{\cal S}_2}^n$ (Example~\ref{ex:S_k}).
	Then $K({{\cal S}_2}^n)$ is CAT(0), isomorphic to the simplicial subdivision of cube $[-1,1]^n$
	consisting of simplices of vertices 
	$
	0, s_{i_1}, s_{i_1} + s_{i_2}, \ldots, s_{i_1} + s_{i_2} + \cdots + s_{i_n}$,
	where $(i_1,i_2,\ldots,i_n)$ is a permutation of $\{1,2,\ldots,n\}$ 
	and $s_{i} \in \{e_i, - e_i \}$ for $i=1,2,\ldots,n$.
\end{Ex}	
In these examples, $K({\cal L})$ 
is isometric to a convex polytope in the Euclidean space, 
and obviously is CAT(0). 
In the case of a modular lattice ${\cal L}$, 
the orthoscheme complex $K({\cal L})$
cannot be realized as a convex polytope of a Euclidean space.
However $K({\cal L})$ still has the CAT(0) property.

\begin{Thm}[{\cite[Theorem 7.2]{CCHO14}}]\label{thm:modularlattice}
	The orthoscheme complex of a modular lattice is {\rm CAT}$(0)$.
\end{Thm}
This property holds for every polar space.
\begin{Thm}[{\cite[Proposition 7.4]{CCHO14}}]\label{thm:polar_CAT(0)}
	The orthoscheme complex of a polar space is {\rm CAT}$(0)$.
\end{Thm}
By using these theorems, in Section~\ref{sec:submodular}, we characterize 
a submodular function on a modular lattice/polar space ${\cal L}$ as a function on ${\cal L}$ such that its Lov\'asz extension 
is convex on $K({\cal L})$.

An oriented modular graph $\varGamma$ is graded (as a poset)~\cite[Theorem 6.2]{CCHO14}.
Thus we can consider $K(\varGamma)$.
A special interest lies in
the subcomplex $K'(\varGamma)$ of $K(\varGamma)$ 
consisting of simplices that correspond to
a chain $p_0 \prec p_1 \prec \cdots \prec p_k$ 
with $p_0 \sqsubseteq p_k$. 
Figure~\ref{fig:ortho} depicts $K'(\varGamma)$,
where $x \prec y \prec z$ but $x \not \sqsubseteq z$, 
hence the simplex corresponding to $\{x,y,z\}$ does not exist.
If $\varGamma$ has a uniform edge-length $s$, 
we multiply $\varphi_{\sigma}$ by $s$ in~(\ref{eqn:phi_sigma}).
	\begin{figure}[t]
		\begin{center}
			\includegraphics[scale=0.6]{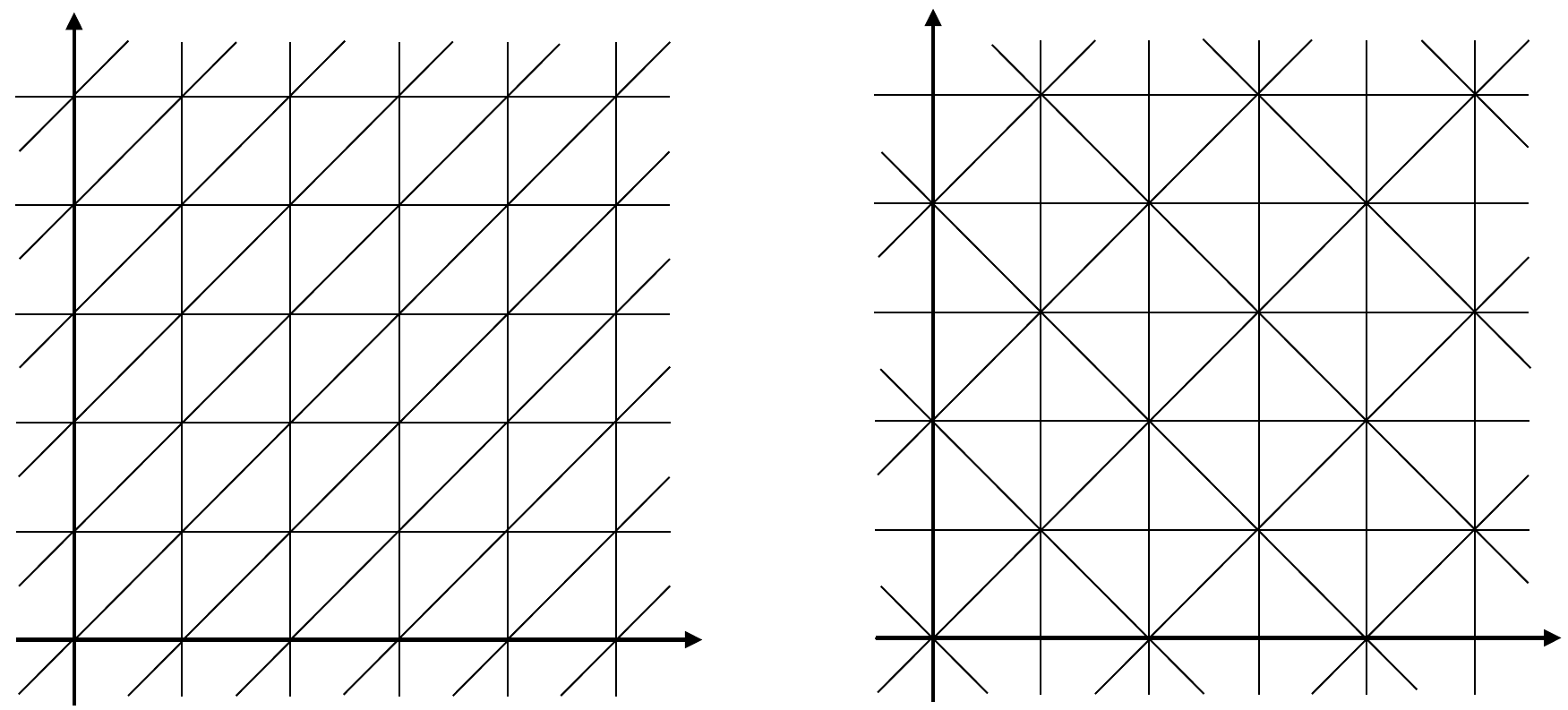}
			\caption{$K'(\vec{\ZZ}^2)$ (left) and $K(\check{\ZZ}^2)$ (right)}
			\label{fig:UJ}
		\end{center}
	\end{figure}\noindent
\begin{Ex}\label{ex:UJ}
    Consider linearly-oriented grid graph $\vec{\ZZ}^n$ (Example~\ref{ex:grid}).
	Then $x \sqsubseteq y$ 
	if and only if $x \leq y$ and $\|x - y\|_{\infty} \leq 1$.
	Therefore the orthoscheme complex 
	$K'(\vec\ZZ^n)$ is the simplicial subdivision of $\RR^n$ into simplices of vertices 
	$
	x, x+ e_{i_1}, x+ e_{i_1} + e_{i_2}, \ldots, x+ e_{i_1} + e_{i_2} + \cdots + e_{i_n}
	$
	over all $x \in \ZZ^n$ and all permutations $(i_1,i_2,\ldots,i_n)$ of $\{1,2,\ldots,n\}$; see the left of Figure~\ref{fig:UJ}.
	Next consider $\check{\ZZ}^{n}$.
	Then $x \sqsubseteq y$ if and only if 
	$\|x - y\|_{\infty}  \leq 1$ and $x_i \not \in 2\ZZ$ implies $x_i = y_i$.
	The orthoscheme complex 
	$K'(\check{\ZZ}^n) = K(\check{\ZZ}^n)$ 
	is the simplicial subdivision of $\RR^n$ into simplices of vertices 
	$
	x, x+ s_{i_1}, x+ s_{i_1} + s_{i_2}, \ldots, x+ s_{i_1} + s_{i_2} + \cdots + s_{i_n}
	$
	over all $x \in (2 \ZZ)^n$,  all permutations $(i_1,i_2,\ldots,i_n)$ of $\{1,2,\ldots,n\}$, and $s_{i} \in \{e_i, - e_i \}$ for $i=1,2,\ldots,n$; 
	see the right of Figure~\ref{fig:UJ}.
	The simplicial complex $K(\check{\ZZ}^n)$
	is known as the 
	{\em Euclidean Coxeter complex of type C} 
	in building theory~\cite{BuildingBook}, and is called the {\em Union-Jack division} in~\cite{Fujishige14}.
\end{Ex}
The above example is easily generalized to 
the product of oriented trees, 
where the orthoscheme complex is 
a simplicial subdivision of the product of metric trees, and is CAT$(0)$. 
A {\em Euclidean building (of type C)} is a further generalization, 
which is defined as a simplicial complex $K$ 
that is the union of subcomplexes, called {\em apartments}, 
satisfying the following axiom (see \cite[Definitions 4.1 and 11.1]{BuildingBook}):
\begin{description}
	\item[B0:] Each apartment is isomorphic to $K(\check{\ZZ}^n)$.
	\item[B1:] For any two simplices $A, B$ in $K$,
	there is an apartment $\varSigma$ containing them.
	\item[B2:] If $\varSigma$ and $\varSigma'$ are apartments containing two simplices $A, B$,
	then there is an isomorphism $\varSigma \to \varSigma'$ being identity on $A$ and $B$.
\end{description}
A Euclidean building naturally gives rise to an oriented modular graph and swm-graph
as follows.
A {\em labeling} is a map $\phi$ from the vertex set of $K$ to $\{0,1,2,\ldots,n\}$
such that two vertices $x,y$ in any simplex 
have different labels $\phi(x) \neq \phi(y)$.
A building has a labeling.
A labeling is uniquely determined 
from its (arbitrarily specified) values on any fixed maximal simplex $A$.
Consider an arbitrary apartment $\varSigma$.
By $\varSigma \simeq K(\check{\ZZ}^n)$, 
the vertex set of $\varSigma$ is identified with $\ZZ^n$.
Define a labeling $\phi$ on $\varSigma$ by
\[
\ZZ^n \ni x = (x_1,x_2,\ldots,x_n) \mapsto  \mbox{the number of $i$ with $x_i$ odd} \in \{0,1,2,\ldots,n\}.
\]
This labeling is uniquely extended to a global labeling of $K$.
Let $\varGamma(K)$ be the subgraph of the $1$-skeleton of $K$ such that
edges $xy$ are given and oriented as $x \to y$ if 
$\phi(x) = \phi(y) - 1$.
Let $H(K)$ denote the graph on the vertices with label $0$ obtained by joining
each pair of vertices $x, y$ 
having a common neighbor (of label 1) in $\varGamma(K)$.
\begin{Thm}[{\cite[Theorem 6.23]{CCHO14}}]\label{thm:building_CAT(0)}
	Let $K$ be a Euclidean building. Then
	$H(K)$ is an swm-graph,  $\varGamma(K)$ is an oriented modular graph, $\varGamma(K) \simeq H^*(K)$, and $K \simeq K(\varGamma(K))$.
\end{Thm}
%
The orthoscheme complex $K(\varGamma(K)) \simeq K$ 
is the same as the standard metrization of $K$, and is known to be CAT$(0)$; see \cite[Section 11]{BuildingBook}.
\begin{Thm}[{see \cite[Theorem 11.16]{BuildingBook}}]
		Let $K$ be a Euclidean building. The orthoscheme complex $K(\varGamma(K))$ is CAT$(0)$. 
\end{Thm}
By using this theorem, 
in Section~\ref{sec:L-convex}, 
we characterize L-convex functions 
on a Euclidean building
by means of the convexity of the Lov\'asz extension.


\section{Submodular functions on modular semilattices}\label{sec:submodular}
In this section, we study submodular functions on modular semilattices.
We quickly review  the (rather complicated) definition 
and basic results established in \cite{HH150ext}.
Then we focus on submodular functions on polar spaces (Section~\ref{subsec:polar-submo}).
In a polar space, the defining inequality 
for submodularity is quite simple (Theorems~\ref{thm:frac_join_of_polar_space} and \ref{thm:polar_submo}). 
We show that the submodularity is characterized by 
the convexity of the Lov\'asz extension 
(Section~\ref{subsub:polar-Lov}), 
and establish the relation to $k$-submodular and $\alpha$-bisubmodular functions (Sections~\ref{subsub:k-submo} and \ref{subsub:alpha_bisub}).

\subsection{Basics}
\subsubsection{Definition}
Let ${\cal L}$ be a modular semilattice and $d$ denote the path metric of the covering graph of ${\cal L}$.
To define ``join" of $p,q \in {\cal L}$, 
the previous work~\cite{HH150ext} introduced the following construction.\footnote{The argument in~\cite{HH150ext} works even if $|{\cal L}| = \infty$.}
Figure~\ref{fig:ConvIpq} is  a conceptual diagram.
\begin{figure}[t]
	\begin{center}
		\includegraphics[scale=0.75]{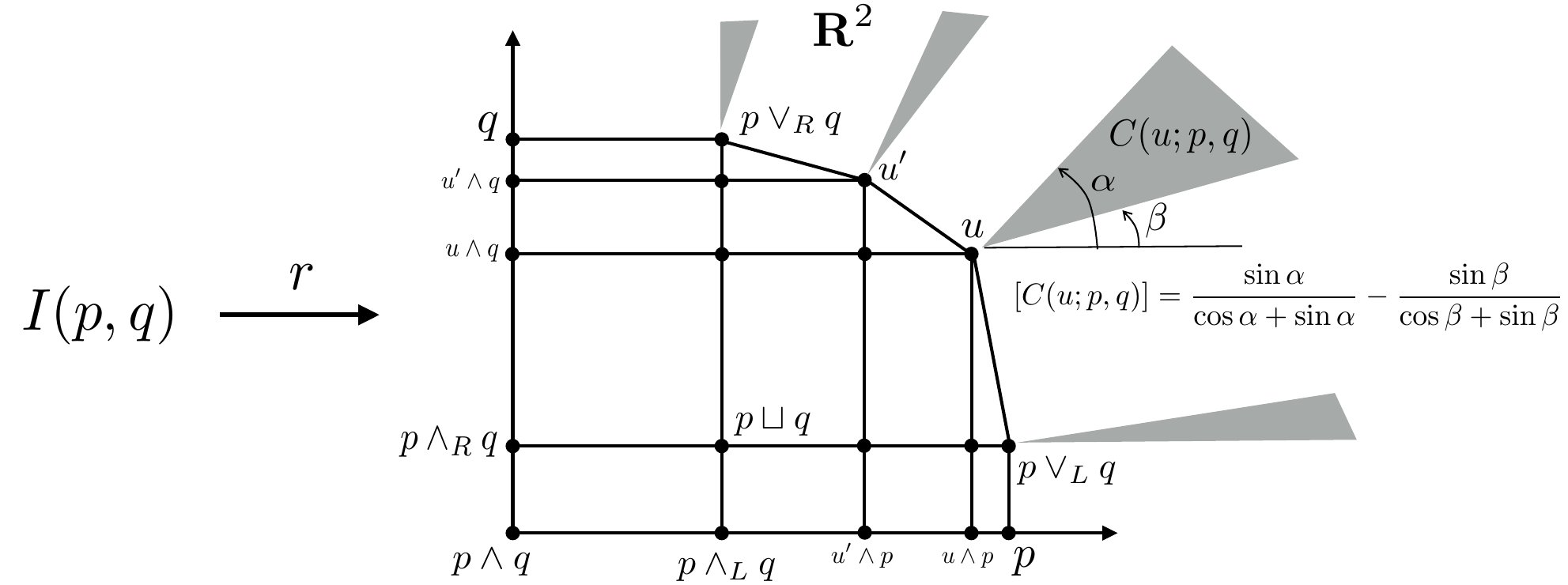}
		\caption{Construction of the fractional join}
			\label{fig:ConvIpq}
	\end{center}
\end{figure}\noindent

\begin{description}
	\item[{\rm (i)}] The metric interval $I(p,q)$ is equal to
	the set of elements $u$ that is represented as $u = a \vee b$ 
	for some $(a,b) \in [p \wedge q,p] \times [p \wedge q,q]$, where
	such a representation is unique, and $(a,b)$ equals $(u \wedge p, u \wedge q)$~\cite[Lemma 2.15]{HH150ext}.
	\item[{\rm (ii)}] For $u \in I(p,q)$, let $r(u;p,q)$ be the vector in $\RR^2_+$ defined by
	\begin{equation}\label{eqn:r}
	r(u; p,q) = (r(u \wedge p)- r(p \wedge q), r(u \wedge q) - r(p \wedge q)).
	\end{equation}
	\item[{\rm (iii)}] Let $\Conv I(p,q)$ denote 
	the convex hull of vectors $r(u;p,q)$ for all $u \in I(p,q)$.
	\item[{\rm (iv)}] Let ${\cal E}(p,q)$ be the set of elements $u$ in $I(p,q)$
	such that $r(u;p,q)$ is a maximal extreme point of $\Conv I(p,q)$.
	Then ${\cal E}(p,q) \ni u \mapsto r(u;p,q)$ is injective~\cite[Lemma 3.1]{HH150ext}.
	\item[{\rm (vi)}] For $u \in {\cal E}(p,q)$, 
	let $C(u;p,q)$ denote 
	the nonnegative normal cone at $r(u; p,q)$:
	\[
	C(u;p,q) := \left\{ w \in \RR_+^2 \bigmid \langle w, r(u;p \wedge q) \rangle = \max_{x \in \Conv I(p,q)} \langle w, x \rangle \right\}.
	\]
	\item[{\rm (vii)}] For a convex cone $C$ in $\RR_+^2$ represented as
	$
	C = \{(x,y) \in \RR_+^2 \mid y \cos \alpha \leq x \sin \alpha, y \cos \beta \geq x \sin \beta \}
	$
	for some $0 \leq \beta \leq \alpha \leq \pi/2$, 
	let 
	\[
	[C] := \frac{\sin \alpha}{\sin \alpha + \cos \alpha} - 
	\frac{\sin \beta}{\sin \beta + \cos \beta}.
	\]
	\item[{\rm (viii)}] The {\em fractional join} of $p,q \in {\cal L}$ is defined as 
	the formal sum $\displaystyle \sum_{u \in {\cal E}(p,q)} [C(u;p,q)] u.$
\end{description}
We use an alternative form of the fractional join.
Let ${\cal E}({\cal L})$ denote the set of 
all binary operations 
$\theta: {\cal L} \times {\cal L} \to {\cal L}$ such that
(i) $\theta(p,q)$ belongs to ${\cal E}(p,q)$ for $p,q \in {\cal L}$ and (ii) the cone
	\[
	C(\theta) := \bigcap_{p,q \in {\cal L}} C(\theta(p,q);p,q)
	\]
	has an interior point, i.e., $[C(\theta)] > 0$.
Then it is shown in \cite[Proposition 3.3]{HH150ext} 
that the following equality holds:
\begin{equation}\label{eqn:frac_join}
\sum_{\theta \in {\cal E}({\cal L})} [C(\theta)] \theta(p,q) = \sum_{u \in {\cal E}(p,q)} [C(u;p,q)] u \quad (p,q \in {\cal L}).	
\end{equation}
	The {\em fractional join operation} of ${\cal L}$, 
	denoted by the formal sum  
	$\sum_{\theta \in {\cal E}({\cal L})} [C(\theta)] \theta$, 
	is a function on ${\cal L} \times {\cal L}$ defined by $(p,q) \mapsto \sum_{\theta \in {\cal E}({\cal L})} [C(\theta)] \theta(p,q)$.

We are now ready to define submodular functions on ${\cal L}$.
A function $f:{\cal L} \to \overline\RR$ is called {\em submodular} if
it satisfies 
\begin{equation}\label{eqn:submodular}
f(p) + f(q) \geq f(p \wedge q) + \sum_{\theta \in {\cal E}({\cal L)}}[C(\theta)] f(\theta(p,q))
\quad (p,q \in {\cal L}),
\end{equation}
where $\sum_{\theta \in {\cal E}({\cal L)}}[C(\theta)] f(\theta(p,q))$
may be replaced by $\sum_{u \in {\cal E}(p,q)} [C(u;p,q)] f(u)$.
If ${\cal L}$ is a (modular) lattice, then 
the fractional join is equal to the join $\vee$, 
and our definition of submodularity coincides with the usual definition.
The class of our submodular functions includes all constant functions, 
and is closed under nonnegative combinations and direct sums~\cite[Lemma 3.7]{HH150ext}. Here
the {\em direct sum} of two functions $g:X \to \overline{\RR}$ and $h:Y \to \overline{\RR}$
is the function on $X \times Y$ defined by $(x,y) \mapsto g(x) + h(y)$.
A canonical example of submodular functions is 
the distance function of a modular semilattice.
\begin{Prop}[{\cite[Theorem 3.6]{HH150ext}}]
	Let ${\cal L}$ be a modular semilattice.
	The distance function $d$
	is submodular on ${\cal L} \times {\cal L}$.
\end{Prop}

Consider the important case where a modular semilattice ${\cal L}$ in question is 
the product of (smaller) $n$ modular semilattices ${\cal L}_1,{\cal L}_2, \ldots, {\cal L}_n$.
For binary operations $\theta_i: {\cal L}_i \times {\cal L}_i \to {\cal L}_i$ 
for $i=1,2,\ldots,n$, the {\em componentwise extension} 
$(\theta_1, \theta_2,\ldots, \theta_n)$ is the binary operation ${\cal L} \times {\cal L} \to {\cal L}$ 
defined by 
\[
(\theta_1, \theta_2,\ldots, \theta_n)(p,q) := 
(\theta_1(p_1,q_1), \theta_2(p_2,q_2), \ldots, \theta_n(p_n,q_n))
\]
for $p = (p_1,p_2,\ldots,p_n), q = (q_1,q_2,\ldots,q_n) \in 
{\cal L} = {\cal L}_1 \times {\cal L}_2 \times \cdots \times {\cal L}_n$.
Then the fractional join of ${\cal L}$ is decomposed as follows.
\begin{Prop}[{\cite[Proposition 3.4]{HH150ext}}]\label{prop:frac_join_prod} 
	\[
	\sum_{\theta \in {\cal E}({\cal L})} [C(\theta)] \theta =
	\sum_{\theta_1,\theta_2,\ldots,\theta_n}  [C(\theta_1) \cap C(\theta_2) \cap \cdots \cap C(\theta_n)] (\theta_1,\theta_2,\ldots, \theta_n),
	\]
	where the summation over $\theta_i$ is taken over all binary operations in ${\cal E}({\cal L}_i)$ for $i=1,2,\ldots,n$.
	Moreover, 
	if ${\cal L}_i = {\cal L}_j$ and $\theta_i \neq \theta_j$ for some $i,j$, 
	then  $[C(\theta_1) \cap C(\theta_2) \cap \cdots \cap C(\theta_n)] = 0$.
\end{Prop}
\subsubsection{Left join $\vee_L$, right join $\vee_R$, and pseudo join $\sqcup$}
We introduce three canonical operations.
For $p,q \in {\cal L}$, there exists a unique maximal element $u \in [p \wedge q, q]$
such that $p$ and $u$ have the join $p \vee u$.
Indeed, suppose that for $u,u' \in [p \wedge q, q]$
both joins $u \vee p$ and $u' \vee p$ exist.
By the definition of a modular semilattice, $u \vee u' \vee p$ exists.
Define the {\em left join} $p \vee_L q$ as the join of $p$ and the maximal element $u$.
Also define the {\em right join} $p \vee_R q$ as the join of $q$ and the unique maximal element 
$v \in [p \wedge q, p]$ with $v \vee q \in {\cal L}$. 
Define the {\em pseudo join} $\sqcup$ as the meet of the left join and the right join:
\begin{equation}\label{eqn:sqcup}
p \sqcup q := (p \vee_L q) \wedge (p \vee_R q) \quad (p,q \in {\cal L}).
\end{equation}
\begin{Ex}[$\vee_L$, $\vee_R$, $\sqcup$ in ${{\cal S}_k}^n$]\label{ex:op-S_k^n}
In 	${{\cal S}_{k}}^{n}$, 
the left join $p \vee_L q$ is obtained from $p$ 
by replacing $p_i$ with $q_i$ for each $i$ with $q_i \neq 0 = p_i$. 
Then $p \sqcup q$ is given by
\begin{equation}
(p \sqcup q)_i := \left\{ 
\begin{array}{ll}
p_i \vee q_i & {\rm if}\ \mbox{$p_i \preceq q_i$ or $p_i \preceq q_i$}, \\
0 & {\rm if}\  0 \neq p_i \neq q_i \neq 0.
\end{array} \right.
\end{equation}
Hence $\sqcup$ here equals the $\sqcup$ used in~\cite{HK12}, 
where $\wedge$ here equals the $\sqcap$ there.
The left and right joins are represented as
$p \vee_L q = (p \sqcup q) \sqcup p$ and
$p \vee_R q = (p \sqcup q) \sqcup q$. 
\end{Ex}

\subsubsection{Submodularity with respect to valuation}\label{subsub:valuation}

A {\em valuation} of a modular semilattice ${\cal L}$
is a function $v: {\cal L} \to \RR$ such that
\begin{description}
	\item[{\rm (1)}] $v(p) < v(q)$ for $p,q \in {\cal L}$ with $p \prec q$, and
	\item[{\rm (2)}] $v(p) + v(q) = v(p \wedge q) + v(p \vee q)$ 
	for $p,q \in {\cal L}$ with $p \vee q \in {\cal L}$.
\end{description}
The rank function $r$ is a valuation.

For a valuation $v$, let $v(u; p,q)$ be defined by 
replacing $r$ with $v$ in $r(u; p,q)$ in (\ref{eqn:r}).
Then $\Conv I(p,q)$, ${\cal E}(p,q)$, $C(u; p,q)$, 
$C(\theta)$, and ${\cal E}({\cal L})$
are defined by replacing $r(u;p,q)$ with $v(u;p,q)$.
In this setting with a general valuation $v$,
the fractional join and the fractional join operation are also defined, and (\ref{eqn:frac_join}) holds.
The corresponding submodular functions are called submodular functions
{\em with respect to valuation $v$}.

Suppose that ${\cal L}$ is the product of modular semilattices ${\cal L}_i$ for $i=1,2,\ldots,n$ and $v$ is a valuation of ${\cal L}$.
Then there exist valuations $v_i$ on ${\cal L}_i$ for $i=1,2,\ldots,n$ 
such that $v$  is represented as
$v(x) = v_1(x_1) + v_2 (x_2) + \cdots + v_n(x_n)$ for 
$x = (x_1,x_2,\ldots,x_n) \in {\cal L}$.
Then Proposition~\ref{prop:frac_join_prod} holds, 
where $C(\theta_i)$ and ${\cal E}({\cal L}_i)$ are defined 
with respect to  the valuation $v_i$, and by ${\cal L}_i = {\cal L}_j$ 
we mean that ${\cal L}_i = {\cal L}_j$ (as a poset) and $v_i = v_j$.

\subsubsection{Minimization} 
Here we consider the problem of minimizing submodular functions 
on the product of finite modular semilattices ${\cal L}_1, {\cal L}_2, \ldots, {\cal L}_n$.
We do not know whether this problem is tractable in the value-oracle model.
We consider the VCSP situation, where submodular function in question is given 
by the sum of submodular functions with a small number of variables.
Let us formulate the {\em Valued Constraint Satisfaction Problem (VCSP)} more precisely.
An instance/input of the problem consists of
finite sets (domains) $D_1, D_2, \ldots, D_n$, 
functions 
$f_i: D_{i_1} \times D_{i_2} \times \cdots \times D_{i_{k_i}} \to \overline{\RR}$
with $i=1,2,\ldots,m$ and $1 \leq i_1 < i_2 < \cdots < i_{k_i} \leq n$, 
where each $f_i$ is given as the table of all function values. 
The objective is to find $p = (p_1,p_2,\ldots, p_n) \in D_1 \times D_2 \times \cdots \times D_n$
that minimizes
$
\sum_{i=1}^m f_i (p_{i_1}, p_{i_2},\ldots,p_{i_{k_i}}).
$
In this situation, the size of the input is $O( n N + m N^k)$, 
where $N := \max_{i} |D_i|$ and $k := \max_{i} k_i$.
As a consequence of a tractability criterion of VCSP 
by Thapper and \v{Z}ivn\'y~\cite{TZ12FOCS} (see also \cite{KTZ13}), we have:
\begin{Thm}[{\cite[Theorem 3.9]{HH150ext}}]\label{thm:vscp}
	Suppose that each $D_i$ is a modular semilattice and $f_i$ is submodular on 
	$D_{i_1} \times D_{i_2} \times \cdots \times D_{i_{k_i}}$.
   Then VCSP is solved in strongly polynomial time.
\end{Thm}

\begin{Rem}
	Well-known oracle-minimizable classes of our submodular functions are those on ${\cal L} = {{\cal S}_1}^n$
	(corresponding to the ordinary submodular functions) and on
	${{\cal S}_2}^n$ (corresponding to bisubmodular functions).
	Fujishige, Tanigawa and Yoshida~\cite{FTY14} and
	Huber and Krokhin~\cite{HK14} showed 
	the oracle tractability of $\alpha$-bisubmodular functions 
	($=$ submodular functions on ${{\cal S}_2}^n$ with valuations); 
	see Section~\ref{subsub:alpha_bisub}. 
	Consider submodular functions on the product ${\cal L}$ of complemented modular lattices of rank 2 ({\em diamonds}). 
	Following a pioneering work of Kuivinen~\cite{Kuivinen11} on NP $\cap$ co-NP characterization,  
	Fujishige, Kir\'aly, Makino, Takazawa, and Tanigawa~\cite{FKMTT14} 
	proved the oracle-tractability of this class of submodular functions. 
\end{Rem}

\subsection{Submodular functions on polar spaces}\label{subsec:polar-submo}
Here we consider submodular functions on polar spaces; 
see Section~\ref{subsec:lattice} for polar spaces.
It turns out that they are natural generalization of bisubmodular functions. 
We first show the explicit formula of the fractional join in a polar space.
\begin{Thm}\label{thm:frac_join_of_polar_space}
	In a polar space, 
	the fractional join operation is equal to
		$(1/2)\vee_L + (1/2)\vee_R$. 
\end{Thm}
Thus submodular functions on polar space ${\cal L}$
are functions satisfying
\begin{equation}\label{eqn:submo_veeLveeR}
f(p) + f(q) \geq f(p \wedge q) + \frac{1}{2} f(p \vee_L q) + \frac{1}{2} f(p \vee_R q) \quad (p,q \in {\cal L}). 
\end{equation}
We present an alternative characterization of submodularity using pseudo join $\sqcup$.
\begin{Thm}\label{thm:polar_submo}
	Let ${\cal L}$ be a polar space. 
	For a function $f: {\cal L} \to \overline{\RR}$, the following conditions are equivalent:
	\begin{description}
		\item[{\rm (1)}] $f$ is submodular on ${\cal L}$.
		\item[{\rm (2)}] $f(p) + f(q) \geq f(p \wedge q) + f(p \sqcup q)$ holds for $p,q \in {\cal L}$.
		\item[{\rm (3)}] $f$ is bisubmodular on each polar frame 
		${\cal F} \simeq {{\cal S}_{2}}^n$.
	\end{description}	
\end{Thm}

\subsubsection{Convexity of the Lov\'asz extension}\label{subsub:polar-Lov}
We discuss the convexity property of Lov\'asz extensions 
with respect to the metric on the orthoscheme complex.
As seen in Example~\ref{ex:Boolean}, 
the orthoscheme complex of a distributive lattice is isometric 
to the order polytope in $[0,1]^n$.
It is well-known that a submodular function on a distributive lattice 
can be characterized by the convexity of its Lov\'asz extension on 
the order polytope.
\begin{Thm}[{\cite{Lovasz83}; see \cite{FujiBook}}]\label{thm:Lovasz}
	Let ${\cal L}$ be a distributive lattice.
	A function $f:{\cal L} \to \overline{\RR}$ is submodular 
	if and only if the Lov\'asz extension $\overline f: K({\cal L}) \to  \overline{\RR}$ is convex.
\end{Thm}
We first establish an analogy of this characterization 
for submodular functions on modular lattices.
Recall Theorem~\ref{thm:modularlattice} that 
the orthoscheme complex of a modular lattice is CAT$(0)$. 
\begin{Thm}\label{thm:Lovasz_modular_lattice}
	Let ${\cal L}$ be a modular lattice.
	A function $f: {\cal L} \to \overline{\RR}$ is submodular if and only if
	the Lov\'asz extension $\overline f: K({\cal L}) \to \overline{\RR}$ 
	is convex.
\end{Thm}
Consider the case where ${\cal L}$ is ${{\cal S}_2}^n = \{-1,0,1\}^n$.
In this case, submodular functions on ${{\cal S}_2}^n$ in our sense 
are bisubmodular functions; see Section~\ref{subsub:k-submo}. 
Also $K({{\cal S}_2}^n)$ is isometric to $[-1,1]^n$.
Thus the Lov\'asz extension of a function $f:{{\cal S}_2}^n \to \overline\RR$
is the same as that given by Qi~\cite{Qi88}.
\begin{Thm}[\cite{Qi88}]\label{thm:Qi88}
	A function $f:{{\cal S}_2}^n \to \overline\RR$ is bisubmodular if and only if
	$\overline f:K({{\cal S}_2}^n) \to \overline{\RR}$ is convex.
\end{Thm}
Recall Example~\ref{ex:S_k} 
that ${{\cal S}_k}^n$ is a polar space.
By Theorem~\ref{thm:polar_CAT(0)}, 
the orthoscheme complex of a polar space is CAT$(0)$.
The generalization of Theorem~\ref{thm:Qi88} is the following.
\begin{Thm}\label{thm:Lovasz_polar_space}
	Let ${\cal L}$ be a polar space.
	A function $f: {\cal L} \to \overline{\RR}$ is submodular if and only if
	the Lov\'asz extension $\overline{f}: K({\cal L}) \to \overline{\RR}$ 
	is convex.	
\end{Thm}
It would be interesting to develop 
minimization algorithms based on this CAT(0) convexity.
Notice that submodularity and convexity 
are not equal in general, even if $K({\cal L})$ is CAT(0); 
consider a modular semilattice ${\cal L} = \{ 0, a, b, c, a'\}$
such that $a,b,a'$ are atoms, and $c$ covers $a,b$ and has no relation with $a'$.
%

\noindent
\subsubsection{$k$-submodular functions}\label{subsub:k-submo}
We discuss $k$-submodular functions of Huber and Kolmogorov~\cite{HK12} from our viewpoint.
Let $k,n$ be nonnegative integers with $k \geq 2$ and $n \geq 1$. 
Then ${{\cal S}_{k}}^n$ is a polar space.
Recall the operation $\sqcup$ (Example~\ref{ex:op-S_k^n}).
A function $f: {{\cal S}_{k}}^n \to \overline{\RR}$ 
is called {\em $k$-submodular}~\cite{HK12}
if  
\begin{equation}
f(p) + f(q) \geq f(p \wedge q) + f(p \sqcup q) \quad 
(p,q \in {{\cal S}_{k}}^n).
\end{equation}
By Theorem~\ref{thm:polar_submo} we have: 
\begin{Thm}\label{thm:k-sub}
	$f: {{\cal S}_{k}}^n \to \overline\RR$ is $k$-submodular 
	if and only if 
	$f$ is submodular on ${{\cal S}_{k}}^n$.
\end{Thm}

\begin{Rem}\label{rem:kl-submo}
	Submodular functions on polar space ${{\cal S}_{k,l}}^n$ 
	(Example~\ref{ex:S_k}) will be a next natural class to be investigated, 
	and  called {\em $(k,l)$-submodular}.
	In Section~\ref{sec:applications}, 
	we see that a $(2,k)$-submodular function arises 
	from the node-multiway cut problem.
\end{Rem}

\subsubsection{$\pmb \alpha$-bisubmodular functions}\label{subsub:alpha_bisub}
Here we study $\alpha$-bisubmodular functions
introduced by Huber, Krokhin, and Powell~\cite{HKP14}, 
and their generalization by Fujishige, Tanigawa, and Yoshida~\cite{FTY14}.
Let ${\cal S}_2 = \{-,0,+\}$.
In addition to $\sqcup$,
define a new binary operation $\sqcup_+$ on ${\cal S}_2$ by
\begin{equation}
x \sqcup_+ y
=  \left\{
\begin{array}{ll}
+ & {\rm if}\ \{x,y\} = \{-,+\}, \\
x \sqcup y & {\rm otherwise,}
\end{array} \right. \quad 
\end{equation}
and extend it to an operation on ${{\cal S}_2}^n$ 
componentwise.
Let ${\pmb \alpha} = (\alpha_1, \alpha_2,\ldots,\alpha_n)$ 
be an $n$-tuple of positives satisfying 
$0 < \alpha_1 \leq \alpha_2 \leq \cdots \leq \alpha_n \leq 1$.
For $i=0,1,2,\ldots,n$, define operation $\sqcup^i$ by
$
\sqcup^i := (\underbrace{\sqcup_+, \sqcup_+,\ldots, \sqcup_+}_{i}, \sqcup, \sqcup, \ldots, \sqcup).
$
A function $f: {{\cal S}_2}^n \to \overline{\RR}$ 
is called {\em $\pmb \alpha$-bisubmodular}~\cite{FTY14} if 
\begin{equation}\label{eqn:alpha_bisub}
f(p) + f(q) \geq f(p \wedge q) + \sum_{i=0}^{n} (\alpha_{i+1} - \alpha_i) f(p \sqcup^i q)  \quad (p,q \in {{\cal S}_2}^n),
\end{equation}
where we let $\alpha_0 := 0$ and $\alpha_{n+1} := 1$.
Notice that $(\alpha, \alpha,\ldots, \alpha)$-bisubmodularity  
coincides with $\alpha$-bisubmodularity in the sense of \cite{HKP14}.

Our framework captures ${\pmb \alpha}$-bisubmodularity 
by using valuations; see Section~\ref{subsub:valuation}.
For $i=1,2,\ldots,n$, define a valuation $v_i$ on ${\cal S}_2$ by
$v_i(0) := 0$, $v_i(+) := 1$, and $v_i(-) := \alpha_i$.
Define a valuation $v_{\pmb \alpha}$ on ${{\cal S}_2}^n$ by
\[
v_{\pmb \alpha}(x) := v_1(x_1) + v_2 (x_2) + \cdots + v_n(x_n) \quad (x = (x_1,x_2,\ldots,x_n) \in {{\cal S}_2}^n ).
\]
\begin{Thm}\label{thm:alpha_bisub}
	$f:{{\cal S}_2}^n \to \overline\RR$ is $\pmb \alpha$-bisubmodular if and only 
	if $f$ is submodular on modular semilattice ${{\cal S}_2}^n$ 
	with respect to the valuation $v_{\pmb \alpha}$.
\end{Thm}
Theorem~\ref{thm:alpha_bisub} can be
obtained from an explicit formula of the fractional join operation of ${{\cal S}_{2}}^n$.
For $i=0,1,2,\ldots,n$, define operations $\vee_L^i$ and $\vee_R^i$ by
\[
\vee_L^i :=  (\underbrace{\sqcup_+, \sqcup_+,\ldots, \sqcup_+}_{i}, \vee_L, \vee_L, \ldots, \vee_L), \ \vee_R^i :=  (\underbrace{\sqcup_+, \sqcup_+,\ldots, \sqcup_+}_{i}, \vee_R, \vee_R, \ldots, \vee_R).
\]
\begin{Prop}\label{prop:frac_join_S_n^2}
	The fractional join operation of ${{\cal S}_2}^n$ with respect to the valuation $v_{\pmb \alpha}$
	is equal to 
	\begin{equation}\label{eqn:frac_join_alpha}
	\sum_{i=0}^{n-1} \left( \frac{1}{1+ \alpha_{i}} - \frac{1}{1+ \alpha_{i+1}} \right) (\vee_L^i + \vee_R^i) + \frac{1 - \alpha_n}{1+ \alpha_{n}} \sqcup^n.
	\end{equation}
\end{Prop}

\section{L-convex functions on oriented modular graphs}\label{sec:L-convex}
In this section, we study L-convex functions 
on oriented modular graphs.
In Section~\ref{subsec:basics_L-convex}, 
we define L-convex functions and
establish some basic properties (closedness under nonnegative summation, local characterization, L-optimality criterion, steepest descent algorithm (SDA)).
A central result in this section is a sharp iteration bound 
of SDA (Theorem~\ref{thm:bound}).
Then we introduce L-extendable functions on swm-graphs via barycentric subdivision, and establish the persistency property (Theorem~\ref{thm:persistency}).
In Section~\ref{subsec:building_L-convex}, 
we study L-convex functions on 
(oriented modular graphs arising from) Euclidean buildings.
In a Euclidean building, the discrete midpoint operators are naturally defined.
We show in Theorem~\ref{thm:building} that L-convex functions are characterized by the discrete midpoint convexity as well as the convexity of the Lov\'asz extension.
Then we explain how our framework captures 
other discrete convex functions, such as 
L$^{\natural}$-convex functions, UJ-convex functions, strongly-tree submodular functions, and alternating L-convex functions.

\subsection{Basics}\label{subsec:basics_L-convex}
\subsubsection{Definition}
Let $\varGamma$ be an oriented modular graph.
We define L-convex functions on $\varGamma$, 
following basically~\cite{HH150ext}. 
Here we allow $\varGamma$ to be an infinite graph, 
and functions on  $\varGamma$ to take an infinite value,
where \cite{HH150ext} assumed that 
the domain is finite and function values are finite.
We utilize notions in Section~\ref{subsec:modular}, 
such as the relation $\sqsubseteq$, 
the principal ideal ${\cal I}_x$, 
the principal $\sqsubseteq$-ideal ${\cal I}_x'$, and the barycentric subdivision.

We first introduce a connectivity concept for subsets of vertices  in $\varGamma$.
A sequence $x = x_0,x_1,\ldots,x_m = y$ of vertices is called a {\em $\varDelta'$-path} if
$x_i \sqsubseteq x_{i+1}$ or $x_{i+1} \sqsubseteq x_{i}$ for $i=0,1,2,\ldots,m-1$.
A subset $X$ of vertices of $\varGamma$ is said to be 
{\em $\varDelta'$-connected} 
if every pair of distinct vertices in $X$ is connected by 
a $\varDelta'$-path in $X$.

Next we introduce the local structure at each vertex 
of $\varGamma$ via barycentric subdivision.
Let $\varGamma^* 
(= \{[x,y] \mid x,y \in \varGamma: x \sqsubseteq y\})$ be the barycentric subdivision. 
Recall that $\varGamma$ is isometrically 
embedded into $\varGamma^*$ by $x \mapsto \{x\}$ (Theorem~\ref{thm:G->G*}).
We regard $\varGamma \subseteq \varGamma^*$.
By Lemma~\ref{lem:well-oriented}, 
$\varGamma^*$ is well-oriented. 
For a vertex $x$ of $\varGamma$, 
the {\em neighborhood semilattice} ${\cal I}^*_x$ is defined by
\[
{\cal I}^*_x := {\cal I}_{\{x\}}(\varGamma^*)  = {\cal I}'_{\{x\}}(\varGamma^*).
\]
By Proposition~\ref{prop:ideal}, ${\cal I}^*_x$ is a (complemented) modular semilattice.

For a function $g: \varGamma \to \overline{\RR}$, define $g^*: \varGamma^* \to \overline{\RR}$ by
\begin{equation}
g^*([x,y]) := (g(x) + g(y))/2 \quad ([x,y] \in \varGamma^*).
\end{equation}
A function $g: \varGamma \to \overline{\RR}$ is called {\em L-convex} 
if $\dom g$ is $\varDelta'$-connected, 
and for every vertex $x$, the restriction of $g^*$ to ${\cal I}^*_x \subseteq \varGamma^*$ is submodular.

The class of L-convex functions includes all constant functions, 
and is closed under a direct sum; the $\varDelta'$-connectivity 
of the domain of direct sum follows from Lemma~\ref{lem:product_subdivision}~(2).
Here the closedness under nonnegative sum is nontrivial. 
\begin{Lem}\label{lem:sum_of_L}
	A nonnegative sum of two L-convex functions is L-convex.   
\end{Lem}
Two basic examples of L-convex functions are given.
The {\em indicator function} $[X]$ of a vertex subset $X$ 
is defined by $[X](x) = 0$ if $x \in X$ and  $[X](x) = \infty$ otherwise..
\begin{Lem}\label{lem:indicator}
	The indicator function of a $d_{\varGamma}$-convex set is L-convex on $\varGamma$.
\end{Lem}
\begin{Thm}[{\cite[Theorem 4.8]{HH150ext}}]\label{thm:d_is_L-convex}
	The distance function $d_{\varGamma}$ 
	is L-convex on $\varGamma \times \varGamma$.
\end{Thm}
%
%

L-convex functions are locally submodular 
in the following sense. 
\begin{Prop}[{\cite[Lemma 4.10]{HH150ext}}]\label{prop:locally-submo}
Let $g$ be an L-convex function on $\varGamma$. 
For every vertex $x$, the restrictions of $g$ to 
the principal $\sqsubseteq$-filter ${\cal F}'_x$ and  to the principal $\sqsubseteq$-ideal  ${\cal I}'_x$ are submodular.
\end{Prop}
The proofs of the above two in \cite{HH150ext} 
work for our non-finite setting.
We show that the converse of Proposition~\ref{prop:locally-submo} 
holds if $\varGamma$ is well-oriented.
\begin{Prop}\label{prop:L-convex=submodular}
	Suppose that $\varGamma$ is well-oriented.
	Then $g: \varGamma \to \overline\RR$ is L-convex 
	if and only if $\dom g$ is $\varDelta'$-connected, 
	and  for every vertex $x$, 
	the restrictions of $g$ to the principal $\sqsubseteq$-filter ${\cal F}'_x$ and to the principal $\sqsubseteq$-ideal ${\cal I}'_x$ 
	are submodular.
\end{Prop}

\subsubsection{L-optimality criterion and steepest descent algorithm}
We present a local-to-global optimality criterion for 
minimization of L-convex functions, 
which is a direct analogue of the L-optimality condition in DCA. 
Let $\varGamma$ be an oriented modular graph, and
let $g$ be an L-convex function on $\varGamma$. 
We assume that the following condition to cope with infiniteness:
\begin{description}
	\item[{\rm (F)}] $\varGamma$ is locally-finite or
	the image of $g$ is discrete in $\overline{\RR}$.\footnote{There is $\epsilon > 0$ 
		such that $[g(x) - \epsilon, g(x) + \epsilon] \cap g(\varGamma) = \{g(x)\} $ for every $x \in \varGamma$.}
\end{description} 
Under this condition (F), we have the following;
the finite case is in \cite{HH150ext}. 
\begin{Thm}\label{thm:L-optimality} 
	A vertex $x \in \dom g$ is a minimizer of $g$ over $\varGamma$
	if and only if $x$ is a minimizer of $g$ over the union of the principal $\sqsubseteq$-filter ${\cal F}'_x$ and  the principal $\sqsubseteq$-ideal ${\cal I}'_x$ of $x$.
\end{Thm}
Since $g$ is submodular on ${\cal F}'_x$ and on ${\cal I}'_x$, 
the optimality can be checked by submodular function minimization.
This leads us to the following descent algorithm, 
which is also a direct analogue of 
the {\em steepest descent algorithm} in DCA~\cite[Section 10.3.1]{MurotaBook}.
\begin{description}
	\item[Steepest Descent Algorithm (SDA)]
	\item[Input:] An L-convex function $g: \varGamma \to \overline\RR$ and a vertex $x$ in $\dom g$.
	\item[step 1:] Let $y$ be a minimizer of $g$ 
	over ${\cal F}'_{x} \cup {\cal I}'_{x}$.
	\item[step 2:] If $g(x) = g(y)$, then $x$ is minimizer, and stop.
	\item[step 3:] $x := y$, and go to step 1.
\end{description}
In many applications, the oriented modular graph in question is 
a product of small (locally-finite) oriented modular graphs 
$\varGamma_i$ $(i=1,2,\ldots, n)$, and
the L-convex function $g$ is a sum of L-convex functions $g_j$ 
of small number of variables.
In this case, the step 1 is reduced to VCSP for submodular functions, 
which can be solved efficiently by Theorem~\ref{thm:vscp}.

To obtain a complexity bound, 
we need to estimate the number of iterations of the algorithm. 
In the case of L$^{\natural}$-convex functions,
the number of iterations is bounded by 
the $l_{\infty}$-diameter of the effective domain~\cite{KS09}, 
and exactly equals a certain directed $l_{\infty}$-distance 
between the initial point and the minimizers~\cite{MurotaShioura14}.

We show that an analogous bound holds when $\varGamma$ is well-oriented.
Let $\varGamma^{\varDelta}$ be the thickening of $\varGamma$; see Section~\ref{subsub:thickening}.
By definition,
SDA yields a $\varDelta'$-path, which is also a $\varDelta$-path 
(a path in $\varGamma^{\varDelta}$).
Let ${\rm opt}(g)$ denote the set of minimizers of $g$.
Obviously 
the total number of the iterations 
is at least $d^{\varDelta}(x, {\rm opt}(g)) 
= \min_{y \in {\rm opt}(g)} d^{\varDelta} (x,y)$.
The next theorem says that this bound is tight; 
special cases are given in~\cite{HH14extendable,HH15node_multi}.
\begin{Thm}\label{thm:bound}
	The total number $N$ of the iterations 
	of SDA with initial vertex $x$ is at most 
	$d^{\varDelta}(x, {\rm opt}(g)) + 2$. 
	In addition, if $g(x) = \min_{y \in {\cal F}'_x} g(y)$ or $g(x)  = \min_{y \in {\cal I}'_x} g(y)$, 
	then $N$ is equal to $d^{\varDelta}(x, {\rm opt}(g))$.
\end{Thm}
%
%
%
In the case where $\varGamma =\varGamma_1 \times \varGamma_2 \times \cdots \times \varGamma_n$, 
the number of iterations is bound by the maximum of diameters of $\varGamma_i$ (by Lemma~\ref{lem:product_thickening}).
%
\begin{Rem}
The assumption that $\varGamma$ is well-oriented is not restrictive.
Indeed, consider the barycentric subdivision $\varGamma^*$, 
and the L-convex function $g^*$ instead of $g$ (see Proposition~\ref{prop:L_is_L-extendable}).
Apply SDA to $g^*$.
Since $\varGamma^*$ is well-oriented, Theorem~\ref{thm:bound} is applicable.
For a minimizer $u= [x,y]$ of $g^*$, both $x$ and $y$ are minimizers of $g$. 
%
\end{Rem}
\subsubsection{L-extendable functions}\label{subsub:L-extendable}
Next we introduce the concept of L-extendability, 
which aims at capturing 
well-behaved NP-hard problems having half-integral relaxations or {\em $k$-submodular relaxations}~\cite{GK13, IWY14}.
L-extendable functions
have been introduced in \cite{HH14extendable}
for the product of trees.

Let $H$ be an swm-graph. 
Then the barycentric subdivision $H^*$ is 
a well-oriented modular graph (Lemma~\ref{lem:well-oriented}).
A function $h : H \to \overline{\RR}$ is called {\em L-extendable}
if there exists an L-convex function $g$ on $H^*$ 
such that the restriction of $g$ to $H$ coincides with $h$.
Then $g$ is called an {\em L-convex relaxation} of $h$.
An L-convex relaxation $g$ of $h$ 
is said to be {\em exact} if the minimum value of $g$ is equal to that of $h$.
Recall that an orientable modular graph $\varGamma$ is also an swm-graph.
The class of L-extendable functions 
contains all L-convex functions with regard to all admissible orientations of $\varGamma$.
\begin{Prop}\label{prop:L_is_L-extendable}
	Let $\varGamma$ be an oriented modular graph.
	Any L-convex function $g: \varGamma \to {\overline\RR}$ is L-extendable, 
	and $g^*: \varGamma^* \to {\overline{\RR}}$ is an exact L-convex relaxation of $g$.
\end{Prop}
The next result, called {\em persistency}, 
says that there exists a minimizer of $h$ reasonably close to 
any minimizer of the relaxation $g$.
\begin{Thm}\label{thm:persistency}
Let $H$ be an swm-graph, $h: H \to \overline{\RR}$ an L-extendable function,
and $g:H^* \to \overline{\RR}$ an L-convex relaxation of $h$.
For any minimizer $x^*$ of $g$ (over $H^*$) 
there exists a minimizer of $h$ (over $H$) in $H \cap {\cal F}_{x^*}(H^*)$.
\end{Thm}
%
A canonical example of L-extendable functions
is the distance function, which is a direct consequence of 
Theorems~\ref{thm:G->G*} and \ref{thm:d_is_L-convex}.
\begin{Prop}\label{prop:d_is_L-extendable}
	Let $H$ be an swm-graph. 
	The distance function $d_H$ is L-extendable on $H \times H$, 
   and an L-convex relaxation is given by the distance function $d_{H^*}$ of $H^*$
\end{Prop}

\begin{Ex}\label{ex:k-subrel}
	A {\em $k$-submodular relaxation}~\cite{GK13,IWY14} of a function $h$ on $\{1,2,\ldots,k\}^n$
	is a $k$-submodular function $g$ on $\{0,1,2,\ldots,k\}^n \simeq {{\cal S}_k}^n$ 
	such that the restriction of $g$ to $\{1,2,\ldots,k\}^n$ is equal to $h$.
	Identify $\{1,2,\ldots,k\}^n$ with the $n$-product ${K_k}^n$ of a complete graphs ${K_k}$.
	Then the barycentric subdivision $({K_k}^n)^*$ is isomorphic to ${{\cal S}_k}^n$ (Example~\ref{ex:subdivision}). 
	L-convex functions and ($k$-)submodular functions are the same on $({K_k}^n)^* \simeq {{\cal S}_{k}}^n$.
	Thus L-extendable functions on ${K_k}^n$ are exactly those functions which admit $k$-submodular relaxations.
	See~\cite{HI15} for further study on $k$-submodular relaxation.
\end{Ex}

\subsection{L-convex functions on Euclidean building}\label{subsec:building_L-convex}
Here we consider L-convex functions on oriented modular graph 
$\varGamma(K)$ associated with a Euclidean building $K$; see Section~\ref{subsub:orthoscheme} for Euclidean building.
We first define the discrete midpoint operator.
Since $K$ is CAT(0) (Theorem~\ref{thm:building_CAT(0)}), 
any pair of points in $K$ can be joined by the unique geodesic (Proposition~\ref{prop:uniquely-geodesic}).
For two vertices $x,y$, 
there is an apartment $\varSigma$ containing them (by B1 in the definition).
Since the subcomplex $\varSigma$ 
is isometric (convex) in $K$~\cite[Theorem 11.16~(4)]{BuildingBook}, 
the unique geodesic $[x,y]$ belongs to $\varSigma$.
Notice that 
$\varSigma$ is isometric to Euclidean space $\RR^n$, 
and vertices $x,y$ are integral vectors in $\RR^n$.
The midpoint $(x+y)/2$ in the geodesic is a half-integral vector.
Consider the simplex of $\varSigma$ containing $(x+y)/2$ in its relative interior.
This simplex is the line segment between two integral vectors: one is obtained from $(x+y)/2$ by rounding each non-integral component to the nearest even integer, and other is obtained by rounding each non-integral component to the nearest odd integer.
The former is denoted by $\lceil (x + y)/2 \rceil$, 
and the latter is $\lfloor (x + y)/2\rfloor$.
Then L-convex functions 
are characterized as follows, where 
the property (3) is called the {\em discrete midpoint convexity}.
\begin{Thm}\label{thm:building}
	Let $K$ be a Euclidean building, and let $\varGamma(K)$ be the associated oriented modular graph.
	For a function $g: \varGamma(K) \to \overline\RR$, the following conditions are equivalent:
	\begin{description}
		\item[{\rm (1)}] $g$ is L-convex.
		\item[{\rm (2)}] Lov\'asz extension $\overline{g}: K \to \overline\RR$ is convex.
		\item[{\rm (3)}]  
		$g(x) + g(y) \geq g( \lfloor (x + y)/2 \rfloor) +  g(\lceil (x + y)/2 \rceil)$
		holds for any $x,y \in \varGamma(K)$.
	\end{description}
\end{Thm}

%
\begin{Ex}[UJ-convex function]\label{ex:UJ-convex}
Fujishige~\cite{Fujishige14} introduced 
a {\em UJ-convex function}, which is defined as a function $g$ on $\ZZ^n$ 
such that its Lov\'asz extension with respect to Union-Jack division $K(\check{\ZZ}^n)$
is convex. 
Here $K(\check{\ZZ}^n)$ itself is a building of a single apartment. 
By Theorem~\ref{thm:building} (1) $\Leftrightarrow$ (2), 
UJ-convex functions are the same as 
L-convex functions on  $\check{\ZZ}^n$ (in our sense).
\end{Ex}

\begin{Ex}[Alternating L-convex function]\label{ex:alt-L-convex}
	Let $T$ be an infinite tree without degree-one vertices (leaves).
	Regard $T$ as a bipartite graph.
	Let $B,W$ denote the color classes of $T$. 
	For $x,y \in T$, there is a unique pair $(u,v)$ of vertices 
	such that $d(x,y) = d(x,u) + d(u,v) + d(v,y)$, 
	$d(u,v) \leq 1$, and $d(x,u) = d(v,y)$.
	Define $(x \bullet y, x \circ y) := (u,v)$ if $(u,v) \in B \times W$, 
	$(v,u)$ if $(v,u) \in B \times W$, and $(u,u)$ if 
	$u = v$ (i.e., $u,v \in B$ or $u,v \in W$). 
	Consider the direct product $T^n$, and extend operations $\bullet$ and $\circ$ componentwise.	
   An {\em alternating L-convex function}~\cite{HH14extendable} is a function $g:T^n \to \overline\RR$ satisfying the inequality
   \[
   g(x) + g(y) \geq g(x \bullet y) + g(x \circ y) \quad    (x,y \in T^n).
   \]   
   Here $T^n$ becomes a Euclidean building with respect to the orientation: 
   $x \to y$ if $x \in B$ and $y \in W$. 
   This orientation is called the {\em zigzag orientation}.
   Apartments are given by $P_1 \times P_2 \times \cdots \times P_n$ 
   for all possible $n$ simple paths $P_1,P_2,\ldots,P_n$ of infinite length.
   Then points $(x \bullet y)$ and $(x \circ y)$ are equal to 
    $\lceil (x+y)/2 \rceil$ and $\lfloor (x+y)/2 \rfloor$, respectively, 
    when vertices in $B$ are associated with even integers.  
    Thus
    alternating L-convex functions coincide with 
    L-convex functions on Euclidean building $T^n$.
\end{Ex}
As seen in the examples above,
the product of zigzag oriented trees forms a Euclidean building.
We here consider a slightly general situation where orientations are arbitrary. 
Let $T_1,T_2,\ldots,T_n$ be trees, where each tree has an edge-orientation.
For two vertices $x,y$ in $T_i$, there uniquely exists a pair $(u,v)$ 
of vertices such that $d(x,y)= d(x,u) + d(u,v) + d(v, y)$,  $d(u,v) \leq 1$,
and $d(x,u) = d(v, y)$.
Define $(\lceil (x+ y)/2 \rceil, \lfloor (x+ y)/2 \rfloor) := (u,v)$ if $u \to v$, 
$(v,u)$ if $v \to u$, and $(u,u)$ if $u = v$.
Consider the product $\varGamma := T_1 \times T_2 \times \cdots \times T_n$, which
is an oriented modular graph. Extend operations 
$\lfloor (x+ y)/2 \rfloor$ and $\lceil (x+ y)/2 \rceil$ componentwise.
A variation of Theorem~\ref{thm:building} is the following. 
Since $\varGamma$ is not necessarily well-oriented, 
we consider $K'(\varGamma)$ (instead of $K(\varGamma)$);  
see Section~\ref{subsec:CAT(0)} for $K'(\varGamma)$.
\begin{Prop}\label{prop:oriented_tree}
	Let  $T_1,T_2,\ldots,T_n$ be oriented trees, 
	and let $\varGamma := T_1 \times T_2 \times \cdots \times T_n$,
	For a function $g: \varGamma \to \overline{\RR}$, the following conditions are equivalent:
	\begin{description}
		\item[{\rm (1)}] $g$ is L-convex.
		\item[{\rm (2)}] Lov\'asz extension 
		$\overline g: K'(\varGamma) \to \overline{\RR}$ is convex. 
		\item[{\rm (3)}]  $g(x) + g(y) \geq g( \lfloor (x + y)/2 \rfloor ) + g( \lceil (x + y)/2 \rceil)$ holds for any $x,y \in \varGamma$.
	%
	\end{description}
\end{Prop}

\begin{Ex}[L$^\natural$-convex function]\label{ex:Lnatural-convex}
	A linearly-oriented grid graph $\vec{\ZZ}^n$ is regarded as the product of directed paths (of infinite length).
	In this case, operations $\lceil \cdot \rceil$  and $\lfloor \cdot \rfloor$ on $\ZZ^n$
	coincide with the rounding up and down operators in (\ref{eqn:midpoint_original}), respectively.
Thus L$^\natural$-convex functions of \cite{MurotaBook} coincide with L-convex functions on $\vec{\ZZ}^n$.
\end{Ex}	
\begin{Ex}[Strongly tree-submodular function]\label{ex:tree-submo}
	Suppose that $T_1,T_2,\ldots, T_n$ are rooted trees, oriented from roots.
	In this case,  operations $\lceil \cdot \rceil$  and $\lfloor \cdot \rfloor$
	coincides with $\sqcup$ and $\sqcap$ in the sense of Kolmogorov~\cite{Kolmogorov11}.
	He introduced a {\em strongly-tree submodular function} as a function $g$ on 
	$\varGamma := T_1 \times T_2 \times \cdots \times T_n$ satisfying $g(x) + g(y) \geq g(x \sqcup y) + g(x \sqcap y)$ for $x,y \in \varGamma$.
	Thus strongly tree-submodular functions coincide with 
	L-convex functions on $\varGamma$ with respect to the rooted orientation.
\end{Ex}
\begin{Rem}
	In the case where $\varGamma$ is a Euclidean building or the product of oriented trees, thanks to the discrete midpoint convexity,
	Theorem~\ref{thm:L-optimality}
	holds without assumption (F). The proof is standard; 
	see the proof of \cite[Theorem 2.5]{HH14extendable}.
\end{Rem}

\section{Applications}\label{sec:applications}
In this section, we present applications 
of the results in previous sections to 
combinatorial optimization problems, such as multiflows, multiway cut, and related labeling problems.
We show that dual objective functions of several well-behaved multiflow problems~\cite{HH09JCTB,HH11folder,HH13tree_shaped,HH14bounded,HH15node_multi}
can be viewed as submodular/L-convex functions in suitable sense (Proposition~\ref{prop:multiflow_dual}).
This is a far-reaching generalization 
of a common knowledge in combinatorial optimization:
{\em the cut function, which is the dual objective of the max-flow problem, is submodular}.
We present an occurrence of L-extendable functions from the node-multiway cut problem  
(Example~\ref{ex:node_free_multi}).
Finally we apply the established iteration bound of SDA to
obtain strong polynomial time solvability of 
the 0-extension problems on orientable modular graphs (Theorem~\ref{thm:strong}), 
whereas the previous work \cite{HH150ext} showed only weak polynomiality.

\subsection{Multiflows}

An {\em undirected network} ${\cal N} = (V, E,c, S)$ consists of undirected graph $(V, E)$, 
an edge-capacity $c:E \to \RR_+$, and a specified set $S \subseteq V$ of nodes, called {\em terminals}.
Suppose that $V = \{1,2,\ldots,n\}$.
An {\em $S$-path} is a path connecting distinct terminals in $S$.
A {\em multiflow} is a pair $({\cal P}, \lambda)$ of a set ${\cal P}$ of $S$-paths and 
a nonnegative flow-value function $\lambda: {\cal P} \to \RR_+$ satisfying the capacity constraint:
$f(e) := \sum \{  \lambda(P) \mid P \in {\cal P}: \mbox{$P$ contains $e$} \} \leq c(e)$ for $e \in E$.

We first consider multiflow maximization problems, 
where the value of a multiflow is specified by a terminal weight.
Let $\mu$ be a nonnegative rational-valued function defined on the set of all distinct unordered pairs of terminals in $S$.
The {\em $\mu$-value} $\mu(f)$ of a multiflow $f$ is defined by
$
\mu(f) := \sum \mu(s,t) \lambda(P)$, where the sum 
is taken over all distinct $s,t \in S$ and all $(s,t)$-paths $P$.
Namely $\mu(s,t)$ is interpreted as the value of a unit $(s,t)$-flow.
The $\mu$-weighted maximum multiflow problem asks 
to find a multiflow of the maximum $\mu$-value.
The weight $\mu$ defines the class of problems.
For example, if $S = \{s,t\}$ and $\mu(s,t) = 1$, 
then the problem is the maximum flow problem.
There are several combinatorial min-max relations 
for special weights $\mu$, 
generalizing Ford-Fulkerson's max-flow min-cut theorem.
Examples include Hu's max-biflow min-cut theorem 
for the maximum 2-commodity flow problem ($\mu(s,t) = \mu(s',t') = 1$ and zero for other terminal pairs) 
and  the Lov\'asz-Cherkassky theorem 
for the maximum free multiflow problem ($\mu(s,t) = 1$ for $s,t \in S$).
See e.g., \cite{Kar89} and \cite[Section 73.3b]{SchrijverBook}
for further generalizations.
Continuing a pioneering work~\cite{Kar98a} by Karzanov, 
the author~\cite{HH09JCTB, HH11folder, HH14bounded} 
has developed a unified theory for these combinatorial dualities, 
which we describe below.

A {\em frame}~\cite{Kar98a} is an orientable modular graph 
without any isometric cycle of length greater than $4$.
An {\em oriented frame} is 
a frame endowed with an admissible orientation.
An oriented frame is exactly an oriented modular graph $\varGamma$
such that each pair $x,y$ of vertices with $x \sqsubseteq y$ has distance at most $2$, or equivalently, such that 
the corresponding orthoscheme complex $K'(\varGamma)$ is $2$-dimensional.
Such an orthoscheme complex is known as a {\em folder complex}~\cite{Chepoi00, HH11folder}.
Consider an oriented frame $\varGamma$.
A triple $(p,q,r)$ of distinct vertices is called a {\em triangle} 
if $p \leftarrow q \leftarrow r$ and $p \sqsubseteq r$ 
or $r \leftarrow q \leftarrow p$ and $r \sqsubseteq p$.
In an oriented frame $\varGamma$,
a vertex subset $X$ is said to be {\em normal}~\cite{HH11folder} if it satisfies:
\begin{description}
	\item[N1:] $X$ is $\varDelta'$-connected. 
	\item[N2:] For each triangle $(p,q,r)$, $\{p,q\} \subseteq X$ if and only if $\{q,r\} \subseteq X$.
	\item[N3:] For any distinct triangles $(p,q,r)$, $(p,q,r')$ sharing an edge $pq$, 
	if $\{p,r,r'\} \subseteq X$ then $q \in X$.
\end{description}

An {\em embedding} ${\cal E}$ of a terminal weight $\mu$ 
for ${\cal N} = (V, E, c, S)$ 
is a pair $(\varGamma, \{F_s\}_{s \in S})$ of an oriented frame $\varGamma$
and a family $\{F_s\}_{s \in S}$ of normal sets $F_s$ indexed by $s \in S$ 
such that
\begin{equation*}
\mu(s,t) = \min_{x \in F_s, y \in F_t} d_{\varGamma} (x,y) \quad (s,t \in S, s \neq t).
\end{equation*}
Define $\omega_{{\cal N}, {\cal E}}: \varGamma^n \to \overline{\RR}$ by
\[
x \mapsto \sum_{s \in S} [F_s](x_s) + \sum_{ij \in E} c(ij) d_{\varGamma}(x_i, x_j),
\] 
where $[F_s]$ is the indicator function of $F_s$.

\begin{Thm}[\cite{HH11folder}]\label{thm:embedding}
	Suppose that $\mu$ has an embedding ${\cal E} = (\varGamma, \{F_s\}_{s \in S})$.
	Then the maximum $\mu$-value of a multiflow in ${\cal N}$ is equal to the minimum of 
	$\omega_{{\cal N}, {\cal E}}(x)$ over all $x \in \varGamma^n$.
\end{Thm}
We verify that the dual objective functions are actually L-convex. 
\begin{Prop}\label{prop:multiflow_dual}
	Suppose that $\mu$ has an embedding ${\cal E} = (\varGamma, \{F_s\}_{s \in S})$.
	Then $\omega_{{\cal N}, {\cal E}}$ is L-convex on $\varGamma^n$.
\end{Prop}
\begin{proof}
	By Lemma~\ref{lem:sum_of_L} and Proposition~\ref{thm:d_is_L-convex}, 
	it suffices to show that the indicator function $[F]$ of any normal set $F$ is L-convex.
	Consider $\varGamma^*$. 
	Observe that $\dom [F]^*$ is normal in well-oriented $\varGamma^*$.
	Thus we can assume that $\varGamma$ is well-oriented.
	Take an arbitrary vertex $x$ in $F$.
	By Proposition~\ref{prop:L-convex=submodular} and N1, 
	it suffices to show that $[F]$ is submodular on ${\cal F}_x$ (and ${\cal I}_{x}$).
	Take $p,q \in {\cal F}_x \cap F$.  
	We prove the claim by showing $p \wedge q \in F$ and ${\cal E}(p,q) \subseteq F$.
	We can assume that $p$ and $q$ are incomparable.
	Notice that the rank of ${\cal F}_x$ is at most $2$.
	Suppose that both $p$ and $q$ have rank $1$.
	Then $p \wedge q = x \in F$. 
	If $p \vee q$ exists (i.e., ${\cal E}(p,q) = \{ p \vee q \}$), 
	then $(x, p, p\vee q)$ is a triangle, 
	and by N2 with $x, p \in F$, 
	we have $p \vee q  \in F$.
    If $p \vee q$ does not exist,   
	then ${\cal E}(p,q) = \{p,q\} \subseteq F$.
	Suppose that $p$ has rank $2$ and $q$ has rank $1$. 
	Then $p \wedge q = x \in F$, 
	and $u \in {\cal E}(p,q) \setminus \{p,q\}$ (if exists)
	is the join of $q$ and the rank 1-element $p' \in [x,p]$; 
	apply N2 to triangle $(x,q,u)$ to obtain $u \in F$.
	Suppose that both $p$ and $q$ have rank $2$.
	Then ${\cal E}(p,q) = \{p,q\} \subseteq F$.
	Suppose that $p \wedge q \neq x (\in F)$.
	Then $(x, p \wedge q, p)$ and $(x, p \wedge q, q)$ are triangle with $x,p,q \in F$.
	By N3, we obtain $p \wedge q \in F$, as required. 
\end{proof}
Papers \cite{HH09JCTB, HH11folder, HH14bounded} contain various examples 
of multiflow combinatorial dualities and constructions of embeddings.
Now they all fall into our theory of L-convexity. 
Other examples of combinatorial multiflow dualities are given.
\begin{Ex}[{Minimum-cost node-demand multiflow problem}]
Suppose that the network ${\cal N}$ has an edge-cost $a: E\to \ZZ_+$ and a node-demand $r:S \to \RR_+$ on terminal set $S$.
A multiflow $f = ({\cal P},\lambda)$ 
is said to be {\em feasible} if for each $s \in S$, 
the sum of $\lambda(P)$ over all paths $P \in {\cal P}$ connecting $s$
is at least $r(s)$.
The problem is to find a feasible multiflow $f$ of the minimum cost $\sum_{e \in E}a(e)f(e)$.
This problem was introduced by 
Fukunaga~\cite{Fukunaga14} in connection with a class of network design problems.

The dual of this problem can be formulated as an optimization over 
the product of {\em subdivided stars}, which is constructed as follows.
For $s \in S$
let $\varGamma_s$ be a path of infinite length and one end vertex $v_s$.
Consider the disjoint union  $\bigcup_{s \in S} \varGamma_s$ 
and identify all $v_s$ into one vertex $0$.
The resulting graph is denoted by $\varGamma$, and the edge-length is defined as $1/2$ uniformly.

Then the minimum value of the problem is equal to the maximum of 
\[
 \sum_{s \in S} r(s) d(x_s, 0)  - [\varGamma_s](x_s)  - \sum_{ij \in E} c(ij) \max \{ d(x_i,x_j) - a(ij), 0 \}
\] 
over all $x = (x_1,x_2,\ldots,x_n) \in \varGamma^n$~\cite{HH14extendable}.
The negative of the dual objective is an L-convex function on $\varGamma^n$ if $\varGamma$ 
is endowed with the zigzag orientation.
In \cite{HH14extendable}, SDA is applied to this objective function, where
each local problem reduces to minimizing a network-representable $k$-submodular function, and 
is efficiently solved by minimum cut computation~\cite{IWY14}.
Combined with domain scaling technique,
we obtain the first combinatorial polynomial time algorithm 
to find a minimum-cost feasible multiflow. 
\end{Ex}
\begin{Ex}[{Node-capacitated free multiflow problem}]\label{ex:node_free_multi}
	Suppose that the network ${\cal N}$ has a node-capacity $b: V \setminus S \to \RR_+$ instead of an edge-capacity $c$, where
	a multiflow $f = ({\cal P}, \lambda)$ should satisfy the node-capacity constraint:
	$
	\sum \{ \lambda(P) \mid  P \in {\cal P}: \mbox{$P$ contains node $i$} \} \leq b(i)$ 
	for $i \in V \setminus S$.
	The node-capacitated free multiflow problem asks to a find a multiflow $f = ({\cal P},\lambda)$ of 
	the maximum total flow-value $\sum_{P \in {\cal P}} \lambda(P)$.
	This problem was considered by Garg, Vazirani, and  Yannakakis~\cite{GVY04} as the dual of an LP-relaxation of the node-multiway cut problem; 
	see Section~\ref{subsec:0ext} below. 
	They showed the (dual) half-integrality, and 
	designed a 2-approximation algorithm for node-multiway cut; see \cite[Section 19.3]{Vazirani}. 
	
	We here reformulate this dual half-integrality according to \cite{HH13tree_shaped,HH15node_multi}.
	Consider a star $\varGamma$ with center $v_0$ 
	and leaf set $\{v_s \mid s \in S\}$ indexed by $S = \{1,2,\ldots,k\}$, 
	where the edge-length is defined as $1/2$ uniformly.
	Consider further the subdivision $\varGamma^*$, 
	where the midpoint between $v_0$ and $v_s$ is denoted by $\bar v_s$.
	Consider points in $\varGamma^* \times \{0,1/4,1/2\}$: 
	\[
	(v_0, 0), (v_0, 1/2), (v_{1}, 0), (v_{2}, 0),\ldots, (v_{k},0),  (\bar v_{1},1/4),  (\bar v_{2},1/4), \ldots, (\bar v_{k},1/4).  
	\]
	A partial order on these points is given  
	by $(v_s,0) \rightarrow (\bar v_s,1/4) \leftarrow (v_0,1/2)$ 
	and $(\bar v_s,1/4) \rightarrow (v_0,0) $ for $s \in S$.
	Let ${\cal S}_{2,k}^+$ denote the resulting modular semilattice, which is a subsemilattice of polar space ${\cal S}_{2,k}$ (Example~\ref{ex:S_kl}).

	The maximum flow-value of a multiflow is equal to the minimum of
	$\sum_{i \in V \setminus S} 2b(i) r_i$ over all 
	$(p_i,r_i) \in {\cal S}_{2,k}^+$ for $i \in V$
	satisfying 
	\begin{equation}\label{eqn:ri+rj}
		r_i + r_j \geq d_{\varGamma^*}(p_i, p_j)  \quad (ij \in E).
	\end{equation}
	This dual objective is viewed as a $(2,k)$-submodular function 
	(Remark~\ref{rem:kl-submo}). 
	An equivalent statement of this fact 
	was presented by Yoichi Iwata at the SAOR seminar in November 9, 2013. 
	
	In \cite{HH15node_multi}, we showed that
	this dual objective is further perturbed into  
	an L-convex function on a Euclidean building so to as each local problem reduces to an easy submodular flow problem.
    Then our SDA yields
	the first combinatorial strongly polynomial time algorithm for 
	the maximum node-capacitated free multiflow problem, 
	which in turn implies the first combinatorial strongly polynomial time
	implementation of Garg-Vazirani-Yannakakis algorithm
	for the node-multiway cut problem.
\end{Ex}

\subsection{Multiway cut and 0-extension}\label{subsec:0ext}

For an edge-capacitated network ${\cal N} = (V, E, c, S)$, 
an {\em (edge-)multiway cut}  is a set $F$ of edges such that 
every $S$-path meets $F$.
The multiway cut problem is to find a multiway cut 
$F$ of the minimum capacity $\sum_{e \in F}c(e)$.
As mentioned in \cite{GK13,IWY14}, 
the multiway cut problem has a natural $k$-submodular relaxation 
(see Example~\ref{ex:k-subrel}) 
that is the dual to the maximum free multiflow problem; see \cite[Example 2.17]{HH14extendable}.

An analogous relation holds in the node-capacitated setting.
Suppose that $b$ is a node-capacity function on $V \setminus S$.
A {\em node-multiway cut} is a subset $C$ of nodes
such that every $S$-path meets $C$.
The node-multiway cut problem is 
to find a node-multiway cut $C$ of the minimum capacity $b(C) := \sum_{i \in V\setminus S} b(i)$.
This problem can be viewed as  
an L-extendable function minimization on ${{\cal S}_{2,k}}^n$.
Recall the notation in Example~\ref{ex:node_free_multi}.
For a node-multiway cut $C$, define $(p,r) = ( (p_i,r_i): i \in V) \in ({\cal S}_{2,k}^+)^n$ 
by $(p_i,r_i) := (v_0, 1/2)$ if $i \in C$,
$(p_i,r_i) := (v_s, 0)$ if $i$ and terminal 
$s$ belong to the same component in the network obtained by deleting $C$,
and $(p_i,r_i) :=(v_{t},0)$ 
for an arbitrary fixed $t \in S$ otherwise.
Then $(p,r)$ satisfies (\ref{eqn:ri+rj}) and $\sum_{i} 2 b(i) r_i = b(C)$,  and
is maximal in ${{\cal S}_{2,k}}^n$.
Conversely, 
from a maximal $(p,r) \in ({\cal S}_{2,k}^+)^n$ satisfying (\ref{eqn:ri+rj}),
we obtain a node-multiway cut $C := \{i \mid (p_i, r_i) = (v_0,1/2)\}$ with $\sum_{i} 2 b(i) r_i = b(C)$.
Thus the objective function of the node-multiway cut problem is 
the restriction of the L-convex function on $({\cal S}_{2,k})^n$ in Example~\ref{ex:node_free_multi} 
to $(K_{2,k})^n$.

We next discuss a generalization of the edge-multiway cut problem.
The {\em minimum 0-extension problem} is: 
Given an input $I$ consisting of 
a number $n$ of variables, undirected graph $\varGamma$, 
nonnegative weights $b_{iv}$ $(i \in \{1,2,\ldots,n\}, v \in \varGamma)$ and $c_{ij}$ $(1 \leq i < j \leq n)$, 
find $x = (x_1,x_2,\ldots, x_n) \in \varGamma^n$ that minimizes the function 
$D_{I}: \varGamma^n \to \RR_+$ defined by
\[
x \mapsto \sum_{i=1}^n \sum_{v \in \varGamma} b_{iv} d_{\varGamma}(x_i, v) + \sum_{1 \leq i < j \leq n} c_{ij} d_{\varGamma}(x_i, x_j). 
\] 
Observe that the multiway cut problem corresponds to $\varGamma = K_k$.
The following complexity dichotomy theorem was
the starting point of our theory.
\begin{Thm}\label{thm:dichotomy}
	\begin{description}
		\item[{\rm (1)}] If $\varGamma$ is orientable modular, 
		then minimum 0-extension problem can be solved in polynomial time~{\rm \cite{HH150ext}}.
		\item[{\rm (2)}] If $\varGamma$ is not orientable modular, 
		then minimum 0-extension problem is NP-hard~{\rm \cite{Kar98a}}.
	\end{description}
\end{Thm}
The polynomial solvability is based on the VCSP-tractability of 
submodular functions and SDA of L-convex functions, 
applied to the following fact: 
\begin{Prop}[\cite{HH150ext}]\label{prop:D(I)}
	If $\varGamma$ is oriented modular, 
	then $D_{I}$ is L-convex on $\varGamma^n$.
\end{Prop}
The algorithm in~\cite{HH150ext} is based on SDA with capacity scaling, and is weakly polynomial.
By using the $l_{\infty}$-bound (Theorem~\ref{thm:bound}), 
we show the strongly polynomial time solvability.
\begin{Thm}\label{thm:strong}
	The minimum 0-extension problem on orientable modular graph $\varGamma$
	can be solved in strongly polynomial time.
\end{Thm}
\begin{proof}
	Consider the barycentric subdivision $\varGamma^*$, and the minimum 0-extension problem for instance 
	$I^* = (n,\varGamma^*,\{b_{iv}\}, \{c_{ij}\})$, 
	where $b_{iv} := 0$ for $v \in \varGamma^* \setminus \varGamma$. 
	Notice that $\varGamma^*$ has $O(|\varGamma|^2)$ vertices by Proposition~\ref{lem:om_Bgated}, and can be constructed in time polynomial of $|\varGamma|$; see~\cite[Lemma 3.7]{CCHO14}.
	By Proposition~\ref{prop:D(I)}, the objective function $D_{I^*}$ is L-convex.
	By Theorem~\ref{thm:G->G*}, it holds
	$
	D_{I^*} (u) = 
	(D_{I} (x) + D_{I}(y))/2,  
	$
	where $u = ([x_1,y_1],[x_2,y_2],\ldots,[x_n,y_n])$,
	$x = (x_1,x_2,\ldots,x_n)$, and $y = (y_1,y_2,\ldots,y_n)$.
	This means that $D_{I^*}$ is an exact L-convex relaxation　of $D_I$.
	Therefore, from an optimal solution of the relaxation, 
	we can obtain an optimal solution of the original problem.
	
	Apply SDA to solve $I^*$, 
	where each local problem is a submodular VCSP, and
	can be solved in strongly polynomial time by Theorem~\ref{thm:vscp}. 
	Now $\varGamma^*$ is well-oriented. 
	By Theorem~\ref{thm:building} and Lemma~\ref{lem:product_thickening}, 
	the number of iterations is bounded by
	the diameter of $\varGamma^* ( \leq |\varGamma|^2)$. 
	Thus the whole time complexity is polynomial in $n$ and $|\varGamma|$.
\end{proof}

\begin{Rem}
	The minimum 0-extension problem on an swm-graph $G$ is viewed 
	as an L-extendable function minimization on $G^n$. 
	Indeed, an L-convex relaxation is obtained by relaxing $G$ to $G^*$ 
	(Proposition~\ref{prop:d_is_L-extendable}).
	The relaxed problem is polynomially solvable, 
	and is a kind of a half-integral relaxation.
	In \cite[Section 6.9]{CCHO14}, we designed a 2-approximation rounding scheme based on this relaxation.
	This generalizes 
	the classical 2-approximation algorithm~\cite[Algorithm 4.3]{Vazirani} for edge-multiway cut.
\end{Rem}

\section{Proofs}\label{sec:proofs}

We first note structural properties of interval $I(p,q)$ 
of a modular semilattice ${\cal L}$, and behavior of submodular functions on $I(p,q)$. 
\begin{Lem}[{\cite[Lemmas 3.11 and 3.12]{HH150ext}}]\label{lem:E(p,q)}
	Let $u,u' \in {\cal E}(p,q)$.  
	\begin{description}
		\item[{\rm (1)}] If $r(u \wedge p) \geq r(u' \wedge p)$, 
		then $u \wedge p \succeq u' \wedge p$ and $u \wedge u' = (u' \wedge p) \vee (u \wedge q)$.
		\item[{\rm (2)}] $p \vee_L u = p \vee_L q$.
	\end{description}	
\end{Lem}
See also Figure~\ref{fig:ConvIpq}.
The following is a slightly sharper version of \cite[Lemma 4.2]{HH150ext}.
\begin{Lem}\label{lem:f(v_i)<f(v_i+1)}
	Let $f: {\cal L} \to \overline\RR$ be a submodular function.
	If $f(p) > f(q)$, 
	then there exists a sequence $(p = u_0,u_1,\ldots,u_k = q)$ 
	in $I(p,q)$ such that $f(p) > f(u_i)$ for $i > 0$
	and $u_i$ and $u_{i+1}$ are comparable for $i \geq 0$.
	In addition, if $f(p \wedge q) > f(q)$, 
	then $u_1 \succeq  p \wedge (p \vee_R q)$.
\end{Lem}
\begin{proof}[Sketch of proof]
	Suppose $f(p) > f(q)$. By submodular inequality~(\ref{eqn:submodular}), 
	there is $u \in \{p \wedge q \} \cup ( {\cal E}(p,q) \setminus \{p,q\})$ with $f(u) < f(p)$.
	Apply an inductive argument to $p,u$ and to $u,q$, to obtain such a sequence, as in the proof of {\cite[Lemma 4.2]{HH150ext}}.
	Suppose further that $f(p \wedge q) > f(q)$.
	We can assume that $p \wedge (p \vee_R q) \succ p \wedge q$.
	Then $p \vee_R q \succ q$, and $[C(q;p,q)] = 0$.
	 By submodular inequality~(\ref{eqn:submodular}) 
	 with $f(q) < f(p \wedge q)$, 
	 it holds $f(p) > \sum_{u \in {\cal E}(p,q)} [C(u;p,q)] f(u)$.
	 Thus there is $u \in {\cal E}(p,q) \setminus \{p,q\}$ with $f(u) < f(p)$.
	 By Lemma~\ref{lem:E(p,q)}~(1), 
	 $p \wedge u \succeq p \wedge (p \vee_R q)$.
	 By applying the inductive argument to $p,u$ and to $u,q$ as above, 
	 we have the latter part.
\end{proof}

\subsection{Section~\ref{sec:submodular}}
\subsubsection{Proof of Theorem~\ref{thm:frac_join_of_polar_space}}
We first show the claim 
for the case where ${\cal L} = {{\cal S}_2}^n = \{-1,0,1\}^n$.
In this case, $p \vee_L q$ is obtained from $p$ 
by replacing $p_i$ by $q_i$ for each $i$ with $p_i = 0 \neq q_i$ (Example~\ref{ex:op-S_k^n}). 
Therefore $p \vee_L q$ and $p \vee_R q$ 
have the same rank, 
which is equal to the number $N$ of indices $i$ 
with $p_i \neq 0$ or $q_i \neq 0$. 
Each $u \in I(p,q)$ must satisfy 
$u_i = 0$ for each index $i$ with $p_i = q_i = 0$, 
and hence $r(u; p,q)$ belongs to $\{(x,y) \in \RR^2_+ \mid x+ y \leq N\}$.
Consequently, 
it must hold that  $\Conv I(p,q) = \{(x,y) \in \RR^2_+ \mid x+ y \leq N, x \leq r(p), y \leq r(q)\}$, 
${\cal E}(p,q) = \{p,q,p \vee_L q, p \wedge_R q\}$, 
and the fractional join of $p,q$ is equal to $(1/2)p \vee_L q + (1/2) p \vee_R q$.

Next we consider the general case.
Let $p,q \in {\cal L}$. 
Consider a polar frame ${\cal F}$ 
containing chains $p \wedge q, p$ and $p \wedge q, q$. 
We show that $\Conv I(p,q)$ is equal to that considered 
in polar frame ${\cal F}$.
For $u = p' \vee q' \in I(p,q)$ with 
$p' \in [p \wedge q, p]$ and $q' \in [p \wedge q, q]$,
consider a polar frame ${\cal F}'$ 
containing chains $p \wedge q, p', p$ and $p \wedge q, q', q$. 
Then $u = p' \vee q'$ must belong to ${\cal F'}$.
Indeed, consider a polar frame ${\cal F}''$ containing 
$p \wedge q, p',u$ and $p \wedge q, q', u$, 
and consider isomorphism ${\cal F}'' \to {\cal F}'$ 
fixing $p \wedge q, p',q'$.
The image of $u$ must be the join of $p',q'$ and equal to $u$.
Now consider isomorphism $\phi: {\cal F}' \to {\cal F}$ 
fixing $p \wedge q, p,q$, and consider images $\phi(p'), \phi(q')$ and $\phi(u) = \phi(p') \vee \phi(q')$.
Then $r(\phi(p')) = r(p')$ and $r(\phi(q')) = r(q')$ must hold.
Thus the point $r(u;p,q)$ 
belongs to $\Conv I(p,q)$ considered in ${\cal F}$.
Necessarily the left and right joins are equal to the left and right joins in ${\cal F}$, respectively. 
Thus the fractional join is equal to $(1/2)p \vee_L q + (1/2) p \vee_R q$. 
We remark that this argument implies
the following rank equality for polar space.
\begin{equation}\label{eqn:rank_eq}
r(p \vee_L q) = r(p \vee_R q) \quad (p,q \in {\cal L}).
\end{equation}

\subsubsection{Proof of Theorem~\ref{thm:polar_submo}}
For three binary operations $\circ, \circ', \circ''$ on ${\cal L}$, 
let $\circ' \circ \circ''$ denote
the operation defined by $(p,q) \mapsto (p \circ' q) \circ (p \circ'' q)$.
Define projection operations $L$ by $(p,q) \mapsto p$ and $R$ by $(p,q) \mapsto q$. 
Define operations $\wedge_L := L \wedge \vee_R$ and $\wedge_R := R \wedge \vee_L$. 
\begin{Lem}\label{lem:1-6}
	\begin{description}
		\item[{\rm (1)}] $p \sqcup q = p \vee q = p \vee_L q = p \vee_R q $ if $p \vee q$ exists.
		\item[{\rm (2)}] $\vee_L \vee_L \vee_R =  \vee_L $ and $\vee_L \vee_R \vee_R = \vee_R$.
		\item[{\rm (3)}] $\wedge_L \sqcup \wedge_R  = \wedge_L  \vee  \wedge_R = \sqcup$.
		\item[{\rm (4)}] $\wedge_L \wedge \wedge_R = \wedge$.
		\item[{\rm (5)}] $L \sqcup \sqcup = L \vee \sqcup = \vee_L$
		and $R \sqcup \sqcup = R \vee \sqcup  = \vee_R$.
		\item[{\rm (6)}] $L \wedge \sqcup  = \wedge_L$ and $R \wedge \sqcup = \wedge_R$.
	\end{description}
\end{Lem}
\begin{proof}
	(1) is obvious from the definition.
	(2) follows from Lemma~\ref{lem:E(p,q)}~(2) (with $p \vee_L q = u_0$ or $u_1$).
	(3) follows from (1) and Lemma~\ref{lem:E(p,q)}~(2).
	(4) and (6) follow from the definition of $\wedge_L$ and $\wedge_R$.
	(5) follows from 
		$p \sqcup (p \sqcup q) =  p \vee (p \sqcup q) = p \vee (p \wedge_L q) \vee (p \wedge_R q) =  p \vee (p \wedge_R q) = ((p \vee_L q) \wedge p) \vee ((p \vee_L q) \wedge q)  = p \vee_L q$,
	where we use (1) for the first equality, (3) for the second, and  
	the unique representation of elements in $I(p,q)$ for the third.
\end{proof}

We are ready to prove Theorem~\ref{thm:polar_submo}.
Fix arbitrary $p,q \in {\cal L}$. 
For an operation~$\circ$, $f(p \circ q)$ is simply denoted by $f(\circ)$.

(1) $\Rightarrow$ (2). Suppose that $f$ is submodular.
By applying Lemma~\ref{lem:1-6}~(2) 
to \eqref{eqn:submo_veeLveeR} for $(p \vee_L q, p \vee_R q)$, we obtain
$
f(\vee_L) + f(\vee_R) \geq f(\sqcup) + \left\{ f(\vee_L) + f(\vee_R)\right\}/2.
$
In particular we have
$
f(\vee_L) + f(\vee_R) \geq 2 f(\sqcup)$.
Thus we have $
f(L) + f(R)  \geq  f(\wedge) + \left\{ f(\vee_L) + f(\vee_R)\right\}/2  \geq  f(\wedge) + f(\sqcup)$.
Hence $f$ satisfies the inequality in (2).

(2) $\Rightarrow$ (1).
Suppose that $f$ satisfies the inequality in (2).
By applying Lemma~\ref{lem:1-6} to this inequality, we have
$f(\wedge_L) + f(\wedge_R) \geq f(\wedge) + f(\sqcup)$,
$f(L) + f(\sqcup) \geq f(\wedge_L) + f(\vee_L)$, and
$f(R) + f(\sqcup) \geq f(\wedge_R) + f(\vee_R)$.
Adding them, we have $f(L) + f(R) + f(\sqcup) \geq f(\wedge) + f(\vee_L) + f(\vee_R)$.
By applying Lemma~\ref{lem:1-6}~(3) to the inequality 
for $(p \vee_L q, p \vee_R q)$, 
we have $f(\vee_L) + f(\vee_R) \geq f(\sqcup) + f(\sqcup).$
Adding them (with multiplying $1/2$ to the second), 
we obtain \eqref{eqn:submo_veeLveeR} as required.

(3) $\Leftrightarrow$ (2). 
As seen in the proof of Theorem~\ref{thm:frac_join_of_polar_space}, 
for each polar frame ${\cal F}$, 
the left and right join in ${\cal F}$ is equal to that in ${\cal L}$.
Consequently, the pseudo join in ${\cal F}$ is equal to that in ${\cal L}$.
Now the inequality in (2) is nothing 
but the bisubmodularity inequality under ${{\cal S}_2}^n \simeq \{-1,0,1\}^n$.
From this, we see the equivalence (3) $\Leftrightarrow$ (2).

\subsubsection{Proof of Proposition~\ref{prop:frac_join_S_n^2}}

We start with  some notation.
For $0 \leq a \leq b \leq \infty$, 
let ${\rm Cone} (a,b)$ 
be the convex cone in $\RR_+^2$ defined by
$
{\rm Cone} (a,b) := \{(x,y) \in \RR^2_+ \mid a x \leq y, x \geq  y/b \},
$
where we let $1/\infty := 0$.
Then
$
[{\rm Cone}(a,b)] = b/(1+b) - a/(1+a) = 1/(1+a) - 1/(1+b).
$

We next determine the fractional join operation on ${\cal S}_{2} = \{0,+,-\}$ 
with respect to valuation $v_i$.
By $I(+,-) = \{0,+,-\}$,
$\Conv I(+,-)$ is the triangle with vertices 
$v_i(0;+,-) = (0,0),  v_i(+;+,-) = (1,0), v_i(-;+,-) = (0, \alpha_i)$, 
and $\Conv I (-,+)$ is the triangle with vertices 
$v_i(0;+,-) =(0,0), v_i(-;-,+) = (\alpha_i,0), v_i(+;-,+) = (0, 1)$.
In particular, ${\cal E}(+,-) = {\cal E}(-,+) = \{-,+ \}$, 
$C(+; +, -) = {\rm Cone}(0, 1/\alpha_i)$, 
$C(-; +, -) = {\rm Cone}(1/\alpha_i, \infty)$, 
$C(-; -, +) = {\rm Cone}(0, \alpha_i)$, and
$C(+; -, +) = {\rm Cone}(\alpha_i, \infty)$.
If the join $x \vee y$ exists, 
then $C(x \vee y; x,y) = \RR^2_+$, 
and any operation $\theta$ in the fractional join operation satisfies 
$\theta(x,y) = x \vee y$.
Hence $C(\theta) = C(\theta(+,-);+,-) \cap C(\theta(-,+);-,+)$.
The operation $\theta$ assigning $(-,+)$ to $-$ and $(+,-)$ to $-$ 
does not appear in the fractional join operation, 
since $C(-;+,-) \cap C(-;-,+) 
= {\rm Cone}(1/\alpha_i, \infty) \cap {\rm Cone}(0, \alpha_i) = \{0\}$.
The other operations are the left join $\vee_L$, the right join $\vee_R$, and $\sqcup_+$, where the corresponding cones are given by
\[
	C(\vee_L) =  {\rm Cone} (0, \alpha_i), \ C(\vee_R) =  {\rm Cone} (1/\alpha_i, \infty),\  
		C(\sqcup_+)  =  {\rm Cone} (\alpha_i, 1/\alpha_i).
\]
To see this, for example, 
$C(\sqcup_+) = C(+;+,-) \cap C(+;-,+) = {\rm cone} (0, 1/\alpha_i) \cap {\rm cone} (\alpha_i, \infty) = {\rm Cone} (\alpha_i, 1/\alpha_i)$.

By Proposition~\ref{prop:frac_join_prod}, 
the fractional join operation on ${{\cal S}_2}^n$ relative to $v$ is equal to
\[
\sum_{\theta_1,\theta_2,\ldots,\theta_n \in \{ \vee_L, \vee_R, \sqcup_+ \}} 
[C(\theta_1) \cap C(\theta_2) \cap \cdots \cap C(\theta_n)]
(\theta_1,\theta_2,\ldots,\theta_n),
\]
where 
$C(\theta_i)$ is considered under valuation $v_i$ for $i=1,2,\ldots,n$.
If $\theta_i \in \{ \vee_L, \vee_R\}$, $\theta_j = \sqcup_+$, and $i < j$,
then $C(\theta_i) \cap C(\theta_j)$ has no interior point.
If $\theta_i = \vee_L$ and $\theta_j = \vee_R$, 
then $C(\theta_i) \cap C(\theta_j)$ has no interior point.
Thus the fractional join operation equals 
\begin{eqnarray*}
	&& \sum_{i=0}^{n-1} [{\rm Cone}(\alpha_{i}, 1/\alpha_{i}) \cap {\rm Cone}(0, \alpha_{i+1})] (\underbrace{\sqcup_+,\ldots \sqcup_+}_i,\vee_L,\ldots,\vee_L) \\
	&& {}+ \sum_{i=0}^{n-1} [{\rm Cone}(\alpha_{i}, 1/\alpha_{i}) \cap {\rm Cone}(1/\alpha_{i+1}, \infty)] (\underbrace{\sqcup_+,\ldots \sqcup_+}_i,\vee_R,\ldots,\vee_R) \\
	&& {}+ [{\rm Cone}(\alpha_{n}, 1/\alpha_{n})] (\sqcup_+, \sqcup_+,\ldots, \sqcup_+) \\
	&& = \sum_{i=0}^{n-1} [{\rm Cone}(\alpha_i, \alpha_{i+1})] \vee_L^i + 
	\sum_{i=0}^{n-1} [{\rm Cone}(1/\alpha_{i+1}, 1/\alpha_{i})] \vee_R^i + [{\rm Cone}(\alpha_{n}, 1/\alpha_{n})] \sqcup_+.
\end{eqnarray*}
From this, we obtain the desired formula~(\ref{eqn:frac_join_alpha}).

\subsubsection{Proof of Theorem~\ref{thm:alpha_bisub}}

\begin{Lem}\label{lem:vee_ij}
	For $1\leq i, j \leq n$, it holds
	$\vee_L^i \wedge \vee_R^i = \sqcup^i$,  
	$\vee_L^i \vee_L^{j} \vee_R^i = \vee_L^{\max\{i,j\}}$,  
	$\vee_L^i \vee_R^{j} \vee_R^i = \vee_R^{\max\{i,j\}}$, and $\vee_L^i \sqcup^{j} \vee_R^i = \sqcup^{\max\{i,j\}}$.
\end{Lem}
\begin{proof}
	We can verify these equations 
	by applying $\vee_L \wedge \vee_R = \sqcup$, $\vee_L \vee_L \vee_R = \vee_L$, 
	$\vee_L \vee_R \vee_R = \vee_R$,  
	$\vee_L \sqcup_+ \vee_R = 
	\sqcup_+ \vee_L \sqcup_+ = \sqcup_+ \vee_R \sqcup_+ = \sqcup \sqcup_+ \sqcup = \sqcup_+$
	to each component.
\end{proof}

Let $f$ be a submodular function on ${{\cal S}_2}^n$ (in our sense).
Fix an arbitrary pair $(p,q)$ of elements in ${{\cal S}_2}^n$. 
We use notation $f(\circ) = f(p \circ q)$.
By Proposition~\ref{prop:frac_join_S_n^2} we have
\[
f(L) + f(R) \geq f(\wedge) + \sum_{i=0}^{n-1} \left( \frac{1}{1+ \alpha_i} - \frac{1}{1+ \alpha_{i+1}}\right) \{ f( \vee_L^i) + f( \vee_R^i)\} + \frac{1- \alpha_n}{1+ \alpha_n} f( \sqcup^n).
\] 
For $k=0,1,2,\ldots,n$, let $B_k$ be defined by
\begin{equation}\label{eqn:B_k}
B_k := 
\sum_{i=k}^{n-1} \left( \frac{1}{1+ \alpha_i} - \frac{1}{1+ \alpha_{i+1}}\right) \{ f( \vee_L^i) + f( \vee_R^i)\} + \frac{1- \alpha_n}{1+ \alpha_n} f( \sqcup^n).
\end{equation}
For $k= 0,1,2,\ldots,n$, we are going show, by induction, 
\begin{equation}\label{eqn:goal_1}
f(L) + f(R) \geq f(\wedge) + \sum_{i=0}^{k-1} (\alpha_{i+1} - \alpha_i) f( \sqcup^i )
+ (1+ \alpha_k) B_k.
\end{equation}
This inequality~\eqref{eqn:goal_1} coincides with the submodularity inequality if $k=0$, and
coincides with the desired inequality if $k=n$.

By applying Lemma~\ref{lem:vee_ij} 
to the submodular inequality for $\vee_L^{k}, \vee_R^{k}$, we obtain
\[
f(\vee_L^{k}) + f(\vee_R^{k}) \geq f(\sqcup^{k}) + \left( 1- \frac{1}{1+ \alpha_{k+1}}\right)
\left\{ f(\vee_R^{k}) + f(\vee_L^{k})\right\} + B_{k+1}.
\] 
Hence we obtain
$
f(\vee_L^{k}) + f(\vee_R^{k}) \geq (1+ \alpha_{k+1}) \left\{ f( \sqcup^{k}) + B_{k+1}\right\}$, and
\begin{eqnarray*}
&& (1+ \alpha_k) B_k =  (1+ \alpha_k) \left[ \left( \frac{1}{1+ \alpha_k} - \frac{1}{1+ \alpha_{k+1}}\right)
\{ f(\vee_L^{k}) + f(\vee_R^{k})\} +  B_{k+1} \right]  \nonumber \\
&& \geq (1+ \alpha_k) \left[ \left( \frac{1}{1+ \alpha_k} - \frac{1}{1+ \alpha_{k+1}}\right)
(1+ \alpha_{k+1}) \left\{ f( \sqcup^{k}) + B_{k+1}\right\} +  B_{k+1} \right] \nonumber \\
&& = (\alpha_{k+1} - \alpha_k) f(\sqcup^k) + (1+ \alpha_{k+1}) B_{k+1}. 
\end{eqnarray*}
Substituting this to (\ref{eqn:goal_1}) at $k$, 
we obtain (\ref{eqn:goal_1}) at $k+1$.
Thus $f$ is $\pmb \alpha$-bisubmodular.

Next we consider the converse direction.
Define $\sqcup_L := \sqcup_+ \wedge \vee_L$ and $\sqcup_R  := \sqcup_+ \wedge \vee_R$.
Then we observe 
$
\sqcup  =  \sqcup_L \wedge \sqcup_R$ and 
$\sqcup_+ =  \sqcup_L \sqcup \sqcup_R = \sqcup_L \sqcup_+ \sqcup_R$.
For $i= 0,1,2,\ldots,n-1$,
define operations $\square_L^i$ and $\bigcirc_L^i$ on ${{\cal S}_2}^n$ by
\[
\square_L^i := (\sqcup_+, \ldots, \sqcup_+, \underbrace{\sqcup_L}_i, \sqcup, \ldots, \sqcup),\quad
\bigcirc_L^i :=  (\sqcup_+, \ldots, \sqcup_+, \underbrace{\sqcup_L}_i, \vee_L, \ldots, \vee_L).
\]
Operations $\square_R^i$ and $\bigcirc_R^i$ are defined by replacing $L$ by $R$.
By using $\sqcup \sqcup \vee_L = \sqcup \sqcup_+ \vee_L = \vee_L$, 
$\sqcup_+ \sqcup \sqcup_L = \sqcup_+ \sqcup_+ \sqcup_L = \sqcup_+$, 
$\sqcup_+ \sqcup \vee_L = \sqcup_L$, and $\sqcup_+ \sqcup_+ \vee_L = \sqcup_+$ componentwise,
we have
\begin{eqnarray*}
	&& \square_L^i \wedge \square_R^i = \sqcup^{i-1}, \quad 
	\square_L^i \sqcup^j \square_R^i = \sqcup^i, \quad \sqcup^i \wedge \bigcirc_L^i = \square^i_L,
	\quad 	\sqcup^i \sqcup^j \bigcirc_L^i = \vee_L^i,\\
	&& \vee_L^{i-1} \wedge \vee_L^i = \bigcirc_L^i, \quad 
	\vee_L^{i-1} \sqcup^j \vee_L^i =
	\left\{ \begin{array}{ll}
		\vee_L^i & {\rm if}\ j \geq i, \\
		\bigcirc_L^i & {\rm otherwise}.
	\end{array}\right. 
\end{eqnarray*}
The relations replacing $L$ by $R$ also hold.
Applying these relations to the $\pmb \alpha$-bisubmodularity inequality~(\ref{eqn:alpha_bisub}), we obtain
\begin{eqnarray*}
	&& f(\square_L^i) + f(\square_R^i) \geq f(\sqcup^{i-1}) + f(\sqcup^{i}), \\ 
	&& f(\sqcup^i) + f(\bigcirc_L^i) \geq f(\square^i_L) + f(\vee_L^i),\\
	&& f(\sqcup^i) + f(\bigcirc_R^i) \geq f(\square^i_R) + f(\vee_R^i),\\
	&& f(\vee_L^{i-1}) + f(\vee_L^i) \geq (1 + \alpha_i) f(\bigcirc_L^i) + (1- \alpha_i) f(\vee_L^i),\\
	&& f(\vee_R^{i-1}) + f(\vee_R^i) \geq (1 + \alpha_i) f(\bigcirc_R^i) + (1- \alpha_i) f(\vee_R^i).
\end{eqnarray*}
Adding them with the fourth and the fifth divided by $1+ \alpha_i$, 
we obtain
\begin{equation}\label{eqn:sqcup^i_i-1}
\frac{1}{1+ \alpha_i} \{ f(\vee_L^{i-1}) + f(\vee_R^{i-1}) \} + f(\sqcup^i)  \geq f(\sqcup^{i-1}) 
+ \frac{1}{1+ \alpha_i} \{ f(\vee_L^i) + f(\vee_R^i) \}.
\end{equation}
We are going to show, by reverse induction on $i= n-1,n-2,\ldots,0$,
\begin{equation}\label{eqn:goal_2}
f(\vee_L^i) + f(\vee_R^i) \geq f(\sqcup^i) + 
\left(1 - \frac{1}{1+ \alpha_{i+1}} \right) \left\{f(\vee_L^i) + f(\vee_R^i)\right\} + B_{i+1}.
\end{equation}
Recall (\ref{eqn:B_k}) for $B_i$.
By $\sqcup^{n} = \vee_L^n = \vee_R^n$, the inequality
\eqref{eqn:sqcup^i_i-1} with $i=n$ gives the base case.
Suppose that \eqref{eqn:goal_2} is true for $i > 0$.
Adding \eqref{eqn:sqcup^i_i-1} to \eqref{eqn:goal_2}, 
we obtain \eqref{eqn:goal_2} for $i-1$:
\begin{equation*}
f(\vee_L^{i-1}) + f(\vee_R^{i-1}) \geq f(\sqcup^{i-1}) + \left( 1 - \frac{1}{1+ \alpha_{i}} \right) \{ f(\vee_L^{i-1}) + f(\vee_R^{i-1}) \} + B_{i}.
\end{equation*}
Thus we have
\begin{equation}\label{eqn:B_0}
f(\vee_L) + f(\vee_R) \geq f(\sqcup) + B_{0}.
\end{equation}
From Lemma~\ref{lem:1-6} and the fact that $p \sqcup q = p \sqcup^+ q$ if $p \vee q$ exists,
we see $\sqcup \sqcup^j L = \vee_L$, $\sqcup \sqcup^j R = \vee_R$, 
$\wedge_L \sqcup^j \wedge_R = \sqcup$, 
and obtain
$f(\sqcup) + f( L )  \geq  f(\wedge_L) + f(\vee_L)$, 
$f(\sqcup) + f( R )  \geq  f(\wedge_R) + f(\vee_R)$, and
$f(\wedge_L) + f(\wedge_R)  \geq  f( \wedge ) + f(\sqcup)$.
Adding them, we obtain
\begin{equation}\label{eqn:f(L)+f(R)}
f(L) + f(R) + f(\sqcup) \geq f( \vee_L) + f(\vee_R) + f( \wedge).
\end{equation} 
Adding (\ref{eqn:B_0}) and (\ref{eqn:f(L)+f(R)}), 
we obtain the submodularity inequality in our sense.

\subsubsection{Proof of Theorem~\ref{thm:Lovasz_modular_lattice}}

Let ${\cal L}$ be a modular lattice of rank $n$. The proof uses the following facts:
\begin{description}
	\item[{\rm (1)}] For two maximal chains there is a distributive sublattice of ${\cal L}$ containing them.
	\item[{\rm (2)}] For a distributive sublattice ${\cal D}$ of rank $n$, the orthoscheme subcomplex $K({\cal D})$ is convex in $K({\cal L})$.
	\item[{\rm (3)}] $f:{\cal L} \to \overline\RR$ is submodular if and only if $f$ is submodular on every distributive sublattice of ${\cal L}$.
\end{description}
(1) follows from \cite[Theorem 363]{Gratzer}.
(2) follows from \cite[Lemma 7.13~(4)]{CCHO14}.
The only-if part of (3) is obvious.
The if-part of (3) follows from the fact that
for $p,q \in {\cal L}$
there is a distributive sublattice containing $p, p \wedge q, q, p \vee q$ (by (1)).

Suppose that the Lov\'asz extension $\overline{f}: K({\cal L}) \to \overline{\RR}$ is convex.
For every distributive sublattice ${\cal D}$ (of rank $n$)
the restriction of $\overline{f}$ to 
$K({\cal D}) \subseteq K({\cal L})$ is also convex by (2). 
By Theorem~\ref{thm:Lovasz}, $f$ is submodular on ${\cal D}$.
By (3), $f$ is submodular on ${\cal L}$.
Suppose that $f:{\cal L} \to \overline\RR$ is submodular.
Take arbitrary two points $x,y$ in $K({\cal L})$.
Then $x$ and $y$ are represented as formal convex combinations 
of two maximal chains $C$ and $C'$, respectively.
By (1), we can take a (maximal) distributive sublattice ${\cal D}$
containing $C$ and $C'$.
The orthoscheme subcomplex $K({\cal D})$ contains $x,y$, and a geodesic $[x,y]$ by (2).
By Theorem~\ref{thm:Lovasz},
the Lov\'asz extension $\overline f$ is convex on $K({\cal D})$.
Therefore $\overline f$ satisfies the convexity inequality (\ref{eqn:convexity}) on $[x,y]$.
Consequently $\overline{f}$ is convex on $K({\cal L})$. 
\subsubsection{Proof of Theorem~\ref{thm:Lovasz_polar_space}}
Let ${\cal L}$ be a polar space of rank $n$. Then the following hold.
\begin{description}
	\item[{\rm (0)}] For a polar frame ${\cal F}$,
	the left and right joins in ${\cal F}$ 
	are equal to those in ${\cal L}$.
	\item[{\rm (1)}] For two maximal chains in ${\cal L}$ there is a polar frame ${\cal F}$ containing them.
	\item[{\rm (2)}] For a polar frame ${\cal F}$, the orthoscheme subcomplex $K({\cal F})$ 
	is convex in $K({\cal L})$.
	\item[{\rm (3)}] $f:{\cal L} \to \overline\RR$ is submodular if and only if $f$ is submodular on every polar frame.
\end{description}
We saw (0) in the proof of Theorem~\ref{thm:frac_join_of_polar_space}.
(1) is axiom (P1).
(2) follows from the argument of the proof of \cite[Proposition 7.4]{CCHO14} 
(the existence of nonexpansive retraction from $K({\cal L})$ to $K({\cal F})$).
(3) follows from the combination of (0) and (1).
Now Theorem~\ref{thm:Lovasz_polar_space} is proved in precisely the same way as Theorem~\ref{thm:Lovasz_modular_lattice} above;
replace (maximal) distributive sublattices by polar frames, and Theorem~\ref{thm:Lovasz} by Theorem~\ref{thm:Qi88}.
Notice that submodular functions on a polar frame ${{\cal S}_2}^n$
are exactly bisubmodular functions.

\subsection{Section~\ref{sec:L-convex}}\label{subsec:proof_L-convex}

\subsubsection{Proof of Lemma~\ref{lem:indicator}}
Let $X$ be a $d_{\varGamma}$-convex set in $\varGamma$.
Let $X^*$ denote the set of vertices $[p,q]$ in $\varGamma^*$ with $p,q \in X$.
Then $[X]^* = [X^*]$ holds.
We first show that $X^*$ is $d_{\varGamma^*}$-convex in $\varGamma^*$.
It suffices to show that any common neighbor $[u,v]$ of any distinct $[p,q],[p',q'] \in X^*$
belongs to $X^*$ (by Lemma~\ref{lem:chepoi}).
We can assume $q' \not \preceq q$.
Then $p \leftarrow p' = u, v = q \leftarrow q'$ or
$u = p \leftarrow p', q \leftarrow q' = v$ or
$p = p' = u$ with $v$ being a common neighbor of $q,q'$ or
$q= q' =v$ with $u$ being a common neighbor of $p,p'$.
Then $u,v \in X$ and $[u,v] \in X^*$, 
where 
the last two cases follow from the $d_{\varGamma}$-convexity of $X$.
Next we prove the L-convexity of $[X]$.
The $\varDelta'$-connectivity of $[X]$ follows from 
the connectivity of the subgraph induced by $X$.
It suffices to show that $[X^*]$ is submodular on each
${\cal I}^*_{p}$ that is $d_{\varGamma^*}$-convex (Proposition~\ref{prop:ideal}).
The intersection $[X^*] \cap {\cal I}^*_{p}$ is $d$-convex in the covering graph of ${\cal I}^*_p$, where $d$ is the path-metric of the covering graph of ${\cal I}^*_p$.
Thus $I([p,q],[p,q'])$ of any $[p,q],[p,q'] \in [X^*] \cap {\cal I}^*_{p}$
is contained in $[X^*] \cap {\cal I}^*_{p}$. 
Thus $\{ [p,q] \wedge [p',q'] \} \cup {\cal E}([p,q],[p',q']) \subseteq [X^*] \cap {\cal I}^*_{p}$, and $[X^*]$ is submodular on ${\cal I}^*_{p}$.

\subsubsection{Proof of Proposition~\ref{prop:L-convex=submodular}}

It suffices to show the if-part.
By well-orientedness,
$\varGamma^*$ is the poset of all intervals of $\varGamma$ with reverse inclusion order (Lemma~\ref{lem:om_Bgated}).
In particular, ${\cal I}^*_p$ is the poset of all intervals containing $p$, 
and is isomorphic to ${\cal I}_p \times {\cal F}_p$ by $[q,q'] \mapsto (q, q')$.
By this isomorphism, $g^*$ can be regarded as a function on 
the product ${\cal I}_p \times {\cal F}_p$ of 
two modular semilattices ${\cal I}_p$ and ${\cal F}_p$, 
defined by $g^*(q,q') := (g(q) + g(q'))/2$.
Since $g$ is submodular on ${\cal I}_p$ and on ${\cal F}_p$, 
the direct sum $g^*$ is also submodular.
This means that $g^*$ is submodular on 
every neighborhood semilattice, and hence $g$ is L-convex.

\subsubsection{Proof of Proposition~\ref{prop:L_is_L-extendable}}
We start with a preliminary argument.
Let ${\cal L}$ be a complemented modular lattice 
with covering graph $\varGamma$.
Now $\varGamma$ is thick 
(since every interval of ${\cal L}$ is a complemented modular lattice). 
By Theorem~\ref{thm:dual_polar}, $\varGamma$ is a dual polar space.
Consider the barycentric subdivision $\varGamma^*$ of $\varGamma$, 
which is the poset of all intervals $[p,q]$ of ${\cal L}$ 
with respect to the reverse inclusion order. 
Then $\varGamma^*$ is equal to the polar space corresponding to $\varGamma$.

We use the following explicit formulas of $\wedge$, 
$\vee_L$, $\vee_R$, and $\sqcup$ in  $\varGamma^*$:
\begin{eqnarray}
&& [p,q] \wedge [p',q'] = [p \wedge p', q \vee q'], \label{eqn:interval_meet} \\ 
&& [p,q] \vee_L [p',q'] = [p \vee (q \wedge p'), q \wedge (p \vee q')] = 
 [p',q'] \vee_R [p,q],  \label{eqn:interval_left} \\
&& [p,q] \sqcup [p',q'] = [(p \vee p') \wedge (q \wedge q'), (p \vee p') \vee (q \wedge q')].
\label{eqn:interval_join} 
\end{eqnarray}
Notice that $p \vee (q \wedge p') = q \wedge (p \vee p') \preceq  q \wedge (p \vee q')$ holds 
by modularity, and hence the left and right joins are well-defined intervals.
Also the join $[p,q] \vee [p',q']$ 
is equal to nonempty intersection $[p,q] \cap [p',q']$;
thus the join exists if and only if $p \vee p' \preceq q \wedge q'$. 

It is easy to see (\ref{eqn:interval_meet}). 
To see (\ref{eqn:interval_left}), 
consider a minimal interval $[s,t]$ 
with $[p',q'] \subseteq [s,t] \subseteq [p \wedge p', q \vee q']$ 
and $[s,t] \cap [p,q] \neq \emptyset$.
Then $t \succeq q \wedge t \succeq p \vee s \succeq p$ necessarily holds.
This implies $t \succeq p \vee q'$. Similarly $s \preceq q \wedge p'$.
On the other hand,
$[q \wedge p', p \vee q'] \cap [p,q] \neq \emptyset$ 
since $(q \wedge p') \vee p = (p \vee p') \wedge q \preceq (p \vee q') \wedge q$.
By minimality we have $t = p \vee q'$ and $s = q \wedge p'$.
From this, 
we obtain (\ref{eqn:interval_left}).
The equality (\ref{eqn:interval_join}) is obtained by definition (\ref{eqn:sqcup}) of $\sqcup$ 
with using modular equality $x \vee (y \wedge z) = (x \vee y) \wedge z$ for $x \succeq z$.

Let us start the proof of Proposition~\ref{prop:L_is_L-extendable}.
It suffices to show that $g^*$ is an L-convex relaxation.
We have seen in the proof of Proposition~\ref{prop:L-convex=submodular}  
that $g^*$ is submodular on every neighborhood semilattice, 
and hence on every principal ideal.
By Lemma~\ref{thm:G->G*}, $\varGamma^*$ is well-oriented. 
Therefore it suffices to show that $g^*$ is submodular on every principal filter of $\varGamma^*$.
Take an arbitrary vertex $X$ of $\varGamma^*$ 
that is represented as $X = [u,v]$ for $u \sqsubseteq v$.
The principal filter of $[u,v]$ is the semilattice of all subintervals of $[u,v]$, 
and is equal to the interval poset of the complemented modular lattice $[u,v]$.
Thus the principal filter is a polar space, and
it suffices to show the inequality in 
Theorem~\ref{thm:polar_submo}~(2).
By (\ref{eqn:interval_meet}), (\ref{eqn:interval_join}), and submodularity of $g$ on $[u,v]$, 
we have 
\begin{eqnarray*}
&& g^*([p,q] \wedge [p',q']) + g^*([p,q] \sqcup [p',q']) \\
&&  = g^*([p \wedge p',q \vee q']) + g^*([(p \vee p') \wedge (q \wedge q'), (p \vee p') \vee (q \wedge q')]) \\
&&  = \{g(p \wedge p') + g(q \vee q'))\}/2 +  
 \{g((p \vee p') \wedge (q \wedge q')) + g((p \vee p') \vee (q \wedge q'))\}/2 \\
&& \leq  \{g(p \wedge p') + g(q \vee q'))\}/2 + \{g(p \vee p') + g(q \wedge q')\}/2 \\ && \leq  \{g(p) + g(p') + g(q) + g(q')\}/2 = g^*([p,q]) + g^*([p',q']).
\end{eqnarray*}
Thus $g^*$ is submodular on the principal filter of every interval, and hence $g^*$ is L-convex.
The exactness is immediate from the definition of $g^*$.

\subsubsection{Proof of Lemma~\ref{lem:sum_of_L}}

Let $H$ be an swm-graph.
For vertices $x,y$, let $\lgate x,y \rgate$ denote the minimum $d_H$-gated set containing $x,y$.
For vertices $x,y$, a $\varDelta$-path $(x=x_0,x_1,\ldots,x_m =y)$
is called a {\it normal $\varDelta$-path}
({\em normal Boolean-gated path})
from $x$ to $y$~\cite[Section 6.6]{CCHO14} 
if for every index $i$ with $0 < i < m$ and every Boolean-gated set $B$ 
containing $\lgate x_{i-1},x_i \rgate$
it holds $B \cap \lgate x_i, x_{i+1} \rgate = \{x_i\}$.
%
\begin{Thm}
	[{\cite[Theorem 6.20]{CCHO14}}]\label{thm:normal}
	For vertices $x,y$, 
	a normal $\varDelta$-path from $x$ to $y$ uniquely exists.
\end{Thm}

We prove a global convexity property 
of the domain of L-extendable functions. 
\begin{Prop}\label{prop:belong}
	Let $h$ be an L-extendable function on an swm-graph $H$.
	For any $x,y \in \dom h$, the normal $\varDelta$-path from $x$ to $y$ is contained in $\dom h$. 
\end{Prop}
\begin{proof}

The proof is based on the idea of Abram and Ghrist~\cite{AbramGhrist04} 
to find normal cube paths in CAT(0) cube complex.
Let $g: H^* \to \overline{\RR}$ be an L-convex relaxation of $h$.
For $x,y \in \dom h$, 
take a $\varDelta$-path $P = (x = x_0,x_1,\ldots,x_m = y)$ in $\dom h$
such that $I_P := \sum_{i=1}^m i \cdot d_{H} (x_{i-1},x_i)$ is minimum.
We remark that $\lgate x_i, x_{i-1} \rgate$ is also Boolean-gated (\cite[Lemma 6.8]{CCHO14}), 
and $\{ x_i\} \wedge \{ x_{i-1}\} = \lgate x_i, x_{i-1} \rgate$ in ${\cal F}_{\lgate x_i, x_{i-1} \rgate}$. 
By submodularity of $g$ on ${\cal F}_{\lgate x_i, x_{i-1} \rgate}$ 
with $x_i,x_{i-1} \in \dom g$,
it holds $\lgate x_i, x_{i-1} \rgate \in \dom g$.
Also $P$ is normal if and only if $\lgate x_{i-1},x_{i} \rgate \wedge_R \lgate x_{i},x_{i+1} \rgate = \{x_i\}$ 
in ${\cal I}_{\{x_i\}}$ for $i=1,2,\ldots,m-1$.

We show that $P$ is normal.
Suppose not. There is an index $i$ $(0 < i < m)$ 
such that $U := \lgate x_{i-1},x_{i} \rgate \wedge_R \lgate x_{i},x_{i+1} \rgate \supset \{x_i\}$ in ${\cal I}_{\{x_i\}}$.
Then $U$ is contained in $\dom g$
since $U$ is the meet of $\lgate x_{i},x_{i+1} \rgate \in \dom g$
and
$\lgate x_{i-1},x_{i} \rgate \vee_L \lgate x_{i},x_{i+1} \rgate \in \dom g$.
Now $\lgate x_i, x_{i+1} \rgate \supseteq U \supset \{x_i\}$.
Consider $U \vee_L \{x_{i+1}\}$ 
in polar space ${\cal F}_{\lgate x_i, x_{i+1} \rgate}$, 
which consists of a single vertex $x'_{i}$ by (\ref{eqn:rank_eq}).  
Also  $x'_{i}$ belongs to $\dom g$ and to $\dom h$, 
and is different from $x_i$
(by $\{ x'_i\} \wedge \{ x_{i+1}\} \subset \lgate x_i,x_{i+1} \rgate = \{x_i\} \wedge \{x_{i+1}\}$). 
Then 
$x_{i+1}$ and $x'_i$ are $\varDelta$-adjacent 
by $x_{i+1},x'_i \in \lgate x_i, x_{i+1} \rgate$.
Also $x_{i-1}$ and $x'_i$ are $\varDelta$-adjacent 
since $x'_i \in U \subseteq U \vee \lgate x_i,x_{i-1} \rgate \supseteq \lgate x_i,x_{i-1} \rgate \ni x_{i-1}$.
Replace $x_i$ by $x'_i$ in $P$, which is again a $\varDelta$-path in $\dom h$.
We finally show that $I_P$ strictly decreases in this modification.
Since $x'_i \in I_{H^*}(x_{i+1},x_{i})$, 
it holds $d_{H^*}(x_{i+1},x_i) = d_{H^*}(x_{i+1},x'_i) + d_{H^*}(x'_{i},x_{i})$, 
and $d(x_{i+1},x_i) = d(x_{i+1},x'_i) + d(x'_{i},x_{i})$ (by Proposition~\ref{thm:G->G*}).
Also $d(x'_i,x_{i-1}) \leq d(x'_i,x_i) + d(x_{i},x_{i-1})$.
Therefore 
$i d(x_{i-1},x_i) + (i+1) d(x_{i},x_{i+1}) \geq 
i d(x_{i-1},x'_i) + (i+1) d(x'_{i},x_{i+1}) + d(x'_i,x_i) > i d(x_{i-1},x'_i) + (i+1) d(x'_{i},x_{i+1})$. Thus $I_P$ strictly deceases. 
This is a contradiction to the minimality of $I_P$. 
\end{proof}

We are ready to prove Lemma~\ref{lem:sum_of_L}.
Let $g$ and $g'$ be L-convex functions on an oriented modular graph $\varGamma$.
It suffices to show that $\dom g \cap \dom g'$ is $\varDelta'$-connected. 
Now $g$ and $g'$ are also regarded as L-extendable functions on swm-graph $\varGamma$~(Proposition~\ref{prop:L_is_L-extendable}).
For $x,y  \in \dom g \cap \dom g'$, the normal $\varDelta$-path 
$(x = x_0,x_1,\ldots,x_m = y)$ from $x$ to $y$  
is contained in $\dom g \cap \dom g'$ (Proposition~\ref{prop:belong}). 
By submodularity on each interval, $\varDelta'$-path
$(x = x_0, x_0 \wedge x_1, x_1,x_1 \wedge x_2, \ldots,x_m = y)$ 
is also contained in $\dom g \cap \dom g'$.
Thus $\dom g \cap \dom g'$ is $\varDelta'$-connected. 

\subsubsection{Proof of Theorem~\ref{thm:L-optimality}}
We can assume that $\varGamma$ is well-oriented.
Otherwise, consider subdivision $\varGamma^*$ and exact L-convex relaxation $g^*$ on $\varGamma^*$.
If $g(p) = g^*([p,p]) > g^*([p',q'])$ and $p' \sqsubseteq p \sqsubseteq q'$, 
then $g(p) > g(p')$ or $g(p) > g(q')$.
Suppose first that the image of $g$ is discrete.
By the $\varDelta'$-connectivity of $\dom g$ and nonoptimality of $g$, 
there is a $\varDelta'$-path $(p = p_0, p_1,\ldots, p_m)$ 
such that $g(p) > g(p_m)$.
Consider all such paths minimizing $\max_i g(p_i)$;
the existence of such paths is guaranteed by
the discreteness of the image of $g$ and $\max_i g(p_i) \geq g(p)$. 
Among these paths, choose a path 
with the index set $I := \{ i \mid g(p_i)  = \max_j g(p_j)\}$ minimal. 
We show that $I = \{0\}$.
We can choose $i \in I$ such that $g(p_{i-1}) \leq g(p_i) > g(p_{i+1})$.
By the well-orientedness and the minimality, 
it holds $p_{i-1} \prec p_i \succ p_{i+1}$ or 
$p_{i-1} \succ p_i \prec p_{i+1}$.
We may assume that $p_{i-1}, p_{i+1} \in {\cal F}_{p_{i}}$, 
and $p_{i-1}$ and $p_{i+1}$ are incomparable. 
By Lemma~\ref{lem:f(v_i)<f(v_i+1)}, 
there are $p_{i-1} = q_0,q_1,\ldots,q_k = p_{i+1} \in {\cal F}_{p_{i}}$ 
such that $g(p_{i}) > g(q_j)$ for $j=1,2,\ldots,k$ 
and $q_i \prec q_{i+1}$ or $q_i \succ q_{i+1}$.
Replacing $p_i$ by $q_1,q_2,\ldots,q_k$ in the path, 
the resulting $\varDelta'$-path decreases 
$\max_i g(p_i) \geq g(p)$ or $I$. 
This is a contradiction to the minimality.

Next suppose that $\varGamma$ is locally-finite.
Then $\varGamma^{\varDelta}$ is also locally-finite.
Indeed, consider a vertex $x$ and its neighbor $y$ in $\varGamma^{\varDelta}$.
Then $y$ is the join of some atoms (neighbors of $x \wedge y$ in $\varGamma$) 
in ${\cal F}_{x \wedge y}$.
Also $x \wedge y$ is the join of some atoms (neighbors of $x$ in $\varGamma$) in $[x \wedge y,x]$.
Consequently, possible neighbors of $x$ in $\varGamma^{\varDelta}$ are finite.

Consider the ball $B = B^{\varDelta}_r(x)$ for large $r$, 
which is a $d_{\varGamma}$-convex set by Lemma~\ref{lem:ball_convex} 
and is a finite set by the local-finiteness of $\varGamma^{\varDelta}$.
Consider $g + [B]$, which is L-convex by Lemma~\ref{lem:sum_of_L}.
Also $p$ is not optimal for $g + [B]$, and the image of $g + [B]$ is discrete.
Thus this reduces to the case above.

\subsubsection{Proof of Theorem~\ref{thm:bound}}

Let $\varGamma$ be an oriented modular graph.
In this case, 
$\lgate x,y \rgate$ (the smallest $d$-gated set containing $x,y$)
is the smallest $d_{\varGamma}$-convex set containing $x,y$ (by Lemma~\ref{lem:chepoi}).
Consider the thickening $\varGamma^{\varDelta}$.
If vertices $x$ and $y$ are adjacent in $\varGamma^{\varDelta}$, 
then $x$ and $y$ are said to be {\em $\varDelta$-adjacent}, 
and $y$ is called a {\em ${\varDelta}$-neighbor} of $x$.
\begin{Lem}\label{lem:oriented123}
	\begin{description}
		\item[{\rm (1)}] vertices $x,y$ are $\varDelta$-adjacent 
		if and only if both $x \vee y$ and $x \wedge y$ exist and $x \wedge y \sqsubseteq x \vee y$.
		\item[{\rm (2)}] If $x \preceq y$, then $I(x,y) = [x,y] = \lgate x,y \rgate$.
		\item[{\rm (3)}] For $x,y,z,v \in \varGamma$ with $x \preceq z \preceq y$, 
		if $d^{\varDelta}(v,x) = d^\varDelta(v,y) = k$, then $d^\varDelta(v,z) \leq k$.
	\end{description}
\end{Lem}
\begin{proof}
	(1) follows from Lemma~\ref{lem:om_Bgated}.
	(2) follows from \cite[Lemmas 4.13 and 4.14]{HH150ext}.
	(3). By (2) and Lemma~\ref{lem:ball_convex}, 
	we have $B^{\varDelta}_k(v) \supseteq I(x,y) = [x,y] \ni v$.
\end{proof}

\begin{Prop}[{\cite[Lemma 6.17, Proposition 6.18]{CCHO14}}]\label{prop:delta-gate}
	For distinct vertices $x,y$, there uniquely exists a $\varDelta$-neighbor $u$ of $x$  
	 having the following properties:
	 \begin{description}
	 	\item[{\rm (1)}] $d^{\varDelta}(x,y) = 1 + d^{\varDelta}(u,x)$.
	 	\item[{\rm (2)}] For a $\varDelta$-neighbor $v$ of $x$,  
	 	if $d^{\varDelta}(x,y) = 1 + d^{\varDelta}(v,x)$, then $u \in \lgate x,v \rgate$.
	 	\item[{\rm (3)}] For a $\varDelta$-neighbor $v$ of $x$, 
	    $d^{\varDelta}(v,y) = 1 + d^{\varDelta}(x,y)$ if and only if 
	    $v$ is not $\varDelta$-adjacent to $u$. 
	 \end{description}
\end{Prop}
This vertex $u$ is called the {\em $\varDelta$-gate} of $y$ at $x$.
To obtain an intuition of the $\varDelta$-gate, 
consider the case of $\varGamma = \ZZ^n$. 
For distinct $x,y \in \ZZ^n$ with $x \leq y$, 
the $\varDelta$-gate of $y$ at $x$ is equal to $x + \sum \{e_i \mid y_i - x_i = \|x - y\|_{\infty} \}$.

Let us start the proof of Theorem~\ref{thm:bound}. 
Suppose that $\varGamma$ is well-oriented.
Let $x = x^0,x^1,\ldots, x^m$ be a sequence generated by 
SDA applied to an L-convex function $g$ and an initial vertex $x$.
\begin{Lem}\label{lem:geodesic}
	Suppose that $g(x) = \min \{  g(y) \mid y \in {\cal F}_{x}\}$ or $g(x) = \min \{  g(y) \mid y \in {\cal I}_{x}\}$.
	For $z \in {\cal F}_{x^k} \cup {\cal I}_{x^k}$, 
	if $g(x^k) > g(z)$, then
	$d^{\varDelta}(x,z) = k+1$.
\end{Lem}

\begin{proof}
	By the well-orientedness of $\varGamma$ and the definition of SDA (with reversed orientation if necessarily)
	we have 
	$
	x = x^0 \succ x^1 \prec x^2 \succ x^3 \prec \cdots. 
	$
	Then it hold $g(x^{i}) = \min \{  g(y) \mid y \in {\cal I}_{x^i}\}$ if $i$ is odd, and 
	$g(x^{i}) = \min \{  g(y) \mid y \in {\cal F}_{x^i}\}$ if $i$ is even.
	We use the induction on $k$.
	Suppose that $k$ is odd.
	Then $x^{k-1}$ and $z$ belong to ${\cal F}_{x^k}$.
	By induction, we have $d^{\varDelta}(x,x^{k-1}) = k-1$ 
	and $d^{\varDelta}(x,x^{k}) = k$.
	By $g(x^{k-1}) > g(z) < g(x^{k}) \leq g(x^{k-1} \wedge z)$  and Lemma~\ref{lem:f(v_i)<f(v_i+1)}, 
	there is $y \in I(x^{k-1},z) \subseteq  {\cal F}_{x^k}$ 
	with $g(y) < g(x^{k-1})$ such that
	$y \in [x^{k-1} \wedge (z \vee_L x^{k-1}), x^{k-1}]$ 
	or $y \succ x^{k-1}$.
	Then $y \succ x^{k-1}$ is impossible by $g(x^{k-1}) = \min \{  g(y) \mid y \in {\cal F}_{x^{k-1}}\}$.
	Thus $y \in [x^{k-1} \wedge z, x^{k-1}]$ holds.
	By induction, $d^{\varDelta}(x,y) = k$ holds.
	Let $h$ be the $\varDelta$-gate of $x$ at $x^{k}$. 
	Then $d^{\varDelta}(x,h) = k-1$ (Proposition~\ref{prop:delta-gate}~(1)).
	By Proposition~\ref{prop:delta-gate}~(2) and Lemma~\ref{lem:oriented123}~(2), 
	we have $\lgate h,x^{k} \rgate \subseteq \lgate x^{k-1}, x^{k} \rgate= [x^{k-1}, x^k]$. 
	In particular, $h$ belongs to $[x^{k-1}, x^k]$.
	Then $h \not \preceq y$ must hold. 
	Otherwise, by Lemma~\ref{lem:oriented123}~(3),
	$d^{\varDelta}(x,x^{k-1}) = d^{\varDelta}(x,h) = k-1$ 
	and $y \in [h, x^{k-1}]$ 
	imply $d^\varDelta(x,y) \leq k-1$, contradicting $d^{\varDelta}(x,y) = k$.
	Consider $y' := x^{k-1} \wedge (z \vee_L x^{k-1})$. 
	Now $y' \preceq y$. Thus $y' \vee h$ is strictly greater than $y'$ 
	(by $h \not \preceq y$).
	Consequently $h$ and $z$ cannot have the join.
	Otherwise, since $y' \vee h$ and $y' \vee z$ exist,
	by definition of modular semilattice,  
	$y' \vee z \vee h$ exists, and  
	is strictly greater than $z \vee_L x^{k-1}$, which 
	contradicts the definition of $\vee_L$.
	By Lemma~\ref{lem:oriented123}, $z$ and $h$ are not $\varDelta$-adjacent.
	By Proposition~\ref{prop:delta-gate}~(3), 
	it holds $d^{\varDelta}(x,z) = k+1$, as required.
	The case of $k$ even is similar.
\end{proof}
The rest of the argument is exactly the same as the proof of \cite[Theorem 2.6]{HH14extendable};
apply the above lemma to L-convex function $g + [B^{\varDelta}_r(x)]$ 
with $r := d^{\varDelta}(x, {\rm opt}(g))$.

\subsubsection{Proof of Theorem~\ref{thm:persistency}}
The proof is exactly the same as the proof of \cite[Theorem 2.10]{HH14bounded} 
for the special case where $\varGamma$ is the product of zigzag-oriented trees.
By Lemma~\ref{lem:geodesic}, \cite[Proposition 2.10]{HH14bounded} holds for our case.
Consequently, the argument in~\cite[Section 2.5.2]{HH14bounded} works by replacing 
$y \sqcup (y \sqcup w)$ with  $y \vee_L w$, $w \sqcup (w \sqcup y)$ with  $y \vee_R w$, and $B^n$ with $H$, respectively.
Since ${\cal F}_z$ is a polar space (Lemma~\ref{lem:well-oriented}), 
the rank of $y \sqcup_R w$ is equal to the rank of $y \sqcup_L w = y$ 
in ${\cal F}_z$ (see the proof of Theorem~\ref{thm:frac_join_of_polar_space}). 
This means that $y \sqcup_R w$ is 
a maximal element in the polar space ${\cal F}_z$, and belongs to $H$.

\subsubsection{Proof of Theorem~\ref{thm:building}}

For a positive integer $m$,
define $\varGamma^{*m}$ by: $\varGamma^{*m} := (\varGamma^{*(m-1)})^*$ if $m \geq 1$ 
and $\varGamma^{*0} := \varGamma$.
For a function $g:\varGamma \to \overline{\RR}$, 
define $g^{*m}: \varGamma^{*m} \to \overline{\RR}$ 
by: $[u,v] \mapsto \{g^{*(m-1)}(u) + g^{*(m-1)}(v)\}/2$ and $g^{*0} := g$.
For an oriented modular graph $\varGamma$ and its subdivision $\varGamma^*$,
$K(\varGamma^*)$ is isometric to $K'(\varGamma)$, 
where the isometry is 
given by $x = \sum_{i} \lambda_i [p_i, q_i] \mapsto \sum_i \lambda_i ( p_i + q_i)/2$~\cite[Proposition 8.7]{CCHO14}.
In particular, $K(\varGamma^{*m})$ is 
a simplicial subdivision of $K'(\varGamma)$.
\begin{Lem}\label{lem:g*m=gbar}
	For a function $g: \varGamma \to \overline{\RR}$, 
	we have $\overline{g^{*m}} = \overline{g}$ for $m=0,1,2,\ldots$.
\end{Lem}
\begin{proof}
	It suffices to show that $\overline{g^{*}} = \overline{g}$.
	For $x = \sum_{i} \lambda_i [p_i, q_i] \simeq \sum_{i} \lambda_i (p_i  + q_i)/2$, 
	we have $\overline{g^{*}}(x) = \sum_{i} \lambda_i g^*([p_i, q_i]) 
		= \sum_{i} \lambda_i (g(p_i) + g( q_i))/2 = \overline{g} (\sum_{i} \lambda_i (p_i  + q_i)/2) = \overline{g}(x)$, where $p_i,q_i$ form a chain, and the last equality follows from the definition of $\overline{g}$.
\end{proof}
We first prove the theorem for the case where building $K$ itself is a single apartment.
Namely, $K = K(\check{\ZZ}^n)$ and $\varGamma(K) = \check{\ZZ}^n$.

(1) $\Rightarrow$ (2).
Suppose that $g$ is an L-convex function on $\check{\ZZ}^n$.
By $\varDelta'$-connectivity of $g$, 
the domain of $\overline g$ is connected.
By the Tietze-Nakajima theorem, 
it suffices to show that $\overline{g}$ is locally convex.
Take an arbitrary $x \in \dom \overline{g} \subseteq  K(\check{\ZZ}^n)$.
For sufficiently large $m$, there is a vertex $p$ of $\varGamma^{*m}$ such that
the point $x$ belongs to the interior of 
the subcomplex $K({\cal F}_p)$ of $K(\varGamma^{*m})$ 
(for ${\cal F}_p = {\cal F}_p(\varGamma^{*m})$).
We show that $\overline{g}$ is convex on $K({\cal F}_p)$.
Since $\overline{g}$ is equal to $\overline{g^{*m}}$ 
and $g^{*m}$ is L-convex on $\varGamma^{*m}$ (Proposition~\ref{prop:L_is_L-extendable}), 
$g^{*m}$ is (bi)submodular on ${\cal F}_p \simeq {{\cal S}_2}^n$ (Theorem~\ref{thm:polar_submo}).
By Theorem~\ref{thm:Qi88}, $\overline{g^{*m}} = \overline{g}$ is convex on $K({\cal F}_p)$.
Thus $\overline{g}$ is locally convex, as required.

(2) $\Rightarrow$ (3).
Suppose that $\overline{g}$ is convex on $K(\check{\ZZ}^n) \simeq \RR^n$.
Take integral vectors $x,y \in \dom g \subseteq \ZZ^n$,
Then the midpoint of the geodesic between $x$ and $y$
is equal to $(x + y)/2$, where $+$ and $\cdot /2$ are 
the usual operations in $\RR^n$.
By convexity of $\overline{g}$ with $\overline{g}(x) = g(x)$ 
and $\overline{g}(y) = g(y)$, we have 
$
g(x) + g(y) \geq 2 \overline g ((x+y)/2).
$
Now $(x+y)/2$ is a half-integral vector, 
and is the midpoint of the edge between $\lceil (x+y)/2 \rceil$ and $\lfloor (x+y)/2 \rfloor$ 
in $K(\check{\ZZ}^n)$. 
Here $\lfloor (x+y)/2 \rfloor \sqsubseteq \lceil (x+y)/2 \rceil$. 
Therefore we have $\overline g((x+y)/2) = g(\lceil (x+y)/2 \rceil )/2 + g( \lfloor (x+y)/2 \rfloor )/2$.
Thus we obtain (3).

(3) $\Rightarrow$ (1). 
Suppose that $g$ satisfies (3). 
We first show that $g$ is submodular in every principal ideal and filter.
It suffices to consider the case of the principal filter ${\cal F}_z$ of an even integral vector $z$.
Then ${\cal F}_z = \{ z + \sum_{i=1}^n s_i \mid s_i \in \{-1,0,1\}\}$, 
and is isomorphic to ${{\cal S}_2}^n$ by 
$z + \sum_{i=1}^{n} s_i \mapsto (s_1,s_2,\ldots,s_n) \in {{\cal S}_2}^n$.
Then one can observe 
that for $x,y \in {\cal F}_z$ 
it holds $\lceil (x+y)/2 \rceil = x \sqcup y$ and $\lfloor (x+y)/2 \rfloor = x \wedge y$.
Therefore $g$ is (bi)submodular on ${\cal F}_z$.

We next show the $\varDelta'$-connectivity of $\dom g$.
Take  $x,y \in \dom g$. 
We show by induction on $k := \|x-y\|_{\infty}$.
Suppose that $k \leq 1$. 
In this case, $x$ and $y$ belong to 
the principal ideal or filter of some vertex $z$. 
The inequality in (3) is equal to $g(x) + g(y) \geq g(x \sqcup y) + g(x \wedge y)$.
Then $(x, x \wedge y, y)$ is a $\varDelta'$-path, as required.
Suppose that $k \geq 2$.
Then $\lceil (x+y)/2 \rceil$ and $\lfloor (x+y)/2 \rfloor$ belong to $\dom g$.
Also $\lfloor (x+y)/2 \rfloor \sqsubseteq  \lceil (x+y)/2 \rceil$.
Both $\| x -\lceil (x+y)/2 \rceil \|_{\infty}$ and $\| y - \lfloor (x+y)/2 \rfloor$
are at most $ \|_{\infty} \leq \lceil k/2 \rceil$.
By induction, pairs $(x, \lceil (x+y)/2 \rceil)$ and $(\lfloor (x+y)/2 \rfloor, y)$ are connected by $\varDelta'$-paths, respectively.
Hence $x$ and $y$ are connected by a $\varDelta'$-path.

Next we consider the general $K$.

(1) $\Rightarrow$ (2).
Suppose that $g$ is L-convex.
Take arbitrary $x,y$ in $K$.
Consider the geodesic $[x,y]$ between them.
Take an apartment $\varSigma$ containing $x,y$.
Then the unique geodesic between $x,y$ is contained in $\varSigma$ 
(by convexity of apartments in $K$; see~\cite[Theorem 11.16(4)]{BuildingBook}).
Now $g$ is L-convex on $\varGamma(\varSigma) \simeq \check\ZZ^n$. 
Therefore $\overline{g}$ is convex on $\varSigma$ 
by the above-proved (1) $\Rightarrow$ (2).
From this $\overline{g}$ satisfies (\ref{eqn:convexity}) on $[x,y]$. 
Hence $\overline{g}$ is convex on $K$.

(2) $\Rightarrow$ (3). Since $\overline{g}$ is convex on every apartment, 
$g$ satisfies (3) as above.

(3) $\Rightarrow$ (1). Suppose that $g$ satisfies (3).
The $\varDelta'$-connectivity 
can be shown similarly by taking an apartment containing any two vertices.
Since $\varGamma(K)$ is well-oriented, 
it suffices to show that $g$ is submodular on the principal ideal ${\cal I}_p$ 
and filter ${\cal F}_p$ of every vertex $p$.
Then both ${\cal I}_p$ and ${\cal F}_p$ 
are polar spaces; compare definitions of polar space and Euclidean building.
We can assume that $p$ has label $0$.
It suffices to show that $g$ is submodular on ${\cal I}_p$.
Take any $x,y$ in ${\cal I}_p$, and take an apartment $\varSigma$ containing $\{x,p\}$ and $\{y,p\}$.
The intersection of ${\cal I}_p$ and (the vertex set of) $\varSigma$
forms a polar frame in the polar space ${\cal I}_p$.
Thus the vertex set of $\varSigma$ is identified with $\ZZ^n$
and ${\cal I}_p$ is identified with $\{-1,0,1\}^n \subseteq \ZZ^n$.
Under this identification, 
it is easy to see that $\lfloor (x+y)/2 \rfloor = x \wedge y$ and $\lceil (x+y)/2 \rceil = x \sqcup y$.
Therefore (3) coincides 
with the submodularity inequality in Theorem~\ref{thm:polar_submo}.

\subsubsection{Proof of Proposition~\ref{prop:oriented_tree}}
We can assume that each $T_i$ is a tree without leaves.

(1) $\Rightarrow$ (2). 
Suppose that $g$ is L-convex.
Consider the barycentric subdivision 
$\varGamma^* = T_1^{*} \times T_2^* \times \cdots \times T_n^*$.
Then $K (\varGamma^*)$ is a Euclidean building, 
where apartments are subcomplexes $K(P_1 \times P_2 \times \cdots \times P_n)$
for infinite paths $P_i \subseteq T^*_i$ $(i=1,2,\ldots,n)$.
Therefore, by Proposition~\ref{prop:L_is_L-extendable}, 
$g^*$ is L-convex on the Euclidean building.
By Theorem~\ref{thm:building}, 
the Lov\'asz extension $\overline{g^*}: K(\varGamma^*) \to \overline{\RR}$ is convex.
Now $K(\varGamma^*)$ is isometric to $K'(\varGamma)$, 
and $\overline{g} = \overline{g^*}$ (Lemma~\ref{lem:g*m=gbar}).
Hence the Lov\'asz extension $\overline{g}: K'(\varGamma) \to \overline{\RR}$ is convex.

(2) $\Rightarrow$ (3).
Take two vertices $x,y \in \varGamma$.
Since $\overline{g}$ is convex, we have
$
\overline{g}(x) + \overline{g}(y) \geq 2 \overline{g}( (x+y)/2).
$
Since $\overline{g}( (x+y)/2) = g^*( [\lfloor (x+y)/2 \rfloor,  \lceil (x+y)/2 \rceil]) 
= \{g(\lfloor (x+y)/2 \rfloor) + g( \lceil (x+y)/2 \rceil)\}/2$, 
we obtain (3), as required.

(3) $\Rightarrow$ (1).
We start with a preliminary argument.
Let $T$ be an oriented tree.
For vertices $p,q$ in $T$, let $p \circ q$ and $p \bullet q$
denote $\lceil (p+q)/2 \rceil$ and $\lfloor (p+q)/2 \rfloor$, respectively.
Let $T^*$ be the subdivision of $T$.
Take an arbitrary vertex $p$ of $T$. 
Consider neighborhood semilattice ${\cal I}^*_p \subseteq T^*$.
Take two $x,y \in {\cal I}^*_p$.
Suppose that $x = [\underline{x}, \overline{x}]$ and $y = [\underline{y}, \overline{y}]$
for $\underline{x}, \overline{x}, \underline{y}, \overline{y} \in T$.
By case-by-case analysis, it holds
\begin{equation}\label{eqn:circ-bullet}
 [\overline{x} \circ \overline{y}  \bullet \underline{x} \bullet \underline{y},\  \overline{x} \circ \overline{y} \circ \underline{x} \bullet \underline{y} ] = x \sqcup y, 
\quad [\overline{x} \bullet \overline{y}  
\bullet \underline{x} \circ \underline{y},\  \overline{x} \bullet \overline{y} \circ \underline{x} \circ \underline{y} ] = x \wedge y,
\end{equation}
where $a \diamond b \diamond' c \diamond'' d$ means  $(a \diamond b) \diamond' (c \diamond'' d)$ 
for $\diamond, \diamond', \diamond'' \in \{ \circ, \bullet \}$.
For example, if $\overline{x} \to \underline{x} = p = \overline{y} \to \underline{y}$,
then one can see (\ref{eqn:circ-bullet}) from $x \sqcup y = x \wedge y = [p,p]$, 
$\overline{x} \circ \overline{y} = \overline{x}$,
$\underline{x} \bullet \underline{y} = \underline{y}$, and
$\overline{x} \bullet \overline{y} = \underline{x} \circ \underline{y} 
= \overline{x} \bullet \underline{y} = \overline{x} \circ \underline{y} = p$.

Suppose that $g$ satisfies (3).
Again the $\varDelta'$-connectivity can be shown in a similar way 
as in the proof of Theorem~\ref{thm:building} (3) $\Rightarrow$ (1).
Take any vertex $p = (p_1,p_2,\ldots, p_n)$.
We show that $g^*$ is submodular on 
${\cal I}^*_p(\varGamma) = {\cal I}^*_{p_1}(T_1) \times  {\cal I}^*_{p_2}(T_2) \times \cdots \times {\cal I}^*_{p_n}(T_n)$.
Take any  $x,y \in {\cal I}^*_p$.
Suppose that 
$x_i = [\underline{x_i}, \overline{x_i}]$ and $y_i = [\underline{y_i}, \overline{y_i}]$ for $i=1,2,\ldots,n$.
Let $\overline{x} = (\overline{x_1}, \overline{x_2},\ldots,\overline{x_n})$,
$\underline{x} = (\underline{x_1}, \underline{x_2},\ldots,\underline{x_n})$,
$\overline{y} = (\overline{y_1}, \overline{y_2},\ldots,\overline{y_n})$, and
$\underline{y} = (\underline{y_1}, \underline{y_2},\ldots,\underline{y_n})$,
Then we have
\begin{eqnarray}
&&2(g^{*}(x) + g^*(y)) = g(\overline{x}) + g(\underline{x}) + g(\overline{y}) + g(\underline{y}) \nonumber \\
&& \geq g(\overline{x} \circ \overline{y}) 
+ g(\overline{x} \bullet \overline{y}) 
+ g(\underline{x} \circ \underline{y}) 
+ g(\underline{x} \bullet \underline{y}) \nonumber \\
&& \geq g( \overline{x} \circ \overline{y}  \bullet \underline{x} \bullet \underline{y}) 
+ g(\overline{x} \circ \overline{y}  \circ  \underline{x}  \bullet \underline{y} )
+ g(\overline{x} \bullet \overline{y} 
\bullet \underline{x} \circ \underline{y} )
+ g(\overline{x} \bullet \overline{y}  \circ \underline{x} \circ \underline{y}  ) \nonumber \\
&& = 2(g^{*}(x \sqcup y) + g^*(x \wedge y)), \nonumber
\end{eqnarray}
where we apply (\ref{eqn:circ-bullet}) to the last equality (componentwise).
Thus $g^*$ is submodular on ${\cal I}_p^*$, and $g$ is L-convex on $\varGamma$.

\section*{Acknowledgments}
The author thanks Kazuo Murota for helpful comments improving the presentation, and
Yuni Iwamasa for careful reading.
The work was partially supported by JSPS KAKENHI Grant Numbers 25280004, 26330023, 26280004.

\end{document}